\documentclass[a4paper,11pt]{article}
\usepackage[latin1]{inputenc}
\usepackage[T1]{fontenc}
\usepackage[british,american]{babel}
\usepackage{amsmath}
\usepackage{amsfonts}
\usepackage{mathtools}
\usepackage{hyperref}
\usepackage[T1]{fontenc}
\usepackage[latin1]{inputenc}
\usepackage{ntheorem}
\usepackage{theorem}
\usepackage{enumitem}
\usepackage{makeidx}
\usepackage{fancyhdr}
\usepackage{hyperref}
\usepackage{verbatim}
\usepackage{bbm}
\usepackage{bm}
\usepackage{amssymb}
\usepackage{dsfont}
\usepackage{color}
\usepackage[table]{xcolor}
\usepackage{makecell}
\usepackage{multirow}
\usepackage{arydshln}
\usepackage{booktabs}
\usepackage{relsize}
\usepackage{tikz}
\usepackage{tcolorbox}
\usepackage[permil]{overpic}
\usepackage{graphicx}
\usepackage{subcaption}
\numberwithin{equation}{section}
\usepackage[font=small,labelfont=bf,justification=centering]{caption}
\usepackage{amssymb}
\usepackage{mathtools}

\newcommand{\be}{\begin{equation}}
\newcommand{\ee}{\end{equation}}
\newcommand{\benn}{\begin{equation*}}
\newcommand{\eenn}{\end{equation*}}
\newcommand{\bea}{\begin{eqnarray}}
\newcommand{\eea}{\end{eqnarray}}

\newcommand{\beann}{\begin{eqnarray*}}
\newcommand{\eeann}{\end{eqnarray*}}

\newtheorem{theorem}{Theorem}[section]
\newtheorem{proposition}[theorem]{Proposition}
\newtheorem{corollary}[theorem]{Corollary}
\newtheorem{lemma}[theorem]{Lemma}

\newtheorem{definition}[theorem]{Definition}
\newtheorem{refproof}[theorem]{Proof}
\theorembodyfont{\rm}

\newtheorem{remark}[theorem]{Remark}

\newcommand{\qed}{\hfill $\Box$\smallskip}
\newcommand{\E}{\noindent{$\mathbb{E}$ \ }}

\usepackage[
backend=biber,
bibstyle=numeric,
giveninits=true,
sorting=anyt
]{biblatex}

\definecolor{LightGrey}{RGB}{235,235,235}

\graphicspath{{Fig/}}

\addto\captionsenglish{}

\makeatletter
\newcommand{\thickhline}{%
    \noalign {\ifnum 0=`}\fi \hrule height 2pt
    \futurelet \reserved@a \@xhline
}
\makeatother

\makeatletter
\newcommand{\doubleline}{%
    \noalign{\ifnum0=`}\fi
    \hrule height 0.8pt
    \vskip 1.2 pt
    \hrule height 0.8pt
    \futurelet\reserved@a\@xhline
}
\makeatother

\def\R{\mathbb{R}}
\def\C{\mathbb{C}}
\def\N{\mathbb{N}}

\def\P{\mathbb{P}}
\def\E{\mathbb{E}}

\def\P{\mathbb{P}}

\def\cF{\mathcal{F}}

\def\cO{\mathcal{O}}
\def\cP{\mathcal{P}}

\def\cW{\mathcal{W}}

\def\txtd{{\textnormal{d}}}

\def\Re{{\textnormal{Re}}}
\def\Im{{\textnormal{Im}}}

\def\I{\infty}

\def\I{\infty}
\def\eps{\varepsilon}

\newcommand{\joinR}{\hspace{-.17em}}
\newcommand{\RomanI}{I}
\newcommand{\RomanII}{\RomanI\joinR\RomanI}
\newcommand{\RomanIII}{\RomanI\joinR\RomanII}
\newcommand{\RomanIV}{\RomanI\joinR\RomanV}
\newcommand{\RomanV}{V}

\makeatletter
\newcommand{\assumptionitem}[2]{%
  \item[(#1)]%
  \phantomsection
  \def\@currentlabelname{(#1)}%
  \label{#2}%
}
\makeatother

\allowdisplaybreaks

\newcommand*\samethanks[1][\value{footnote}]{\footnotemark[#1]}
\title{Impact of spinning on the early-warning signs in non-Markovian stochastic systems}
\author{Paolo Bernuzzi\thanks{University of Konstanz, Department of Mathematics and Statistics,  Universit\"atsstra\ss{}e~10 78464 Konstanz, Germany. E-Mail: paolo.bernuzzi@uni-konstanz.de, alexandra.blessing@uni-konstanz.de, dennis.rudik@uni-konstanz.de},
~~~Alexandra Blessing Neam\c tu\samethanks ~~~and~~Dennis Rudik\samethanks
}

\begin{document}

\maketitle

\begin{abstract}
\noindent We construct early-warning signals for impending critical transitions in non-Markovian systems. We analyze stochastic forcings such as fractional Brownian motion, fractional Ornstein-Uhlenbeck processes and red noise in fast-slow systems exhibiting such transitions. We show that the effectiveness of indicators such as autocovariance, autocorrelation, and spectral density depends on several properties of the underlying system. In particular, we compare the influence of the Hurst index and the bifurcation type. We prove that the rotatory dynamics associated with a Hopf bifurcation substantially alters the scaling laws of these observables. Finally, we provide practical guidelines for implementing these signals and validate them on both theoretical and applied models.
\end{abstract}

\section{Introduction}

Many complex systems exhibit abrupt transitions between qualitatively different dynamical regimes.
These events typically happen rather fast after a long period of slow evolution. Although predicting such critical points before they are reached is extremely difficult, M. Scheffer et.al. \cite{Scheffer} introduced generic early-warning signs (EWSs), which might indicate the approach of a critical threshold based on an increasing variance and/or autocorrelation of the corresponding dynamical system.\\

Recently, lots of attention has been devoted to the development of EWSs for critical transitions.~One of the most prominent mechanisms is critical slowing down, whereby the recovery rate from perturbations decreases as stability is lost near a bifurcation.~This phenomenon manifests statistically through increasing variance, rising autocorrelation, spectral reddening, and enhanced sensitivity to stochastic fluctuations.~Such indicators have been observed in both theoretical models \cite{bernuzzi2026early,bernuzzi2024warning,bernuzzi2025early,bernuzzi2026critical,ditlevsen2010tipping,kuehn2022warning} and empirical data from diverse applications, such as climate science \cite{Lucarini1,morr2024detection,Morr_Red_Noise}, physics, e.g.~the behavior of the two-dimensional Ising model near the critical temperature \cite{Ising}, ecology \cite{Ecology} or social sciences \cite{Soc}. \\

The main goal of this work is to construct and analyze different EWSs detecting transitions / bifurcations in non-Markovian stochastic systems as well as to quantify the influence of the memory on EWSs. In this setting, a central tool to study EWSs is given by fast-slow systems. These systems are characterized by the interaction of variables evolving on distinct time scales and are of the form
\[
\dot{x}=f(x,y,\varepsilon) + \sigma \dot{\xi}, \qquad
\dot{y}=\varepsilon g(x,y,\varepsilon),
\]
where $(0<\varepsilon\ll1)$ denotes the ratio of time scales between the fast variable \(x\) and the slow variable \(y\), $f$ and $g$ are nonlinear terms and $\xi$ stands for the noise and $\sigma$ denotes its intensity. Such a time scale separation is common in many applications e.g. in climate sciences \cite{CessiBoxModel,Morr_Red_Noise} or neurosciences \cite{BGN}.~In this framework, critical transitions occur when slow parameter drift drives trajectories toward a bifurcation point of the fast subsystem, resulting in rapid qualitative changes in the full dynamics. For more details on slow-fast systems we refer to \cite{kuehn2015multiple,kuehn2019scaling}.\\

While deterministic slow-fast systems provide valuable insight into the underlying mechanisms, they often fail to capture the variability, irregularity, and spontaneous transitions observed in applications. These phenomena are frequently driven by stochastic fluctuations originating from unresolved degrees of freedom, environmental variability, or intrinsic microscopic processes.~To account for these effects, noise is commonly incorporated into the slow-fast systems.~The stochastic forcing is often assumed to be Markovian, typically represented by white noise with no temporal correlations.~This assumption may not accurately reflect the properties of fluctuations encountered in applications.~Physical, biological, and environmental processes often possess memory, persistence, or delayed interactions, resulting in noise with finite correlation times and nontrivial temporal structure. For this reason, non-Markovian noise has emerged as an important extension of traditional stochastic modeling frameworks \cite{Nils,eichinger2020sample,kuehn2022warning,Morr_Red_Noise}. By incorporating memory effects, non-Markovian fluctuations provide a more faithful description of systems in which past states influence future stochastic behavior.~Such a noise forcing can influence the transition rates, synchronization properties, oscillatory dynamics, and stability characteristics, particularly in fast-slow systems where the interaction between intrinsic time scales and noise correlation times plays a crucial role. Consequently, understanding the influence of non-Markovian noise is essential for accurately describing and predicting the behavior of complex dynamical systems. In this work, we contribute to this aspect and quantify the impact of non-Markovian noise for a dynamical system approaching a critical transition.~Below, we briefly describe the types of bifurcations as well as the non-Markovian fluctuations investigated in this work. 

\paragraph{Types of bifurcations and their applications.}

Sudden real-life phenomena are often associated to tipping points that indicate a threshold upon which abrupt changes in an environment occur.~Although such events are common in many applications, a prime example is climate science. In fact, multiple tipping elements, crucial to the maintenance of the current natural global status, have displayed destabilization within the last century.~Among them are the Greenland Ice Sheet loss, the decline of the Amazon rainforest and the collapse of the Atlantic Meridional Overturning Circulation (AMOC) \cite{boers2025destabilization}. Such abrupt dynamics can be attributed in energy balance models (EBMs) to critical transitions that are driven by the crossing of a slow parameter of a bifurcation threshold.~For such a reason, we observe EWSs to the approach of codimension-$1$ bifurcations, in which the current tracked state loses stability \cite{kuehn2015multiple}. These are listed as:
\begin{itemize}
    \item fold bifurcations, that indicate the merge and collapse of a stable and unstable state;
    \item transcritical bifurcations, which occur in the case a stable and unstable equilibrium cross and exchange stabilities;
    \item pitchfork bifurcations, for which either a stable state splits into a stable and two unstable states (supercritical pitchfork) or an unstable state separates into an unstable and two stable states (subcritical pitchfork);
    \item Hopf bifurcations, associated to the loss of stability of a state and the onset of a stable limit cycle (supercritical Hopf) or to an equivalent event with inverted stabilities (subcritical Hopf).
\end{itemize}
A clear distinction between such bifurcations is the fact that only the latter displays rotatory dynamics.~Consequently, the others can be collected as codimension-$1$ bifurcations in one-dimensional models, upon an appropriate change of variable to the corresponding normal forms \cite{kuehn2015multiple}. In particular, ocean EBMs that observe the status of AMOC are known to display multiple bifurcation types \cite{bailie2024bifurcation,bernuzzi2024warning,ChenPSDforHopf}.~A specific system, the Stommel-Cessi model \cite{CessiBoxModel}, is studied in Section \ref{sec:applications} and used in the cross-validation of the EWSs developed in the article. Furthermore, we investigate the rise of limit cycles in a theoretical model inherent to multiple applications, such as extended harmonic oscillators \cite{Oded}.

\paragraph{Non-Markovian noise.}

Fractional Brownian motion (fBm) is a famous example of stochastic process used in 
order to model memory effects or long-range dependencies. An fBm is a centered, stationary 
Gaussian process parameterized by a so-called Hurst index/parameter $H\in(0,1)$. For 
$H = 1/2$, one recovers the classical Brownian motion. However, for $H \in (1/2, 1)$ 
and $H \in (0, 1/2)$, fBm exhibits a different behaviour than Brownian motion. 
Its increments are no longer independent, but positively correlated for $H > 1/2$, 
and negatively correlated for $H < 1/2$. Fractional Brownian motion has been used to model 
a wide range of phenomena, extending from mathematical finance to fluid 
dynamics or fractional transport models \cite{Oana,Flandoli}. The dynamics of slow-fast systems subject to non-Markovian noise recently captured lots of attention, where several aspects have been investigated, such as sample path properties \cite{Nils,eichinger2020sample}, EWSs \cite{kuehn2022warning} or averaging dynamics \cite{HL2,HL1,HS}. The estimating of the Hurst index from observations was treated in \cite{Oana,Tudor}.
\medskip




Going beyond the setting of fractional Brownian motion, we also analyze the effects of several other non-Markovian fluctuations such as (fractional) Ornstein-Uhlenbeck processes or red noise investigated in  \cite{Morr_Red_Noise} whose structure is motivated by the Mori-Zwanzig formalism.  As seen in~\cite{kuehn2022warning}, the variance fails to be an EWS for an SDE driven by an Ornstein-Uhlenbeck process. In this work, we show how statistics based on the spectral density of the underlying system can be used as EWSs, providing a clear explanation for the color blindness observed in \cite{kuehn2022warning}. 

Finally, we mention that our reults can be extended to  Volterra processes as considered in \cite{kuehn2022warning}. These can be represented as an integral of a (memory) kernel with respect to the Brownian motion and include e.g.~the Liouville fractional Brownian motion and the Rosenblatt which models non-Gausssian fluctuations. 

\paragraph{Main results and structure of the paper.}

This work provides an in-depth analysis of EWSs for stochastic differential equations driven by fractional Brownian motion (fBm) across the full range of Hurst parameters $H\in(0,1)$. We adopt the standard setting for the study of critical transitions based on fast-slow systems, in which a slowly varying variable acts as a control parameter that gradually approaches a bifurcation threshold. To characterize the local behavior near the impending transition, we consider the associated linearized fast subsystem along an attracting branch of steady states. Within this setting, EWSs arise from the scaling properties of statistical observables as the distance to criticality decreases. Compared to existing results \cite{kuehn2022warning,Morr_Red_Noise}, which are typically restricted to $H\in\left(1/2,1\right)$, we extend the scaling theory of EWSs to the regime of both persistent and anti-persistent noise. In particular, we establish explicit asymptotic scaling laws for the time-asymptotic autocovariance \cite{kuehn2022warning}, autocorrelation \cite{bernuzzi2024warning}, and spectral density \cite{ChenPSDforHopf} in the vicinity of critical transitions, highlighting how memory effects depend on both the Hurst index and the underlying bifurcation structure.~Moreover, we verify the applicability of such EWSs to multiple types of stochastic forcings, such as red non-Markovian noise \cite{Morr_Red_Noise} and fractional Ornestein-Uhlenbeck perturbations (fOU) \cite{kuehn2022warning}. 
\begin{table}[h]
    \centering
    \begin{tabular}{|c|c|c|c|c|}
    \hline
         & \multirow{2}{*}{$V$} & \multirow{2}{*}{$AC(\tau)$} & \multicolumn{2}{c|}{\parbox[c][.8cm][c]{1.8cm}{$S$ at peak}} \\
         \cline{4-5}
         & & & \parbox[c][.8cm][c]{1.8cm}{\centering $H\in\left(0,\frac{1}{2}\right]$} & \parbox[c][.8cm][c]{1.8cm}{$H\in\left(\frac{1}{2},1\right)$} \\
         \hline
         fBm & Div $(-2H)$ & Conv $(2H)$ & Div $(-1-2H)$ & Inf \\
         \hline
         Red noise & Div $(-2H)$ & Conv $(\text{min}\{1,2H\})$ & Div $(-1-2H)$ & Inf \\
         \hline
         fOU  & Conv $(2-2H)$  & Conv $(2-2H)$ & Conv $(1-2H)$ & Inf \\
    \hline
    \end{tabular}
    \caption{Scaling laws of statistical observables in the proximity to a codimension-$1$ one-dimensional bifurcation. The cells indicate the diverging (Div), convergent (Conv) or infinite (Inf) behavior of the indicators for different types of noise terms. We observe that the hyperbolic exponent associated to the rates, in parentheses, depend explicitly on the Hurst index $H$. Moreover, while the spectral density is not bounded for $H\in\left(\frac{1}{2},1\right)$ prior to the critical transition, we show that it can be used as an EWS as we study its growth on specific frequencies.}
    \label{tab:fold overview}
\end{table}

A second main contribution concerns the role of rotational dynamics in two-dimensional systems undergoing Hopf bifurcations \cite{kuehn2015multiple}. We show that the onset of oscillatory motion introduces a mixing mechanism that fundamentally alters the propagation of memory in the system.~In contrast to codimension-$1$ bifurcations, where long-range dependence is directly inherited from the driving fractional Brownian motion, the rotational component acts as a stabilizing mechanism as it can suppress the effective persistence of correlations, leading to a form of memory loss induced by phase mixing. Moreover, such a phenomenon may result in the unmasking of spectral components that are otherwise hidden in non-rotational regimes \cite{kuehn2022warning}. These effects are rigorously quantified through the scaling behavior of the associated EWSs, revealing a qualitative distinction between dissipative memory amplification in one-dimensional systems and rotation-driven mixing in Hopf normal forms. In particular, the resulting scaling laws depend on the interplay between the Hurst parameter and the rotational structure of the linearized dynamics, leading to different asymptotic regimes.

\begin{table}[h]
    \centering
    \begin{tabular}{|c|c|c|c|}
    \hline
         & $V$ & $AC(\tau)$ & $S$ at peaks \\
         \hline
         fBm &  &  &  \\
         \cline{1-1}
         Red noise  & Div $(-1)$  & Conv $(1)$ &  Div $(-2)$\\
         \cline{1-1}
         fOU &  &  & \\
    \hline
    \end{tabular}
    \caption{Scaling laws of statistical observables in the proximity to a Hopf bifurcation. The notation is equivalent to Table \ref{tab:fold overview}. The spinning drive in the proximity to the critical threshold induces an equivalent rate for all types of noise addressed, which is not qualitatively dependent on $H$. The loss of stability on more directions introduces however further complexity in the scaling laws of the time-asymptotic autocorrelation, which is affected by the change of the rotatory dynamics, such a shape and velocity. Furthermore, we note that the spectral density can still be unbounded for large $H$. Nonetheless, the structure of the drift dynamics induces the rise of two further peaks that diverge in the critical limit.}
    \label{tab:Hopf overview}
\end{table}

To clarify the distinction, we summarize the main results in two separate tables: Table \ref{tab:fold overview} collects the scaling laws for one-dimensional codimension-$1$ bifurcations (including fold, transcritical, and pitchfork types), where all memory effects are directly governed by the Hurst parameter and the proximity to criticality; Table \ref{tab:Hopf overview} focuses on Hopf bifurcations, where rotational dynamics introduce additional mixing terms that modify or suppress long-range dependence signatures and reshape the structure of the spectral density. Both tables report the scaling laws of the asymptotic behavior of the time-asymptotic autocovariance $V$, time-asymptotic autocorrelation $AC$, and spectral density $S$ as hyperbolic exponents in function of the distance to the bifurcation threshold and the Hurst index, allowing for a direct comparison between non-rotational and rotational regimes. For both cases, we construct and discuss EWSs based on the analytic rates adopted by the observables for any $H\in(0,1)$. Specifically, we observe the increase of the indicators along specific modes or frequencies in the critical limit. The results show that the spinning effect has a stabilizing influence across all noise types considered in this paper and for all EWSs.

\begin{center}
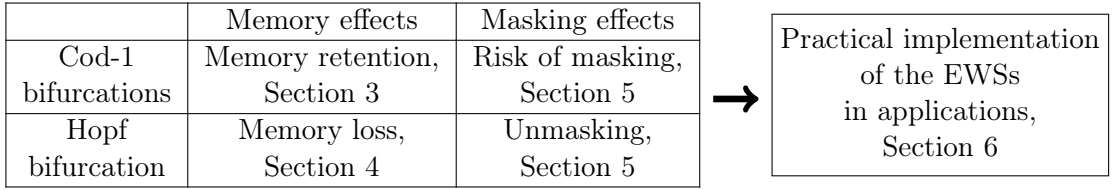

    \begin{tabular}{|c|c|c|}
        \hline
        & Memory effects & Masking effects \\
        \hline
        $\begin{matrix}
            \textnormal{Cod-1}\\
            \textnormal{bifurcations}
        \end{matrix}$ 
        & $\begin{matrix}
            \textnormal{Memory retention,} \\ \textnormal{Section \ref{sec:fbm 1d}}
        \end{matrix}$ 
        & $\begin{matrix}
            \textnormal{Risk of masking,} \\ \textnormal{Section \ref{sec:red noise}}
        \end{matrix}$ \\
        \hline
        $\begin{matrix}
            \textnormal{Hopf}\\
            \textnormal{bifurcation}
        \end{matrix}$ 
        & $\begin{matrix}
            \textnormal{Memory loss,} \\ \textnormal{Section \ref{sec:fbm 2d}}
        \end{matrix}$ 
        & $\begin{matrix}
            \textnormal{Unmasking,} \\ \textnormal{Section \ref{sec:red noise}}
        \end{matrix}$ \\
        \hline
    \end{tabular}
    \begin{tikzpicture}
        \draw[->, line width=2.5pt] (0,0) -- (.6,0);
    \end{tikzpicture}
    \fbox{$\begin{matrix}
        \textnormal{Practical implementation} \\ \textnormal{of the EWSs}  \\ \textnormal{in applications,} \\ \textnormal{Section \ref{sec:applications}}
    \end{matrix}$}
    \captionof{figure}{Structure of the paper and summary of the main results.}
    \label{fig: structure}
\end{center}

The paper is structured as follows.~Section \ref{sec:preliminaries} introduces the probabilistic foundation and dynamical setting, including the relevant statistical observables, fractional Brownian motion, and fast-slow systems with non-Markovian forcing. We also present the approximation methods and spectral notation used throughout the paper.

Section \ref{sec:fbm 1d} develops the analysis for one-dimensional codimension-1 bifurcations, establishing the scaling laws for EWSs and clarifying the role of the Hurst parameter within the rates. This dependence is, in fact, addressed as memory retention. Such results are stated in Theorem \ref{thm:autocov fast system}, Corollary \ref{cor:autocor fast system} and Lemma \ref{lem:fBm 1d PSD} and are proven extending methods from \cite{bernuzzi2024warning,ChenPSDforHopf,Cheridito_fOU,kubilius2017parameter}.

Section \ref{sec:fbm 2d} adapts these results to two-dimensional systems undergoing Hopf bifurcations. We characterize the impact of rotational dynamics on the propagation of memory and on the structure of the spectral density, highlighting qualitative differences with respect to the one-dimensional case. We notice in fact that the scaling laws of the observables are qualitatively stabilized. For instance, the rates of divergence of the time-asymptotic autocovariance, in Theorem \ref{thm:autocov fast system complex}, and the diverging Fourier modes within the spectral density, in Theorem \ref{thm:psd_on_elements}, are not explicitly dependent on $H$. Such an effect can be indicated as memory loss due to rotation of the trajectory. Moreover, the implementation of some EWSs requires also to address the ulterior complexity given by the change of stability in the critical threshold along multiple modes, such as in Corollary \ref{cor:autocor fast system complex}. The proofs within the section include approaches from \cite{bernuzzi2024warning,Pipiras_Taqqu_fBM}.

Section \ref{sec:red noise} considers fast-slow systems with alternative non-Markovian forcings \cite{kuehn2019scaling,Morr_Red_Noise}, emphasizing the interaction between noise coloring effects and bifurcation-induced scaling. We analyze the masking, or color blindness \cite{kuehn2019scaling}, associated to diminished rates of the observables. We observe how rotational dynamics may mitigate or bypass masking phenomena in EWSs. The methodology in this section is based on the previous approaches and the techniques in \cite{Basse_stationary_increments, bernuzzi2026critical, Priestley_Spectral_Analysis}.

In Section \ref{sec:applications}, we describe the implementation of the EWSs on collected data. We provide guidelines for interpreting the observables, together with the caveats arising from analytical approximations and the applicability highlighted by the study of the indicators in the previous sections. Furthermore, we use ergodic properties of the trajectories to choose statistical estimators for the numerical implementation of the results \cite{kubilius2017parameter}. Finally, we cross-validate our results on an applied ocean model \cite{CessiBoxModel} and on a theoretical system, both driven by the stochastic forcings addressed in the paper.

For a better visualization of our results, we provide an overview in Figure \ref{fig: structure}.

\paragraph{Acknowledgments.} The authors acknowledge support from DFG CRC 1432 " Fluctuations and Nonlinearities in Classical and Quantum Matter beyond Equilibrium" - Project ID 425217212. A. Blessing and D. Rudik have been supported by DFG Project ID 543163250.


\section{Preliminaries and notations} \label{sec:preliminaries}

To facilitate the exposition, we collect in this section the notation and preliminary results underlying our analysis. The discussion is organized into several subsections addressing the different tools and concepts employed throughout the paper.

\subsection{Statistical observables for stochastic processes}
For stochastic processes there are many different statistics that can be considered. One of the most important is the autocovariance function and directly related to it, the autocorrelation function.~In particular, it is well-known that Gaussian processes are completely characterized by their autocovariance.~In the following, we let by $(\Omega,\cF,\P)$ stand for a probability space and denote by $\E$ the expectation and by $|\cdot|$ the Euclidean norm. Moreover, we write the variables in bold $(\mathbf{x})$ if they are vectors and if they are scalar we write them as $(x)$.

\begin{definition} \label{def:var and autocor}
Let $(\mathbf{x}_t)_{t\geq 0}$ be a $n$-dimensional stochastic process. Then the autocovariance at time $t\ge0$ 
and lag time $\tau\ge0$ is defined as 
\begin{align}\label{eq:ACov def}
    V_t(\tau):=\E[\mathbf{x}_t\overline{\mathbf{x}}_{t+\tau}^{\text{T}}]-\E[\mathbf{x}_t]\E[\overline{\mathbf{x}}_{t+\tau}^{\text{T}}],
\end{align}
where $\overline{\mathbf{x}}$ is the complex conjugate and $\mathbf{x}^{\text{T}}$ is the transpose.
For $\tau=0$ we recover the variance of the process. For $n=1$ we can also define the autocorrelation function
\begin{align}\label{eq:ACor def}
    AC_t(\tau):=\frac{V_t(\tau)}{V_t(0)}.
\end{align}
\end{definition}

Note that the autocovariance $V_t(\tau)$ and autocorrelation $AC_t(\tau)$ generally depend on $t$ and $\tau$. If the process $(\mathbf{x}_t)_{t\geq 0}$ is stationary, as defined below, these statistics only depend on $\tau$. 
\begin{definition}
    A stochastic process $(\mathbf{x}_t)_{t\geq 0}$ is called strictly stationary or strongly stationary if for all $k\ge0$ and all sequences $(t_1,\dots, t_m)\in\R^m_+$ we have
    $$(\mathbf{x}_{t_1},\dots,\mathbf{x}_{t_m})\overset{d}{=}(\mathbf{x}_{t_1+k},\dots,\mathbf{x}_{t_m+k}).$$
   The process $(\mathbf{x}_t)_{t\geq 0}$ is called second order stationary or weakly stationary, if we have for all $t\in\R$ $\E[|\mathbf{x}_t|]<\infty,~\E[|\mathbf{x}_t|^2]<\infty$ and for all $s,t,\tau\ge0$
    $$V_t(\tau)=V_s(\tau).$$
    As the autocovariance only depends on $\tau$, we drop the subscript and write $V(\tau)$ whenever we consider a stationary process.
\end{definition}
Clearly, strong stationarity implies weak stationarity if the process has finite first and second moments, but not vice versa. In the case of a weak stationary process $(\mathbf{x}_t)_{t\geq 0}$ we can also consider the spectral density. We first give the standard definition. Afterwards, we give some remarks on the different definitions and also the terms energy spectrum density and power spectrum density that are often used in this context. 
\begin{definition}{\em (Spectral density)}\label{def:PSD}
    Let $(\mathbf{x}_t)_{t\geq 0}$ be a (weakly) stationary $n$-dimensional stochastic process with
    $$\int_\R V(\tau)~\txtd\tau<\infty.$$ Then the spectral density is defined as 
    \begin{align}\label{eq:def spectral density}
        S_\mathbf{x}(\omega):= \mathcal{F}(V(\tau))(\omega)=\frac{1}{2\pi}\int_\R e^{-i\omega \tau} V(\tau)~\txtd \tau,
    \end{align}
    where the Fourier transform $\cF$ is interpreted componentwise.
\end{definition}
There is an alternative definition for the spectral density using only the path itself and not the autocovariance.
For a stationary process $(\mathbf{x}_t)_{t\geq 0}$ we do not have a Fourier integral representation. This means that for 
$$G_T(\omega)=\frac{1}{\sqrt{2\pi}}\int_{-T}^Tx_t e^{-i\omega t}~\txtd t,$$
we do not have a well-defined limit $\lim\limits_{T\to\infty}G_T(\omega)$.
Intuitively this is clear, as the stationarity implies that the distribution is the same for all times and hence the tails of $\mathbf{x}_t$ for $t\to\pm\infty$ cannot converge to zero. However, we notice that
$$\frac{|G_T(\omega)|^2}{2T}$$ 
has a well-defined limit for $T\to\infty$. Finally, we get an alternative definition for the (power) spectral density function by taking the expected value.
\begin{definition}{\em((Power) spectral density)}\label{def:PSD2}
  Let $(\mathbf{x}_t)_{t\geq 0}$  be a stationary process. Then 
  \begin{align}\label{eq:PSD via Fourier}
    S_\mathbf{x}(\omega)=\lim_{T\to\infty} \E\left[\frac{|G_T(\omega)|^2}{2T}\right]
\end{align}
is called (power) spectral density. 
\end{definition}
The Wiener-Khintchine Theorem relates~\eqref{eq:def spectral density} and~\eqref{eq:PSD via Fourier} as follows.
\begin{theorem}{\em (Wiener-Khintchine)}\label{wk}
    Let $(\mathbf{x}_t)_{t\geq 0}$ be a (weakly) stationary process. Then,
    $AC(\tau)$ is a autocorrelation function of some (weakly) stationary process $\mathbf{x}_t$ if and only if there exists a function $F(\omega)$ having the properties of a distribution function on $(-\infty,\infty)$ such that for all $\tau\in\R$
\begin{align*}
    AC(\tau)=\int_\R e^{i\omega\tau}~\txtd F(\omega),
\end{align*}
where $F(\omega)$ is called the spectral distribution of $(\mathbf{x}_t)_{t\geq 0}$. Moreover, if 
$$\int_\R AC(\tau)~\txtd \tau<\infty,$$
we have
\begin{align*}
    F(\omega)=\int_{-\infty}^\omega \frac{S_\mathbf{x}(\theta)}{V(0)}~\txtd\theta.
\end{align*}

\end{theorem}

\begin{remark}
    The definition~\eqref{eq:PSD via Fourier} is used e.g. in~\cite{ChenPSDforHopf} or~\cite{morr2024detection}, where the power spectral density was considered as an EWS for systems with red noise.~We work with~\eqref{eq:PSD via Fourier} referring to it as spectral density, in the light of Theorem~\ref{wk}.  
\end{remark}

\subsection{Fractional Brownian motion}

\begin{definition}
A scalar fractional Brownian motion $\left(W^{H}_t\right)_{t\geq 0}\subset \R$ with Hurst index $H\in(0,1)$ is a centered Gaussian process with covariance 
\begin{equation}
\label{eq:cov_fBm} 
\E \left[W^H_t W^H_s\right]=\frac{1}{2} (t^{2H} +s^{2H} -|t-s|^{2H}) \;.
 \end{equation}
\end{definition}

For $H=\frac12$ we recover the standard Brownian motion in $\R$, whereas for $H\neq \frac12$ we obtain a process which is neither Markov nor a semi-martingale.~We list some important properties of fBm.~For more details, we refer to~\cite{bookfbm}.

\begin{proposition} Let $(W^H_t)_{t\geq 0}$ be a scalar fractional Brownian motion with Hurst index in $H(0,1)$. Then:
\begin{description}
    \item[{\bf(Long-range dependence)}] If $H>1/2$ then 
    \[ \sum\limits_{n=1}^\infty \E [W^H_1(W^{H}_{n+1}-W^H_n)]=\infty. \]
    \item[{\bf(Short-range dependence)}] If $H<1/2$ then
    \[ \sum\limits_{n=1}^\infty \E [W^H_1(W^{H}_{n+1}-W^H_n)]<\infty. \]
 \item [(Self-similarity)] Let $a>0$. Then 
 \[ (a^H W^H_t)_{t\geq 0} \stackrel{\text{law}}{=}(W^{H}_{at})_{t\geq 0}. \]
 \item [(Time-inversion)] It holds
 \[ (t^{2H}W^H_{1/t})_{t>0}\stackrel{\text{law}}{=}(W^H_t)_{t> 0}.  \]
\item [(Stationarity of the increments)] For all $h>0$ we have 
\[ (W^{H}_{t+h} -W^H_t )_{t\geq 0} \stackrel{\text{law}}{=} W^H_h. \]
\item [(Regularity of the increments)] The scalar fractional Brownian motion has a version which is $\alpha$-H\"older continuous with $\alpha<H$. 
\end{description}    
\end{proposition}

Moreover, we extend the definition to the $n$-dimensional setting describing fractional Brownian motion processes within $\R^n$.

\begin{definition}
A fractional Brownian motion (fBm) $\left(\mathbf{W}^{H}_t\right)_{t\geq 0}\subset \R^n$ with Hurst index $H\in(0,1)$ is a centered Gaussian process with covariance 
\begin{equation}
\label{eq:cov_fBm n-dim} 
\left(\E \left[\mathbf{W}^H_t \left(\mathbf{W}^H_s\right)^\textnormal{T}\right]\right)_{j_1,j_2\in\{1,\dots,n\}}=\frac{1}{2} (t^{2H} +s^{2H} -|t-s|^{2H}) \delta_{j_1,j_2},
 \end{equation}
for $\delta$ that indicates the Kronecker delta. Specifically, the fBm in $\R^n$ is composed by i.i.d. scalar fBms with Hurst index $H$.
\end{definition}

Since we only consider stochastic differential equations driven by additive noise, we use the Wiener-integral of a deterministic time-dependent function with respect to fBm~\cite{Biagini,duncan2002fractional}.~We also refer to Appendix~\ref{app:A} for the properties of this integral, in particular for the well-posedness of SDEs with additive fractional noise.\\

Moreover, for the computation of spectral densities we use the following results derived in \cite{Pipiras_Taqqu_fBM}.~To this aim we introduce the linear space
$$\tilde\Lambda^H:=\left\{f: f\in L^2(\R), \int_\R |\mathcal{F}(f)(x)|^2|x|^{1-2H}~\txtd x<\infty \right\},$$
with the inner product
$$\langle f,g\rangle_{\tilde\Lambda^H}:=\frac{\Gamma(2H+1)\sin(\pi H)}{2\pi}\int_\R \mathcal{F}(f)(x)\overline{\mathcal{F}(g)(x)}|x|^{1-2H}~\txtd x,$$ where $\cF$ denotes the Fourier-transform and $\overline{\mathbf{w}}$ refers to the complex conjugate of the complex vector $\mathbf{w}$.~In particular, this allows one to compute the covariance of the integral process in terms of the Fourier-transform as
\begin{align}\label{eq:Cov Pipiras Taqqu}
    \text{Cov}\left(\int_\R f(t)~\txtd W_t^H ,\int_\R g(t)~\txtd W_t^H\right) = \frac{\Gamma(2H+1)\sin(\pi H)}{2\pi}\int_\R \mathcal{F}(f)(x)\overline{\mathcal{F}(g)(x)}|x|^{1-2H}~\txtd x.
\end{align}

\subsection{Fast-slow systems and tipping} \label{sec: FS}

We consider stochastic fast-slow systems of the form
\begin{align}\label{eq:fast system}
    \begin{cases}
        \txtd \mathbf{x}_t = f(\mathbf{x}_t,y_t,\eps)~\txtd t + \Sigma~\txtd \mathbf{W}_t^H, \\
        \txtd y_t = \eps g(\mathbf{x}_t,y_t,\eps)~\txtd t,
    \end{cases}
\end{align}
for $t\geq 0$. The fast component $(\mathbf{x}_t)_{t\geq0}\subset \R^n$ is perturbed by a noise term described by a fractional Brownian motion process $(\mathbf{W}_t^H)_{t\geq 0}\subset \R^n$ with Hurst index $H\in(0,1)$. Such a forcing is filtered by a noise intensity matrix $\Sigma\in\R^{n\times n}$ which has non-negative real spectrum. Along with the slow component $(y_t)_{t\geq0}\subset \R$, its trajectory is described also by the smooth maps $f,g:\R^n\times\R\times\R\to\R$. The discrepancy in the velocity associated with the components is attributed to the parameter $0<\eps\ll1$. The system \eqref{eq:fast system} is described in fast time, i.e. it is observed on the time scale of the motion of $\mathbf{x}$ in which $y$ is considered as slow due to the explicit term $\eps$ in its time derivative. Conversely, in the system translated to slow time, meaning that we consider the natural time of the slow component $s=\eps t$, the component $\mathbf{x}$ appears to be fast in comparison to $y$. In this form, the model is
\begin{align}\label{eq:slow system}
    \begin{cases}
        \eps\txtd \mathbf{x}_s = f(\mathbf{x}_s,y_s,\eps)~\txtd s + \eps^{1-H}\Sigma~\txtd \mathbf{W}_s^H,\\
        \txtd y_s = g(\mathbf{x}_s,y_s,\eps)~\txtd s,
    \end{cases}
\end{align}
for $s\geq0$. In both settings, we can consider the limit case $\eps\to0$ and study the corresponding subsystems. In the fast time perspective, \eqref{eq:fast system}, the limit provides the fast subsystem, in which the slow component is constant. Intuitively, the fast component is much faster compared to the others, making it so that $y$ does not change effectively in the corresponding time scale. In the slow regime, the slow subsystem is the differential-algebraic model
\begin{align*}
    \begin{cases}
        0 = f(\mathbf{x}_s,y_s,0) \\
        \txtd y_s = g(\mathbf{x}_s,y_s,0)~\txtd s.
    \end{cases}
\end{align*}
The set $C_0=\{(\mathbf{x},y):f(\mathbf{x},y,0)=0\}$ is called critical manifold. While such a manifold constrains the trajectories of the slow subsystem, it also indicates the steady states in the deterministic fast subsystem, i.e. for $\Sigma$ being the null matrix. A subset $S_0\subset C_0$ is called normally hyperbolic if $\partial_\mathbf{x} f(p_1,p_2,0)\neq0$ for all $(p_1,p_2)\in S_0$. The $S_0$ is attracting if for all $(p_1,p_2)\in S_0$ we have $\partial_\mathbf{x} f(p_1,p_2,0)<0$. The attracting part of the center manifold $C_0^a$ is given by a graph $\{\mathbf{x}=h_0^a(y)\}$. Given the structure of $C_0$, bifurcations in $y$ indicate thresholds in which the dynamics change drastically \cite{guckenheimer2013nonlinear}. Such an impactful change is called tipping phenomenon and is usually associated to negative irreversible transitions. Their nature justifies the importance of their prediction by highlighting and recognizing preceding events. The loss of local stability in $C_0^a$ as $y$ approaches the bifurcation threshold is referred to as critical slowing down. As such, the stochastic term is less dwindled by the deterministic attracting forcing by the drift. Observables capable of addressing similar changes of regime to predict tipping in \eqref{eq:fast system} are called early-warning signs (EWSs) \cite{bernuzzi2026early,ditlevsen2010tipping}. In this paper, we focus on two families of EWSs that capture critical slowing down through the observation of the indicators introduced in Definition \ref{def:var and autocor} and Definition \ref{def:PSD2} within close proximity to the bifurcation threshold.\\
Prior to the crossing of a bifurcation threshold $\lambda^*$ and under the assumption of minor stochastic forcing, trajectories of \eqref{eq:fast system} track $C_0^a$ for long times \cite{kuehn2015multiple}. In such a framework, we refer to the critical parameter $y=\lambda$ and linearize the system along a steady solution in $C_0^a$, obtaining
\begin{align}\label{eq:fast system simplified}
    \txtd \mathbf{x}_t = M(\lambda) \mathbf{x}_t~\txtd t + \Sigma~\txtd \mathbf{W}_t^H,
\end{align}
where $M(\lambda):=\partial_\mathbf{x} f(h_0^a(\lambda),\lambda,0)\in\R^{n\times n}$ has all eigenvalues with real negative part for $\lambda<\lambda^*$. By definition of the codimension-1 bifurcations \cite{guckenheimer2013nonlinear} at $\lambda=\lambda^*$, at least one eigenvalue of $M(\lambda)$ crosses the imaginary axis.\\
In the slow regime, \eqref{eq:slow system} we set $g\equiv1$ for simplicity. This choice does not indicate a loss of generality under the assumption of a linear approach to the threshold $\lambda^*$. Hence, we can identify $y$ with the slow time $s$. Next, we linearize along $(h_0^a(s),s)$ and use the self-similarity of the fBm to get the equation
\begin{align}\label{eq:slow system simplified}
    \eps \txtd \mathbf{x}_s = \tilde{M}(s) \mathbf{x}_s ~\txtd s + \eps^{1-H}\Sigma~\txtd \mathbf{W}_s^H,
\end{align}
where $\tilde{M}(s):=\partial_\mathbf{x} f(h_0^a(s),s,0)$. 

\subsection{Codimension-1 bifurcations and spectral perspective} \label{sec: S1 and S2}

Since $\lambda^*$ is a codimension-1 bifurcation threshold, we address systems \eqref{eq:fast system simplified} with drift matrix $M(\lambda)$ that satisfies spectral properties in accordance with the normal form of the corresponding bifurcation \cite{kuehn2015multiple}. As the eigenvectors of $M(\lambda)$ can have complex elements, we introduce the standard scalar product $\left\langle \mathbf{w}_1, \mathbf{w}_2 \right\rangle_{m}:= \mathbf{w}_1^{\text{T}} \overline{\mathbf{w}_2}$ and the norm $\left|\left| \mathbf{w}_1 \right|\right|_m^2:=\left\langle \mathbf{w}_1, \mathbf{w}_1 \right\rangle_m$ for any $\mathbf{w}_1, \mathbf{w}_2 \in \C^m$ and $m\in \N_{>0}$. Moreover, to simplify the notation, we write $\left\langle \mathbf{w}_1, \mathbf{w}_2 \right\rangle =\left\langle \mathbf{w}_1, \mathbf{w}_2 \right\rangle_{m}$ and $\left|\left| \mathbf{w}_1 \right|\right|^2:=\left\langle \mathbf{w}_1, \mathbf{w}_1 \right\rangle$ if $m=n$. We distinguish between two settings:
\begin{description}
    \assumptionitem{S1}{S1} In the case $\lambda$ approaches a fold, a pitchfork or a transcritical bifurcation \cite{kuehn2015multiple}, we consider $n=1$ and $M(\lambda)=A(\lambda)\in\R$. Moreover, we set $\Sigma=\sigma\in\R$. 
    Finally, we label the solution of \eqref{eq:fast system simplified} as $(x_t)_{t\geq0}$ to indicate this scalar framework.
    \assumptionitem{S2}{S2} As $\lambda^*$ indicates a Hopf bifurcation \cite{kuehn2015multiple}, we set $n=2$ and assume the drift matrix $M(\lambda):=\partial_\mathbf{x} f\left(\mathbf{h}_0^a(\lambda),\lambda,0\right)\in\R^{2\times 2}$ to have two complex conjugate eigenvalues $A(\lambda) \pm i B(\lambda)\in\C$.
    We associate to these eigenvalues the following eigenvectors:
\begin{align*}
    \left(A(\lambda)+ i B(\lambda)\right) \mathbf{e}_{1}(\lambda) = M(\lambda) \mathbf{e}_{1}(\lambda), &\quad \left(A(\lambda)- i B(\lambda)\right) \mathbf{e}_{2}(\lambda) = M(\lambda) \mathbf{e}_{2}(\lambda),\\
    \left(A(\lambda)- i B(\lambda)\right) \mathbf{e}_{1}^*(\lambda) = M(\lambda)^\text{T} \mathbf{e}_{1}^*(\lambda) \quad &\text{and} \quad \left(A(\lambda)+ i B(\lambda)\right) \mathbf{e}_{2}^*(\lambda) = M(\lambda)^\text{T} \mathbf{e}_{2}^*(\lambda),
\end{align*}
which are complex and assumed to be continuous in $\lambda\leq \lambda^*$. Consequently, we obtain that $\left\langle \mathbf{e}_{1}(\lambda), \mathbf{e}_{2}^*(\lambda) \right\rangle = \left\langle \mathbf{e}_{2}(\lambda), \mathbf{e}_{1}^*(\lambda) \right\rangle=0$ for any $\lambda\leq\lambda^*$.
\end{description}
The functions $A:(-\infty,\lambda^*]\to\R$ and $B:(-\infty,\lambda^*]\to\R$ are chosen to uphold the bifurcations described and the loss of stability in the followed branch. As such, they are continuous and satisfy $A(\lambda)<0$ for $\lambda<\lambda^*$, $A(\lambda^*)=0$ and $B(\lambda)\neq0$ for any $\lambda\leq\lambda^*$. \\

In the setting described above, we indicate the autocovariance of the solution of \eqref{eq:fast system simplified} along any pair of proxy functions $\mathbf{w}_1, \mathbf{w}_2 \in\C^n$ with
\begin{align*}
    V_t(\tau)[\mathbf{w}_1,\mathbf{w}_2] = \mathbb{E}\left[\left\langle \mathbf{x}_t, \mathbf{w}_1 \right\rangle \overline{\left\langle \mathbf{x}_{t+\tau}, \mathbf{w}_2 \right\rangle} \right] - \mathbb{E}\left[\left\langle \mathbf{x}_t, \mathbf{w}_1 \right\rangle \right]\mathbb{E}\left[ \overline{ \left\langle \mathbf{x}_{t+\tau}, \mathbf{w}_2 \right\rangle} \right]
\end{align*}
for $t\geq 0$ and lag time $\tau\geq 0$. Moreover, we indicate the time-asymptotic observable with
\begin{align} \label{eq:time-asymptotic autocov}
    V_\infty(\tau)[\mathbf{w}_1,\mathbf{w}_2]= \underset{t\to\infty}{\lim}V_t(\tau)[\mathbf{w}_1,\mathbf{w}_2]. 
\end{align}
For any $\mathbf{w}\in\C^2$ such that $\mathbf{w}\not\in\text{Ker}\left(\Sigma^\text{T}\right)$, we define the time-asymptotic autocorrelation function along $\mathbf{w}$ as
\begin{align} \label{eq:time-asymptotic autocorr}
    AC_\infty(\tau)[\mathbf{w}]:=\frac{V_\infty(\tau)[\mathbf{w},\mathbf{w}]}{V_\infty(0)[\mathbf{w},\mathbf{w}]},
\end{align}
for any $\tau\geq 0$. For $n>1$, these observables refer to similar indicators to those introduced in Definition \ref{def:var and autocor} upon projections of $(\mathbf{x}_t)_{t\geq 0}$ along fixed modes as a scalar stochastic process. We label the time-asymptotic modal autocovariance as the indicator in \eqref{eq:time-asymptotic autocov} for $\mathbf{w}_1$ and $\mathbf{w}_2$ to be taken as the eigenvectors of $M(\lambda)^\text{T}$. Specifically, in the setting \nameref{S2} we consider $\mathbf{w}_1, \mathbf{w}_2 \in\left\{ \mathbf{e}_{1}^*(\lambda), \mathbf{e}_{2}^*(\lambda)\right\}$ for any $\tau\geq 0$ and $\lambda\leq \lambda^*$ in \eqref{eq:fast system simplified}. \\

Among the EWSs studied in Section \ref{sec:fbm 1d} and Section \ref{sec:fbm 2d} are the scaling laws adopted by the observables defined above. We introduce then the asymptotic notation by considering $b_1:(-\infty,\lambda^*]\to \C$ and $b_2:(-\infty,\lambda^*]\to \C$. We say that
\begin{align*}
    b_1(\lambda)=\mathcal{O}\left( b_2(\lambda) \right), \quad \text{if} \quad \lim_{\lambda\to\lambda^*} \left|\frac{b_1(\lambda)}{b_2(\lambda)}\right| \in [0,\infty),
\end{align*}
and that
\begin{align} \label{eq:asymp}
    |b_1(\lambda)|\asymp |b_2(\lambda)|, \quad \text{if} \quad 0<\liminf_{\lambda\to\lambda^*} \left|\frac{b_1(\lambda)}{b_2(\lambda)}\right| \leq \limsup_{\lambda\to\lambda^*} \left|\frac{b_1(\lambda)}{b_2(\lambda)}\right|<\infty.
\end{align}
The $\cO$ notation indicates an upper bound to a scaling law, while $\asymp$ refers to the exact rate adopted by an observable in the critical limit, i.e. asymptotic comparability. However, we note that in the results below we refer to the actual limit within \eqref{eq:asymp}. As a consequence, throughout the paper we write $b_1(\lambda)=\mathcal{O}(1)$ when $|b_1(\lambda)|$ converges to a limit in $[0,\infty)$ as $\lambda\to\lambda^*$, and $|b_1(\lambda)|\asymp1$ when the limit belongs in $(0,\infty)$. Throughout the paper, we use also the property
\begin{align} \label{eq:I_ran_out_of_names}
    \frac{c_1+b_1(\lambda)}{c_2+b_2(\lambda)}
    = \frac{c_1}{c_2}+\cO(b_1(\lambda)) + \cO(b_2(\lambda))
\end{align}
for $c_1,c_2\in\C$ and continuous $b_1,b_2:(-\infty,\lambda^*]\to \C$ such that $b_1(\lambda^*)=b_2(\lambda^*)=0$.

\begin{figure}[h!]
    \centering
    \subfloat{\begin{overpic}[scale=0.7]{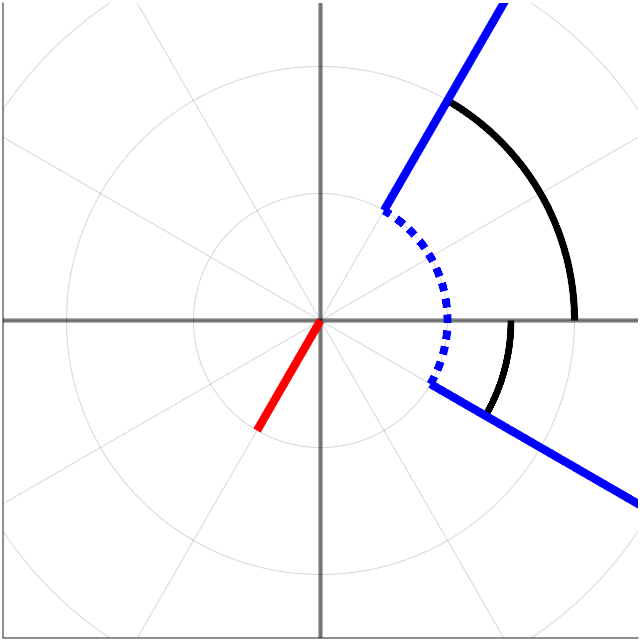}
    \put(1050,490){$\operatorname{Re}$}
    \put(480,1030){$\operatorname{Im}$}
    \put(920,630){$\theta_1$}
    \put(850,390){$\theta_2$}
    \end{overpic}}
    
    \caption{Illustration of the portions of rays in the complex plane along which the incomplete gamma functions are studied. While the upper incomplete gamma function is well-posed only on rays with angles in $\left(-\frac{\pi}{2}, \frac{\pi}{2}\right)$, such as the blue solid lines, the lower incomplete gamma function does not share such a restriction. In fact, its integral can be taken along the red line in the left side of the complex plane. Nonetheless, such a function requires the first argument to be positive to allow for integration on the origin.}
    \label{fig:Illustration_1}
\end{figure}
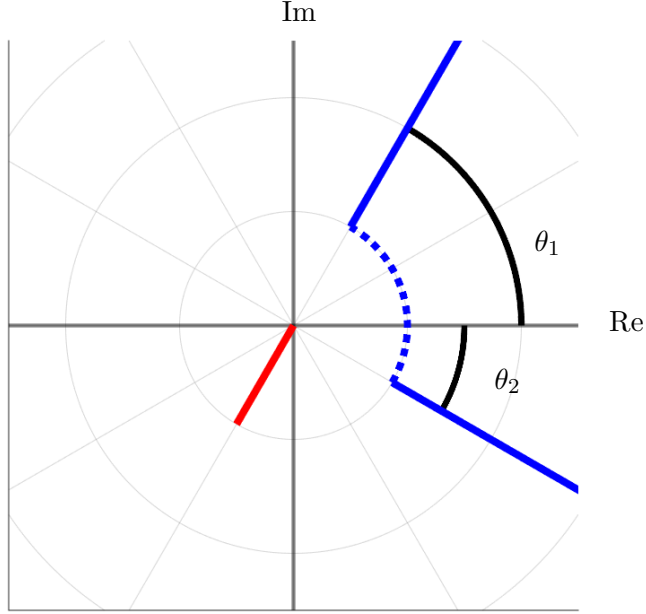 

Since the eigenpairs of $M(\lambda)$ can assume complex values, the study of these indicators requires the use of functions with complex arguments. First, for any $z\in\C\setminus\{0\}$ and $q\in\R$ we introduce the notation
\begin{align*}
    z^q:= \text{exp} \left( q \text{log}|z| + i\; q \arg(z)  \right) .
\end{align*}
In particular, we note that
\begin{align} \label{eq:complex_property_1}
    \arg(z^q) =  q \arg(z) .
\end{align}
We also introduce the gamma function $\Gamma:\R_{>0}\to\R$ such that
\begin{align*}
    \Gamma(s):= \int_0^\infty w^s e^{-w} \text{d}w ,
\end{align*}
the upper incomplete gamma function $\Gamma:\R\times\left\{ z\in\C | \operatorname{Re}(z)>0 \right\} \to \C$ defined by
\begin{align*}
    \Gamma(s,z):= \int_{z}^{e^{i \arg(z)} \infty} w^s e^{-w} \text{d}w
\end{align*}
and the lower incomplete gamma function $\gamma:\R_{>0}\times\C \to \C$ defined by
\begin{align*}
    \gamma(s,z):= \int_0^{z} w^s e^{-w} \text{d}w .
\end{align*}
We note that the integrals of the incomplete gamma functions are taken on the trajectories along the corresponding complex rays. These are defined on the complex plane as half-lines from the origin with angle $\arg(z)$. We finally underline some properties of the introduced functions that are employed in Appendix \ref{app:B}. Consider first $c\in\R_{>0}$ and $\theta_1, \theta_2 \in \left(-\frac{\pi}{2}, \frac{\pi}{2}\right)$ as in the illustration in Figure \ref{fig:Illustration_1}. We obtain by contour integration that $\Gamma(s,c e^{i\theta_1})$, whose integral is taken outwards along a blue solid line, corresponds to the sum of $\Gamma(s,c e^{i\theta_2})$, through integration on the other solid blue line, with an integral along the connection arc, in dotted linestyle. As such, we cannot assume equality among these incomplete gamma functions for $c>0$. In contrast, in the limit $c\to 0$, we obtain $\underset{c\to 0}{\lim} \Gamma(s,c e^{i\theta_1})= \underset{c\to 0}{\lim} \Gamma(s,c e^{i\theta_2})$. By generality of the angles, it follows that 
\begin{align} \label{eq:complex_property_2}
    \underset{c\to 0}{\lim} \Gamma\left(s,c e^{i\theta_1}\right) 
    = \int_{0}^{e^{i \theta_1} \infty} w^s e^{-w} \text{d}w
    = \Gamma(s)
\end{align}
for any $\theta_1 \in \left(-\frac{\pi}{2}, \frac{\pi}{2}\right)$ and $s>0$. We also note that the lower gamma function is well defined for any $z$ with argument in $\left[0, 2\pi\right)$ since $s>0$ and such an integral is not taken in the outer limit of the complex ray.

\section{Warning in one-dimensional models}\label{sec:fbm 1d}
In this section, we study fast subsystems \eqref{eq:fast system simplified} in the setting \nameref{S1}. As we focus on the construction of EWSs of linearized models, we extend the results in \cite{kuehn2022warning} to the assumption of stochastic forcing with Hurst index $H\in\left(0,1\right)$. We employ methods from \cite{bernuzzi2024warning,ChenPSDforHopf,kuehn2022warning} to construct and describe the observables.

\subsection{Autocovariance and autocorrelation} \label{sec:auto 1d}
The time-asymptotic variance and autocorrelation of trajectories in linearized regimes provide classical EWSs under the assumption of stochastic forcing by Brownian motion \cite{ditlevsen2010tipping}. While under memoryless Gaussian forcing, i.e.~$H=\frac{1}{2}$, the first is shown to diverge as $|A(\lambda)|^{-1}$, recent findings~\cite{kuehn2022warning} indicate that positive persistence, i.e.~$H\in\left(\frac{1}{2},1\right)$, enhances its increase in the proximity of the bifurcation. A hyperbolic scaling law of the time-asymptotic autocovariance and an exponential rate of convergence of the time-asymptotic autocorrelation have also been shown for $H=\frac{1}{2}$ for any finite lag time; nonetheless, such results have not been extended to the case of non-Markovian forcing in the literature. In this subsection, we investigate the laws adopted by the discussed observables for $H\in\left(0,1\right)$, providing thus novel findings in the cases of positive and negative persistence. Using the notation defined in Section \ref{sec: S1 and S2}, we introduce the function
\begin{align} \label{eq:function_P}
    \begin{split}
        P(\alpha,\beta,H,\tau):=&e^{\overline{\beta} \tau} (-\overline{\beta})^{-2H+1} \Gamma(2H) \\
        &+ e^{-\alpha \tau} (-\alpha)^{-2H+1} \Gamma(2H,-\alpha\tau)
        + e^{\overline{\beta} \tau} \overline{\beta}^{-2H+1} \gamma(2H,\overline{\beta}\tau),
    \end{split}
\end{align}
for any $\alpha,\beta\in\C$ with negative real part and $\tau\geq 0$. Such a function is employed in the lemma below to study the time-asymptotic autocovariance between two scalar Ornstein-Uhlenbeck process driven by the same non-Markovian stochastic forcing.

\begin{lemma}\label{lem:autocov fast system complex lemma}
Set $\bm{\sigma}_1,\bm{\sigma}_2\in\C^2$. Assume $\left(\mathbf{W}_t^{H}\right)_{t\geq 0}\subset \C^2$ to be an fBm with Hurst index $H\in\left(0,1\right)$. Let $\left(\psi_t^{(1)}\right)_{t\geq 0}$ and $\left(\psi_t^{(2)}\right)_{t\geq 0}$ be scalar solutions to the system
\begin{align}\label{eq:fast system simplified hopf 2}
    \begin{split}
        \txtd \psi_t^{(1)} = \alpha\psi_t^{(1)}~\txtd t + \left\langle \txtd \mathbf{W}_t^{H} , \bm{\sigma}_1 \right\rangle,\\
        \txtd \psi_t^{(2)} = \beta \psi_t^{(2)}~\txtd t + \left\langle \txtd \mathbf{W}_t^{H} , \bm{\sigma}_2 \right\rangle,
    \end{split}
\end{align}
each initial conditions at $t=0$ in $\C$. Lastly, consider $\alpha,\beta\in\C$ with negative real part. Then, for all $\tau\ge0$, we obtain that
\begin{align}\label{eq:autocov fast subsystem complex 2}
    &\underset{t\to\infty}{\lim}\text{Cov}\left(\psi_t^{(1)}, \psi_{t+\tau}^{(2)} \right)
    = H \underbrace{\left\langle \bm{\sigma}_2, \bm{\sigma}_1 \right\rangle}_{\text{masking}} \underbrace{P(\alpha,\beta,H,\tau)}_{\text{memory}} \underbrace{\left( -\alpha - \overline{\beta} \right)^{-1}}_{\text{mitigation}},
\end{align}
with $P$ introduced in \eqref{eq:function_P}.
\end{lemma} 
The lemma is proven in Proof \ref{proof:autocov fast system complex lemma}. The proof employs the Mandelbrot-van-Ness representation, following a similar approach to \cite{kubilius2017parameter} but extending the results to the study of the observable on a time-asymptotic regime and under complex drift. We note that other methods in the literature provide upper bounds to the indicator \cite[Lemma 3.1]{Nils}. As emphasized in \eqref{eq:autocov fast subsystem complex 2}, the time-asymptotic autocovariance of the systems that solve \eqref{eq:fast system simplified hopf 2} can be described by three terms in the limit $\lambda\to\lambda^*$:
\begin{itemize}
    \item The masking term $\left\langle \bm{\sigma}_2, \bm{\sigma}_1 \right\rangle$ describes the correlation between the noise components in \eqref{eq:fast system simplified hopf 2}. Noise with uncorrelated components or null noise can induce masking and qualitatively affect the rate of divergence of the observable.
    \item The memory item $P(\alpha,\beta,H,\tau)$ collects the effect of persistence and lag in the EWS.
    \item The mitigation term $\left( -\alpha - \overline{\beta} \right)^{-1}$ observes the loss of stability of $\psi^{(1)}$ and $\psi^{(2)}$. It therefore refers to the mitigation of the deterministic forcing at the base of the critical slowing down phenomenon.
\end{itemize}

In the theorem below, we prove the divergence of the time-asymptotic autocovariance $V_\infty(\tau)=\lim\limits_{t\to\infty}V_t(\tau)$ for all $\tau\ge0$. Moreover, we quantify the effect of memory and persistence in the EWS associated to \eqref{eq:fast system simplified} in the setting \nameref{S1}. As such, for the remainder of the section we consider real drift and noise intensity terms.

\begin{theorem}\label{thm:autocov fast system}
Let $H\in(0,1)$, $\tau\ge0$ and $(x_t)_{t\geq 0}$ be a solution to \eqref{eq:fast system simplified} in the setting \nameref{S1}. Then, we obtain that 
\begin{align}\label{eq:autocov fast subsystem}
    V_\infty(\tau)=& \frac{\sigma^2}{2}H|A(\lambda)|^{-2H}\\
    &\times \left( e^{A(\lambda)\tau}\Gamma(2H) + e^{-A(\lambda)\tau}\Gamma(2H,|A(\lambda)|\tau)-e^{A(\lambda)\tau} \int_0^{|A(\lambda)|\tau} e^w w^{2H-1}~\txtd w \right). \nonumber
\end{align}
\end{theorem}
The proof follows directly from Lemma \ref{lem:autocov fast system complex lemma} upon taking $\psi_t^{(1)}=\psi_t^{(2)}=x_t$ for any $t\geq0$, with $(x_t)_{t\geq0}$ the solution of \eqref{eq:fast system simplified} in \nameref{S1}. As a direct consequence, the variance is given by
\begin{align}\label{eq:Var}
    V_\infty(0) = \sigma^2H\Gamma(2H)|A(\lambda)|^{-2H}
\end{align}
for $H\in(0,1)$. This coincides with the results in \cite{kuehn2022warning,kuehn2019scaling} where \eqref{eq:Var} was proven to hold for $H\in\left[\frac{1}{2},1\right)$. Since the term in the second row of \eqref{eq:autocov fast subsystem} is nonzero almost everywhere, it follows that
\begin{align}\label{eq:Autovar}
    |V_\infty(\tau)| \asymp |A(\lambda)|^{-2H}
\end{align}
for $H\in(0,1)$, $\sigma>0$ and almost every $\tau\geq 0$. Moreover, due to its continuity in $\tau$, \eqref{eq:Autovar} holds in a neighborhood of $\tau=0$. Thus, it provides a precise scaling law of the observable in the linearized regime. Moreover, the rate of divergence does not depend on the lag time $\tau$. However, we see in Section \ref{sec:applications} that the lag leads to a delay in the divergence of $V_\infty(\tau)$. In the following corollary, we employ these results to describe the rate of convergence of the time-asymptotic autocorrelation for the system \eqref{eq:fast system simplified} with $H\in(0,1)$ and in setting \nameref{S1}.
\begin{corollary}\label{cor:autocor fast system}
    Let $H\in(0,1)$, $\sigma>0$, $\tau\ge0$ and $(x_t)_{t\geq 0}$ be a solution to \eqref{eq:fast system simplified} that satisfies \nameref{S1}. Then, we get that
\begin{align}\label{eq:autocor fast subsystem}
    AC_\infty(\tau)= 1-|A(\lambda)|^{2H}\frac{\tau^{2H}}{\Gamma(2H+1)}+\cO(|A(\lambda)|^2).
\end{align}
\end{corollary}
The proof is given in the Proof \ref{proof:autocor fast system}. First, we observe that $AC_\I(\tau)$ converges to $1$ as $\lambda\to\lambda^*$ for all $\tau\ge0$. The dependence on $H$ and on $\tau$ of the time-asymptotic autocorrelation in the critical limit further describes the characteristic of the tipping phenomenon. It indicates a delay of information given by the lag time and negative persistence, associated with a slower rate of convergence as shown in \eqref{eq:autocor fast subsystem}. In contrast, the case $H\in\left(\frac{1}{2},1\right)$ presents a faster convergent behavior in the limit $\lambda\to\lambda^*$. Finally, we note that for $H\neq\frac{1}{2}$ and any $\tau>0$, the time-asymptotic autocorrelation $AC_\infty(\tau)$ is not analytic in $\lambda^*$.

\subsection{Spectral density}\label{sec:SD S1 fBm}
We now consider the spectral density (SD) as an EWS. This has been used in~\cite{ChenPSDforHopf} as an EWS to detect a Hopf bifurcation in a system driven by white noise, whereas in~\cite{morr2024detection} the SD was used to analyze different formulations of red noise. ~We analyze the SD for systems driven by fBm with Hurst index $H\in(0,1)$. \\
As the SD is only defined for stationary processes, we consider the stationary solution to \eqref{eq:fast system simplified} in the setting \nameref{S1}, instead of taking a solution and letting $t\to\infty$. As we consider a stationary solution we write $V(\tau)$ and drop the time argument. First we recall the SD of a fOU process.
\begin{lemma}\label{lem:fBm 1d PSD}
    Let $H\in(0,1)$ and let $(x_t)_{t\geq 0}$ be a stationary solution to \eqref{eq:fast system simplified} in the setting \nameref{S1}. Then, we get for all $\omega\in\R$
    \begin{align}\label{eq:PSD fBM}
        S_x(\omega) &= \sigma^2\Gamma(2H+1)\sin(\pi H) \frac{|\omega|^{1-2H}}{A(\lambda)^2 + \omega^2}\nonumber \\
        &=\sigma^2C_H \frac{|\omega|^{1-2H}}{A(\lambda)^2 + \omega^2}.
    \end{align}
\end{lemma}
\begin{proof}
    Using \eqref{eq:Cov Pipiras Taqqu}, it is shown in \cite[Remark 2.4]{Cheridito_fOU} that for $\tau>0$ we obtain
    \begin{align*}
        V(\tau) = \sigma^2\frac{\Gamma(2H+1)\sin(\pi H)}{2\pi}\int_{-\I}^\I e^{i\tau z}\frac{|z|^{1-2H}}{A(\lambda)^2+z^2}~\txtd z,
    \end{align*}
    which includes an inverse Fourier transform. Hence, using the definition \eqref{eq:def spectral density} and applying the Fourier transform we get
    \begin{align*}
        S_x(\omega) = \mathcal{F}(V(\tau)) = \sigma^2\Gamma(2H+1)\sin(\pi H)\frac{|\omega|^{1-2H}}{A(\lambda)^2+\omega^2}.
    \end{align*}
    \qed \\
\end{proof}
As was already seen for the white noise case ($H=\frac{1}{2}$), the SD undergoes 'reddening' as the bifurcation is approached, meaning that the power in the low frequencies increases as $\lambda\to\lambda^*$. But this on its own is not easy to track as an EWS. 
Instead, we introduce the following statistics to quantify the increase of the SD.
\begin{definition}
    Let $(x_t)_{t\geq 0}$ be a stationary process and $S_x(\omega)$ be its SD. If $S_x$ is bounded and satisfies $S_x(\omega)\xrightarrow{\omega\to\pm\I}0$, we define 
    \begin{align}\label{eq:PSD max}
        S_x^{\max}:= \max\limits_{\omega\in\R}S_x(\omega).
    \end{align}
\end{definition}
In \cite{ChenPSDforHopf}, $S_x^{\max}$ was used for numerical simulations of systems driven by white noise, but was not computed analytically. Since the maximum of the SD does not always exist, we only compute the divergence rate of $S^{\max}_x$ for the regime in which the maximum is attained. 
\begin{lemma}\label{lem:PSD_max for fBM}
    Let $H\in\left(0,\frac{1}{2}\right]$ and $(x_t)_{t\geq 0}$ be a solution to \eqref{eq:fast system simplified} in the setting \nameref{S1}. Then, $S_x^{\max}$ is well-defined and we have
    $$|S^{\max}_x|\asymp |A(\lambda)|^{-1-2H},$$
    as $\lambda\to\lambda^*$. 
\end{lemma}
\begin{proof}
    By definition, we have
    \begin{align*}
        S^{\max}_x
        &=\sigma^2\Gamma(2H+1)\sin(\pi H) \max\limits_{\omega\in\R} \frac{|\omega|^{1-2H}}{A(\lambda)^2+\omega^2}.
    \end{align*}
    As $1-2H\ge0$ we find that the maximum is taken at $\omega^{\max}(\lambda)=\sqrt{\frac{1-2H}{1+2H}}A(\lambda)$ and that 
    $$S^{\max}_x = \sigma^2\Gamma(2H+1)\sin(\pi H)\left(\frac{1-2H}{1+2H}\right)^{\frac{1-2H}{2}} \frac{1+2H}{2} A(\lambda)^{-1-2H}.$$
    \qed
\end{proof}\\

As for $H\in\left(\frac{1}{2},1\right)$ the SD has a singularity at $\omega=0$, we cannot compute the maximum in that regime. But from \eqref{eq:PSD fBM} we can still expect the SD to undergo reddening. Hence, we now show how the SD behaves for a fixed frequency $\delta>0$ close to zero.
\begin{lemma}\label{lem:PSD0 for fBm}
    Let $H\in(0,1)$ and $(x_t)_{t\geq 0}$ be a solution to \eqref{eq:fast system simplified} in the setting \nameref{S1}. Then, for any $\delta>0$, there is a local regime $|A(\lambda)|>\delta$ in which $S_x(\delta)$ has power-law behavior with exponent $-2$ meaning  
    $$\frac{\txtd \log(S_x(\delta))}{\txtd \log(|A(\lambda)|)} = -2.$$
    Furthermore, for $\lambda\to\lambda^*$ we have
    $$S_x(\delta)\xrightarrow{\lambda\to\lambda^*}\sigma^2\Gamma(2H+1)\sin(\pi H)|\delta|^{-1-2H}.$$
\end{lemma}
\begin{proof}
    Let $\delta>0$. 
    For $H=\frac{1}{2}$, we notice that $\log(S_x(0))$ is well-defined, but not for all other $H\neq \frac{1}{2}$.~This is a fundamental difference in comparison to the Brownian motion.~Therefore, it is necessary to consider the regime in which $|A(\lambda)|>\delta$. We plug in \eqref{eq:PSD fBM} and use the chain rule to get
    \begin{align*}
        \frac{\txtd \log(S_x(\delta))}{\txtd \log(|A(\lambda)|)}
        &=\frac{\txtd (\log(|\delta|^{1-2H})-\log(\delta^2+A(\lambda)^2))}{\txtd \log(|A(\lambda)|)} \\
        &=|A(\lambda)|\frac{\txtd}{\txtd |A(\lambda)|} (-\log(\delta^2+A(\lambda)^2)) \\
        &= -2\frac{A(\lambda)^2}{A(\lambda)^2+\delta^2}.            
    \end{align*}
    In conclusion, as long as $|A(\lambda)|$ is larger than $\delta$, we obtain a power-law behavior with exponent $-2$.~The claim about the convergence of $S_x(\delta)$ holds trivially by plugging in $A(\lambda)=0$ into $S_x(\delta)$ to obtain 
    $$S_x(\delta)|_{A(\lambda)=0}=\sigma^2\Gamma(2H+1)\sin(\pi H) |\delta|^{-1-2H}.$$
    \qed\\
\end{proof}

\begin{remark}\label{remark:sd}
\begin{enumerate}
    \item [1)] For $H=\frac{1}{2}$ we can set $\delta=0$ in 
    $$\frac{\txtd \log(S_x(\delta))}{\txtd \log(|A(\lambda)|)}
        = -2\frac{A(\lambda)^2}{A(\lambda)^2+\delta^2},$$
    and therefore even have $|S_x(0)|\asymp A(\lambda)^{-2}$.
    \item [2)] Note that $\delta$ should be chosen close to zero, to allow for a large regime with a power-law behavior.
    \item [3)] Depending on whether we consider $S_x^{\max}$ or $S_x(\delta)$ as an EWS, the memory of the noise influences them in different ways but it is present in both cases. For $S_x^{\max}$ we can only compute it for the range $H\in\left(0,\frac{1}{2}\right]$ and $H$ appears in the exponent. For $S_x(\delta)$ the memory of the noise effects it more subtly. As for $H>\frac{1}{2}$ the SD has a singularity at zero and for $H<\frac{1}{2}$ the SD is zero in zero, the memory makes us unable to consider $S_x(0)$ directly as an EWS. Instead, we need to consider a frequency close to zero to capture an EWS.
    \item [4)] Similarly to \eqref{eq:autocov fast subsystem complex 2}, we can find a mitigation term, memory term and masking term in the SD. Here $\frac{1}{A(\lambda)^2+\omega^2}$ is the mitigation term as this is clearly the part which causes the increase as we approach the bifurcation. Furthermore, $C_H|\omega|^{1-2H}$ is the memory term, as this captures the dependence on $H$. Last, $\sigma^2$ can be interpreted as a the masking term. In Section \ref{sec:red noise} we see how the SD of the noise can be a part of the masking term.
\end{enumerate}
\end{remark}

A different EWS based on the SD was introduced in \cite{Clarke_2023}. 
\begin{definition}
    Consider a differential equation 
    $$\txtd x_t = M(\lambda) x_t + \xi_t~\txtd t,$$
    where $\xi$ is a stationary process with SD $S_\xi(\omega)$.
    Then, the ratio of spectra (ROSA) is defined as
    \begin{align*}
        R(\omega) = \frac{S_x(\omega)}{S_{\xi}(\omega)}.
    \end{align*}
\end{definition}
\begin{corollary}
    In the setting \nameref{S1}, we immediately obtain that the ROSA of \eqref{eq:fast system simplified} is
$$R(\omega)=\frac{1}{A(\lambda)^2+\omega^2}.$$
\end{corollary}

In order to compute the ROSA we need to assume that $S_{\xi}(\omega)$ is known. In \cite{Clarke_2023},  ROSA is computed from the data and then $A(\lambda)$ is estimated from ROSA. The EWS is then given by $A(\lambda)$ approaching zero.
On the other hand, ROSA can directly be used as an EWS. Obviously, $R(0)=A(\lambda)^{-2}$, so plotting ROSA in a loglog plot gives a slope of $-2$ whenever a bifurcation is approached.

As can clearly be seen, the power-law behavior of ROSA is similar to the SD at some fix frequency $S_x(\delta)$. However, they also have crucial differences. In order to compute ROSA we need explicit knowledge of $S_\xi(\omega)$. According to \cite{Clarke_2023}, ROSA is resistant to false positives, meant as an increase in variance even though there is no approaching bifurcation, due to changes in the noise $\xi$ over time.
An example for this can be seen in \cite[Figure 2]{Clarke_2023}. However, under the assumption of explicit knowledge of the noise $\xi$ or the SD $S_\xi(\omega)$, one can adjust the variance and the SD in such a way that it resistant to false positives due to change in noise. Both EWSs are resistant to false negatives due to specific type of noise such as a fractional Ornstein-Uhlenbeck process considered in Section \ref{sec:red noise}.

\section{Warning of Hopf bifurcations} \label{sec:fbm 2d}
In this section, we construct EWSs for the approach to a Hopf bifurcation \cite{kuehn2015multiple} in a stochastic fast system \eqref{eq:fast system simplified} in the setting \nameref{S2}. We assume the stochastic forcing to be a fractional Brownian motion in $\R^2$ with Hurst index $H\in\left(0,1\right)$. As discussed further in this work, the rotatory phenomenon preceding the rise of a stable cycle strongly affects the behavior of the observables prior to the critical transition. The case $H=\frac{1}{2}$ in \cite{kuehn2019scaling} provides a comparison of regimes in the study of the time-asymptotic variance. It follows that the rise of a limit cycle hinders memory persistence and the rates in $H\in\left(0,1\right)$ are qualitatively equivalent. More specifically, the spinning frequency in the critical threshold, $B(\lambda^*)\neq 0$, provides different scalings than the ones described in Section \ref{sec:fbm 1d} as a consequence of the convergence of the memory terms \eqref{eq:function_P} and the onset of multiple peaks in the spectral density.

\subsection{Autocovariance and autocorrelation} \label{sec:auto 2d}

In the theorem below, we implement Lemma \ref{lem:autocov fast system complex lemma} to study the solution of \eqref{eq:fast system simplified} in the setting \nameref{S2} along spectral and general eigenmodes. In its statement, we employ the notation introduced in Section \ref{sec: S1 and S2} and \eqref{eq:function_P}. We then define the symmetric matrix $Q:=\Sigma \Sigma^{\text{T}}$.

\begin{theorem}\label{thm:autocov fast system complex}
Let $H\in\left(0,1\right)$ and $(\mathbf{x}_t)_{t\geq 0}$ be a solution to \eqref{eq:fast system simplified} in the setting \nameref{S2}. For any $\tau\ge0$, it holds that
\begin{align}\label{eq:autocov fast subsystem complex spectral}
    \begin{split}
    &V_\infty(\tau)\left[ \mathbf{e}_{j_1}^*(\lambda), \mathbf{e}_{j_2}^*(\lambda) \right]\\
    =& \left(-\zeta_{j_1}(\lambda)-\overline{\zeta_{j_2}(\lambda)}\right)^{-1} H \left\langle  \mathbf{e}_{j_2}^*(\lambda), Q \mathbf{e}_{j_1}^*(\lambda)\right\rangle P(\zeta_{j_1}(\lambda),\zeta_{j_2}(\lambda),H,\tau) ,
    \end{split}
\end{align}
with $\zeta_1(\lambda)= A(\lambda) + i B(\lambda)= \overline{\zeta_2(\lambda)}$ and for any $j_1,j_2\in\{1,2\}$. Under the assumption that $\Sigma$ is not null, it follows that
\begin{align}\label{eq:autocov fast subsystem complex 3}
    |V_\infty(\tau)\left[ \mathbf{v}_1, \mathbf{v}_2 \right]|\asymp& |A(\lambda)|^{-1},
\end{align}
for almost every $\tau\ge0$ and any $\mathbf{v}_1, \mathbf{v}_2 \in \R^2$ such that
\begin{equation} \label{eq: civil war}
    \begin{split}
        &e^{-i B(\lambda^*)\tau} \left\langle  \mathbf{e}_{1}^*(\lambda^*), Q \mathbf{e}_{1}^*(\lambda^*)\right\rangle \overline{\left\langle \mathbf{v}_1, \mathbf{e}_{1}(\lambda^*) \right\rangle} \left\langle \mathbf{v}_2, \mathbf{e}_{1}(\lambda^*) \right\rangle\\
        &+
        e^{i B(\lambda^*)\tau} \left\langle  \mathbf{e}_{2}^*(\lambda^*), Q \mathbf{e}_{2}^*(\lambda^*)\right\rangle \overline{\left\langle \mathbf{v}_1, \mathbf{e}_{2}(\lambda^*) \right\rangle} \left\langle \mathbf{v}_2, \mathbf{e}_{2}(\lambda^*) \right\rangle \neq 0.
    \end{split}
\end{equation}
\end{theorem}

The proof of the theorem can be found in Proof \ref{proof:autocov fast system complex}. Its results are multiple. First, we know from \eqref{eq:autocov fast subsystem complex 3} that the scaling law of the time-asymptotic autocovariance is not qualitatively dependent on $H$.~This indicates that the rotatory forcing induced by the rising limit cycle dampens the memory property from the stochastic forcing. Such an effect is observed by the time-asymptotic modal autovariances in \eqref{eq:autocov fast subsystem complex spectral} in $j_1=j_2$. Moreover, $\left|V_\infty(\tau)\left[\mathbf{e}_j^*(\lambda),\mathbf{e}_j^*(\lambda)\right]\right| \asymp \left|V_\infty(\tau)\left[\mathbf{v}_1,\mathbf{v}_2\right]\right|$ for fixed $\tau \geq 0$ and non-zero vectors $\mathbf{v}_1, \mathbf{v}_2 \in \R^2$ satisfying the structural condition that, for any fixed $\mathbf{v}_1$ and strictly positive $Q$, the admissible choices of $\mathbf{v}_2$ form a dense subset of $\R^2$ (and vice versa). This follows from \eqref{eq: civil war}, which shows that both $\mathbf{v}_1$ and $\mathbf{v}_2$ receive contributions from stochastic perturbations across $\mathbf{e}_1^*(\lambda^*)$ or $\mathbf{e}_2^*(\lambda^*)$, and that the sum of the corresponding time-asymptotic modal autovariances does not vanish. Each term highlighted in \eqref{eq:autocov fast subsystem complex 2} affects the rate of divergence in \eqref{eq:autocov fast subsystem complex spectral}:
\begin{itemize}
    \item The masking item $\left\langle  \mathbf{e}_{j_2}^*(\lambda), Q \mathbf{e}_{j_1}^*(\lambda)\right\rangle$ quantifies the effect of stochastic perturbations along the spectral modes. In the case the modes enter in the kernel of $Q$, vanishing noise can hinder the scaling law;
    \item The memory item $P(\zeta_{j_1}(\lambda),\zeta_{j_2}(\lambda),H,\tau)$ does not diverge for any $\tau\geq0$ and spectral mode in the limit $\lambda\to\lambda^*$ due to the fact that $B(\lambda^*)\neq 0$. This is in contrast to the assumption of Theorem \ref{thm:autocov fast system}, where the memory affects the scaling law. Its continuity and analiticity imply the fact that it is not zero on $\lambda=\lambda^*$ for almost every $\tau\geq0$ and for all $\tau$ in a neighborhood of zero;
    \item The mitigation term $\left(-\zeta_{j_1}(\lambda)-\overline{\zeta_{j_2}(\lambda)}\right)^{-1}$ clearly diverges only for $j_1=j_2$ and is the only diverging term in the setting \nameref{S2}.
\end{itemize}
The next lemma describes the limit and the adopted scaling law of the time-asymptotic autocorrelation of the solution in \eqref{eq:fast system simplified hopf 2} in the limit $\lambda\to\lambda^*$. The lemma is proven in Proof \ref{proof:autocor fast system complex} and relies on a Taylor expansion approach of the memory term. In fact, we note that the rate of the time-asymptotic autocorrelation is not affected by the limit of the masking and mitigation terms.

\begin{lemma}\label{lem:autocor fast system complex}
    Let $H\in\left(0,1\right)$ and $\left(\psi_t^{(1)}\right)_{t\geq 0}$ be a solution to the first equation in \eqref{eq:fast system simplified hopf 2} with $||\bm{\sigma}_1||^2\neq 0$. Consider $\alpha=\alpha(\lambda)\in\C$ such that $\underset{\lambda\to\lambda^*}{\lim} \alpha(\lambda)=\alpha(\lambda^*)\in i \R\setminus\{0\}$. Then, it holds that
\begin{align}\label{eq:autocor fast subsystem complex}
    \left| \frac{\underset{t\to\infty}{\lim}\text{Cov}\left(\psi_t^{(1)}, \psi_{t+\tau}^{(1)} \right)}{\underset{t\to\infty}{\lim}\text{Cov}\left(\psi_t^{(1)}, \psi_{t}^{(1)} \right)} - e^{-i \operatorname{Im}(\alpha(\lambda))\tau} \right| = -\operatorname{Re}(\alpha(\lambda)) \tau \Delta(\tau) + \mathcal{O}\left( \operatorname{Re}(\alpha(\lambda))^2 \right) ,
\end{align}
with $\Delta(\tau)\neq 0$ for almost every $\tau \geq 0$. In particular, for almost every $\tau\ge0$, we have that
\begin{align*}
    &\frac{\underset{t\to\infty}{\lim}\text{Cov}\left(\psi_t^{(1)}, \psi_{t+\tau}^{(1)} \right)}{\underset{t\to\infty}{\lim}\text{Cov}\left(\psi_t^{(1)}, \psi_{t}^{(1)} \right)}
    = e^{-i \operatorname{Im}(\alpha(\lambda^*))\tau} 
    + \cO \left(\operatorname{Re}(\alpha(\lambda))\right) + \cO\left( \operatorname{Im}(\alpha(\lambda)-\alpha(\lambda^*))\right) .
\end{align*}
\end{lemma}

Finally, we employ Theorem \ref{thm:autocov fast system complex} and Lemma \ref{lem:autocor fast system complex} to further study the behavior of the solution to \eqref{eq:fast system simplified} in the setting \nameref{S2}. Specifically, we obtain the limit and the rate of convergence of its time-asymptotic autocorrelation along general functions $\mathbf{v}\in\R^2$.

\begin{corollary} \label{cor:autocor fast system complex}
Let $H\in\left(0,1\right)$ and $(\mathbf{x}_t)_{t\geq 0}$ be a solution to \eqref{eq:fast system simplified} in the setting \nameref{S2} with non-null $\Sigma$. Then, for almost every $\tau\ge0$, it holds that
\begin{align*}
    &\left| AC_\infty(\tau)\left[\mathbf{v}\right]
    - \left( c_1(\lambda^*) e^{-i B(\lambda)\tau}
    + c_2(\lambda^*) e^{i B(\lambda)\tau} \right)
    \right|\\
    =& \cO\left(-A(\lambda)\right) + \cO\left(\left|\left| \mathbf{e}_1^*(\lambda) - \mathbf{e}_1^*(\lambda^*) \right|\right| \right) + \cO\left(\left|\left| \mathbf{e}_2^*(\lambda) - \mathbf{e}_2^*(\lambda^*) \right|\right| \right)
\end{align*}
for any $\mathbf{v}$ in a dense subset of $\R^2$ and in the limit $\lambda\to\lambda^*$. The complex parameters $c_j(\lambda)$ are defined for any $j\in\{1,2\}$ and $\lambda\leq \lambda^*$ by
\begin{align*}
    c_j(\lambda):= \frac{ \left|\left\langle \mathbf{v}, \mathbf{e}_j(\lambda) \right\rangle\right|^2 \left|\left| \Sigma^\text{T} \mathbf{e}_j^*(\lambda) \right|\right|^2}{ \left|\left\langle \mathbf{v}, \mathbf{e}_1(\lambda) \right\rangle\right|^2 \left|\left| \Sigma^\text{T} \mathbf{e}_1^*(\lambda) \right|\right|^2 + \left|\left\langle \mathbf{v}, \mathbf{e}_2(\lambda) \right\rangle\right|^2 \left|\left| \Sigma^\text{T} \mathbf{e}_2^*(\lambda) \right|\right|^2},
\end{align*}
for $\mathbf{v}$ in the dense subset mentioned above. Moreover, we have that
\begin{align*}
    AC_\infty(\tau)\left[\mathbf{v}\right] =& c_1(\lambda^*) e^{-i B(\lambda^*)\tau}
    + c_2(\lambda^*) e^{i B(\lambda^*)\tau} + \cO\left(-A(\lambda)\right) \\
    &+ \cO\left(\left|\left| \mathbf{e}_1^*(\lambda) - \mathbf{e}_1^*(\lambda^*) \right|\right| \right) + \cO\left(\left|\left| \mathbf{e}_2^*(\lambda) - \mathbf{e}_2^*(\lambda^*) \right|\right| \right) + \cO\left( |B(\lambda)-B(\lambda^*)| \right) .
\end{align*}
\end{corollary}

The proof of the corollary can be found in Proof \ref{proof: autocor fast system complex cor}. With its statement, we have proven that the damping of memory due to arising spinning occurs also in the observation of the time-asymptotic autocorrelation. Its rate of convergence is unaffected by $H\in\left(0,1\right)$. Moreover, its scaling law is equivalent to that of case $H=\frac{1}{2}$, as shown in \cite{bernuzzi2024warning} along spectral modes. On more general directions in $\R^2$, the rate depends also on the behavior of $\mathbf{e}_1^*(\lambda)$ and $\mathbf{e}_2^*(\lambda)$ in $\lambda\to\lambda^*$. This is given by the fact that two distinct eigenvalues cross the imaginary axis in the limit.

\subsection{Spectral density}
Next we show how the SD can be used as an EWS for stationary solutions of \eqref{eq:fast system} in the setting \nameref{S2}, extending the results in Section \ref{sec:SD S1 fBm} to the two-dimensional case. 
\begin{lemma}\label{lem:PSD S2 fBm}
    Let $(\mathbf{x}_t)_{t\geq 0}$ be the stationary solution of \eqref{eq:fast system} in the setting \nameref{S2}. Then, the SD of $\mathbf{x}$ denoted by $ S_{\mathbf{x}\mathbf{x}}$ is given for all $H\in(0,1)$ by
    \begin{align}\label{eq:PSD for n-dim}
    S_{\mathbf{x}\mathbf{x}}(\omega) = C_H\left(M(\lambda)-i\omega I\right)^{-1}Q\left(M(\lambda)^*+i\omega I\right)^{-1}|\omega|^{1-2H}.
    \end{align}
\end{lemma}
\begin{proof}
The stationary solution is given by
$$\mathbf{x}_t = \int_{-\infty}^t e^{-M(\lambda) (t-u)} \Sigma ~\txtd W_u^H=\sum_{j=1}^2\int_{-\infty}^t e^{-M(\lambda) (t-u)} \bm{\sigma}_j ~\txtd W_u^{H,j},$$
where $\bm{\sigma}_j$ are the columns of $\Sigma$.
Then, the autocovariance can be computed to be
\begin{align*}
    V(\tau) 
    &=\E\left[\left(\sum_{j=1}^2 \int_{-\I}^t e^{-M(\lambda)(t-u)} \sigma_{j} ~\txtd W_u^{H,j} \right)\left(\sum_{k=1}^2 \int_{-\I}^{t+\tau} e^{-M(\lambda)(t+\tau-u)} \sigma_{k} ~\txtd W_u^{H,k} \right)^* \right]\\
    &=\sum_{j=1}^2\E\left[\left( \int_\R \mathbbm{1}_{\{u\le0\}}e^{M(\lambda)u} \sigma_{j} ~\txtd W_u^{H,j} \right)\left( \int_\R e^{-M(\lambda)\tau}\mathbbm{1}_{\{u\le\tau\}} e^{M(\lambda)u} \sigma_{j} ~\txtd W_u^{H,j} \right)^* \right]\\
    &=\sum_{j=1}^2 \left\langle \mathbbm{1}_{\{u\le0\}} e^{M(\lambda)u}\bm{\sigma}_j, e^{-M(\lambda)\tau}\mathbbm{1}_{\{u\le\tau\}} e^{M(\lambda)u}\bm{\sigma}_j  \right\rangle_{\tilde\Lambda_H} ,
\end{align*}
where we used \eqref{eq:Cov Pipiras Taqqu}. We note that the Fourier transform of $\mathbbm{1}_{\{u\le0\}} e^{M(\lambda)u}$ is given by
\begin{align}\label{eq:Fourier transform of exponential}
    \mathcal{F}\left(\mathbbm{1}_{\{u\le\tau\}} e^{M(\lambda)u}\right)(z)=e^{iz\tau}\left(M(\lambda)-izI\right)^{-1}e^{M(\lambda)\tau }=e^{iz\tau}e^{M(\lambda)\tau }\left(M(\lambda)-izI\right)^{-1}.
\end{align}
Now we use the definition of the inner product on the space $\tilde\Lambda_H$ to get
\begin{align*}
    &\sum_{j=1}^2 \left\langle \mathbbm{1}_{\{u\le0\}} e^{M(\lambda)u}\bm{\sigma}_j, e^{M(\lambda)\tau}\mathbbm{1}_{\{u\le\tau\}} e^{M(\lambda)u}\bm{\sigma}_j  \right\rangle_{\tilde\Lambda_H}\\
    =& \sum_{j=1}^2 \frac{C_H}{2\pi}\int_\R \left(M(\lambda)-izI\right)^{-1}\bm{\sigma}_j\bm{\sigma}_j^*\left(M(\lambda)^*+izI\right)^{-1}e^{iz\tau}e^{M(\lambda)^*\tau}e^{-M(\lambda)^*\tau}|z|^{1-2H}~\txtd z \\
    =& \frac{C_H}{2\pi}\int_\R \left(M(\lambda)-izI\right)^{-1}Q\left(M(\lambda)^*+izI\right)^{-1}e^{iz\tau}|z|^{1-2H}~\txtd z,
\end{align*}
where $C_H=\Gamma(2H+1)\sin(\pi H)$.
Hence, using the definition of the spectral density \eqref{eq:def spectral density} it is given by 
\begin{align*}
    S_{\mathbf{x}\mathbf{x}}(\omega) = C_H\left(M(\lambda)-i\omega I\right)^{-1}Q\left(M(\lambda)^*+i\omega I\right)^{-1}|\omega|^{1-2H}.
\end{align*}
\qed \\
\end{proof}
Since the SD is a matrix, we introduce the notation $S_\mathbf{x}^{j_1j_2}(\omega):=\left(S_\mathbf{x}(\omega)\right)_{j_1 j_2}$ to refer to its entries. Our goal is to use the SD as an EWS. We note that as in the setting \nameref{S1}, the SD undergoes reddening. However, in setting \nameref{S2} we can consider the SD at the frequency $\pm B(\lambda)$, which corresponds to the frequencies at which the SD has a peak. For these frequencies, we can observe the divergence due to the approaching bifurcation without any interference of the noise. 
\begin{theorem} \label{thm:psd_on_elements}
    Let $(\mathbf{x}_t)_{t\geq 0}$ be the stationary solution of \eqref{eq:fast system} in the setting \nameref{S2}. Then, we have for all $H\in(0,1)$
    $$S_{\mathbf{x}}^{11}(\pm B(\lambda))\asymp A(\lambda)^{-2},~~~S_{\mathbf{x}}^{22}(\pm B(\lambda))\asymp A(\lambda)^{-2},~~~ |S_{\mathbf{x}}^{12}(\pm B(\lambda))|\asymp A(\lambda)^{-2},$$
    where $A(\lambda)=\Re(M(\lambda)),B(\lambda)=\Im(M(\lambda))$.
    Moreover, the behavior of the SD at $\omega=0$ is dictated by $H$. 
\end{theorem}
\begin{proof}
We start by computing
\begin{align*}
    \left(M(\lambda)-i\omega I\right)^{-1}Q\left(M(\lambda)^*+i\omega I\right)^{-1}.
\end{align*}
First, since $M(\lambda)$ has complex conjugate eigenvalues, it has the structure
$$M(\lambda)=\left(\begin{array}{cc}
    A(\lambda) & -B(\lambda) \\
    B(\lambda) & A(\lambda)
\end{array}\right).$$
Hence,
\begin{align*}
    \left(M(\lambda)-i\omega I\right)^{-1} &= \frac{1}{(A(\lambda)-i\omega)^2+B(\lambda)^2}
    \left(\begin{array}{cc}
    A(\lambda)-i\omega & B(\lambda) \\
    -B(\lambda) & A(\lambda)-i\omega
\end{array}\right) \\
\left(M(\lambda)^*+i\omega I\right)^{-1} &= \frac{1}{(A(\lambda)+i\omega)^2+B(\lambda)^2}
    \left(\begin{array}{cc}
    A(\lambda)+i\omega & -B(\lambda) \\
    B(\lambda) & A(\lambda)+i\omega
\end{array}\right).
\end{align*}
Next, we note that
\begin{align*}
    &\frac{1}{(A(\lambda)-i\omega)^2+B(\lambda)^2}\frac{1}{(A(\lambda)+i\omega)^2+B(\lambda)^2}\\
    =& \frac{1}{2}\left(\frac{1}{(\omega+B(\lambda))^2+A(\lambda)^2}+\frac{1}{(\omega-B(\lambda))^2+A(\lambda)^2}\right).
\end{align*}
Through this, we get the following elements of the matrix:

\begin{align*}
    S_\mathbf{x}^{11}(\omega)
    &=\frac{|\omega|^{1-2H}}{2} \left(\frac{1}{(\omega+B(\lambda))^2+A(\lambda)^2}+\frac{1}{(\omega-B(\lambda))^2+A(\lambda)^2}\right)\\
    &~~~\times\Big((A(\lambda)^2+\omega^2)Q_{11} + B(\lambda)^2Q_{22} - 2A(\lambda)B(\lambda) Q_{12}\Big) ;\\
    S_\mathbf{x}^{12}(\omega)
    &= \frac{|\omega|^{1-2H}}{2} \left(\frac{1}{(\omega+B(\lambda))^2+A(\lambda)^2}+\frac{1}{(\omega-B(\lambda))^2+A(\lambda)^2}\right) \\
    &~~~\times\Big((A(\lambda)-i\omega)Q_{11}-(A(\lambda)+i\omega)Q_{22} + (A(\lambda)^2-B(\lambda)^2+\omega^2)Q_{12}\Big);  \\
    S_\mathbf{x}^{21}(\omega)&=\overline{S_\mathbf{x}^{12}(\omega)} ;\\
    S_\mathbf{x}^{22}(\omega)
    &=\frac{|\omega|^{1-2H}}{2} \left(\frac{1}{(\omega+B(\lambda))^2+A(\lambda)^2}+\frac{1}{(\omega-B(\lambda))^2+A(\lambda)^2}\right) \\
    &~~~\times \Big(B(\lambda)^2Q_{11} + (A(\lambda)^2+\omega^2)Q_{22} - 2A(\lambda)B(\lambda) Q_{12}\Big).
\end{align*}
All four terms show peaks near $\pm B(\lambda)$. We start by analyzing $S_\mathbf{x}^{11}(B(\lambda))$ which is given by 
\begin{align*}
    S_\mathbf{x}^{11}(B(\lambda))
    &=\frac{|B(\lambda)|^{1-2H}}{2} \left(\frac{1}{(2 B(\lambda))^2+A(\lambda)^2}+\frac{1}{A(\lambda)^2}\right) \\
    &~~~\times\Big(A(\lambda)^2Q_{11} - 2A(\lambda)B(\lambda) Q_{12} + B(\lambda)^2(Q_{11}+Q_{22}) \Big). 
\end{align*}
Because $Q_{11}+Q_{22}>0$ and $B(\lambda)\neq0$ for all $\lambda$, the claim follows for $S_\mathbf{x}^{11}(B(\lambda))$. By the same argument the claim follows for $-B(\lambda)$ and $S_\mathbf{x}^{22}(\pm B(\lambda))$. As $S_\mathbf{x}^{12}(\pm B(\lambda))$ is complex we consider its absolute value. It holds
\begin{align*}
    |S_\mathbf{x}^{12}(\pm B(\lambda))|
    &= \frac{|B(\lambda)|^{1-2H}}{2} \left(\frac{1}{4B(\lambda)^2+A(\lambda)^2}+\frac{1}{A(\lambda)^2}\right) \\
    &~~~\times\sqrt{
    A(\lambda)^2(Q_{11}-Q_{22}+Q_{12})^2+B(\lambda)^2(Q_{11}+Q_{22})^2}
\end{align*}
Due to $Q_{11}+Q_{22}>0$ and $B(\lambda)\neq0$ for all $\lambda$, the claim follows. \\
For the second claim, we first note the behavior is determined by the term $|\omega|^{1-2H}$. Hence, the behavior of the SD at zero depends on $H$.
\qed
\end{proof}

\begin{remark}
    In contrast to the one-dimensional case, we do not consider $S_x^{\max}$ as it is computationally more involved and less effective. Instead, we can now consider the SD at a fixed frequency for all $H\in(0,1)$, so that our analysis does not depend on memory. Memory still shapes the SD significantly at zero, just as in the one-dimensional case. As the spectral components of the drift are complex and the noise term is real, it is possible to distinguish the effects of both components. In particular, the SD still can be used to detect the memory structure of the noise. 
\end{remark}

\section{Further types of memory: red and Ornstein-Uhlenbeck noise} \label{sec:red noise}

In this section, we study fast-slow systems under a different non-Markovian forcing compared to the assumptions considered above. We focus specifically on red noise \cite{bernuzzi2026critical,Morr_Red_Noise} and Ornstein-Uhlenbeck perturbations \cite{kuehn2022warning} driven by an fBm process. In such a  setting, we extended the results from Section \ref{sec:fbm 1d} and Section \ref{sec:fbm 2d} by comparing equivalent EWSs upon the approach to a codimension-1 bifurcation. We address in particular the masking phenomenon \cite{kuehn2022warning} on a variety of observables and discuss the unmasking effect induced by spinning in a Hopf bifurcation.

\subsection{Setting}
We consider the following system
\begin{align}\label{eq:DE coloured noise}
    \begin{cases}
        \txtd \mathbf{x}_t = f(\mathbf{x},y,\eps) \txtd t + \xi_t~\txtd t  \\
        \txtd y_t = \eps g(\mathbf{x},y,\eps)~\txtd t,
    \end{cases}
\end{align}
for $t\geq0$ and equivalent assumptions as in Section \ref{sec: FS}. Similarly to Section \ref{sec:preliminaries}, we can derive the fast linearized subsystem
\begin{align}\label{eq:general noise fast subsystem}
    \txtd \mathbf{x}_t = [M(\lambda) \mathbf{x}_t + \xi_t]~\txtd t,
\end{align}
where $M(\lambda):=\partial_\mathbf{x} f(h_0^a(\lambda),\lambda,0)\in\R^{n\times n}$ that satisfies the conditions in Subsection \ref{sec: S1 and S2}, i.e. is negative for $\lambda<\lambda^*$ and at least one of its eigenvalues enters in the imaginary axis for $\lambda=\lambda^*$. In particular, we indicate its eigenvalues by $\left\{A_j(\lambda)+i B_j(\lambda)\right\}_{j\in\{1,\dots,n\}}$ and the corresponding eigenvectors as $\left\{\mathbf{e}_j(\lambda)\right\}_{j\in\{1,\dots,n\}}$. Moreover, the eigenvalues of its transpose are $\left\{A_j(\lambda)-i B_j(\lambda)\right\}_{j\in\{1,\dots,n\}}$ and the eigenvectors are $\left\{\mathbf{e}_j^*(\lambda)\right\}_{j\in\{1,\dots,n\}}$. Following the same notation as Section \ref{sec: S1 and S2}, the pair of families of vectors $\left\{\left\{\mathbf{e}_j(\lambda)\right\},\left\{\mathbf{e}_j^*(\lambda)\right\}\right\}_{j\in\{1,\dots,n\}}$ is a biorthogonal system for any $\lambda\leq\lambda^*$, i.e. $\left\langle \mathbf{e}_{j_1}(\lambda),\mathbf{e}_{j_2}^*(\lambda) \right\rangle = \delta_{j_1,j_2}$ with $j_1,j_2\in\{1,\dots,n\}$. Among various types of forcing $\left(\xi_t\right)_{t\geq 0}$, we distinguish between two types which are often labeled in the literature as red noise \cite{bernuzzi2026critical,kuehn2022warning,Morr_Red_Noise} and describe the associated linearized fast subsystems of the form \eqref{eq:general noise fast subsystem}. First, we define the fractional Ornstein-Uhlenbeck process
\begin{align} \label{eq: ou forcing}
    \text{d} Z_t = -\mu Z_t \text{d}t+ \text{d} W_t^H,
\end{align}
where the drift parameter $\mu>0$ is not in $\left\{A_j(\lambda)+i B_j(\lambda)\right\}_{j\in\{1,\dots,n\}}$ for any $\lambda$ close to $\lambda^*$ and $\left( W_t^H \right)_{t\geq 0}$ is a scalar fBm with Hurst index $H\in(0,1)$. Under the assumption $\xi_t= \bm{\sigma} Z_t$ for any $t\geq 0$ and non-null $\bm{\sigma}\in\R^n$, we indicate with $\left( \mathbf{X}_t \right)_{t\geq0} = \left( \mathbf{x}_t \right)_{t\geq0}$ the solution of
\begin{align} \label{eq: real red system}
    \text{d} \mathbf{X}_t = (M(\lambda) \mathbf{X}_t + \bm{\sigma} Z_t) \text{d}t
\end{align}
with initial condition in $\R^n$. For $\xi_t= \bm{\sigma} \frac{\text{d} Z_t}{\text{d}t}$, or more rigorously $\xi_t \text{d}t= \bm{\sigma} \text{d} Z_t$, we label as $\left( \mathbf{Y}_t \right)_{t\geq0} = \left( \mathbf{x}_t \right)_{t\geq0}$ the solution of
\begin{align} \label{eq: false red system}
    \begin{split}
        \text{d} \mathbf{Y}_t &= M(\lambda) \mathbf{Y}_t \text{d}t + \bm{\sigma} \text{d} Z_t \\
        &= M(\lambda) \mathbf{Y}_t \text{d}t - \mu \bm{\sigma} Z_t + \bm{\sigma} \text{d} W_t^H
    \end{split}
\end{align}
for an initial condition in $\R^n$. If $H=\frac{1}{2}$, $(Z_t)_{t\geq 0}$ is the standard Ornstein-Uhlenbeck process.

\subsection{Autocovariance and autocorrelation}
In this subsection, we discuss the scaling law of the time-asymptotic autocovariance of the solutions of \eqref{eq: real red system} and \eqref{eq: false red system}. This is achieved by studying extended systems that are solved by $\left( \mathbf{X}_t \right)_{t\geq0}$, or $\left( \mathbf{Y}_t \right)_{t\geq0}$ respectively, coupled to the noise trajectory $\left( Z_t \right)_{t\geq0}$. Then, we describe a heuristic approach for the study of such an observable following the steps in Lemma \ref{lem:autocov fast system complex lemma}, Theorem \ref{thm:autocov fast system} and Theorem \ref{thm:autocov fast system complex}. Specifically, we distinguish the time-asymptotic modal autocovariance indicators along with the masking, memory and mitigation terms within. We then observe their behavior in the proximity to the critical limit. The extended perspective introduces an additional spectral mode that can influence the dynamics depending on the bifurcation approached.

The extended systems are
\begin{align} \label{eq: bold_x}
    \text{d} \begin{pmatrix}
        \mathbf{X}_t \\ Z_t
    \end{pmatrix} &= \tilde{M}_\mathbf{X}(\lambda) \text{d}t + \begin{pmatrix}
        \mathbf{0}_{n\times 1} \\ 1 
    \end{pmatrix} \text{d} W_t^H
\end{align}
and
\begin{align} \label{eq: bold_y}
    \text{d} \begin{pmatrix}
        \mathbf{Y}_t \\ Z_t
    \end{pmatrix} &= \tilde{M}_\mathbf{Y}(\lambda) \text{d}t + \begin{pmatrix}
        \bm{\sigma} \\ 1
    \end{pmatrix} \text{d} W_t^H,
\end{align}
respectively, for $\mathbf{0}_{m_1\times m_2}$ the null matrix in $\R^{m_1\times m_2}$ with $m_1, m_2 \in \N_{>0}$. The drift matrices take the form
\begin{align*}
    \tilde{M}_\mathbf{X}(\lambda)=\begin{pmatrix}
         M(\lambda) & \bm{\sigma} \\ \mathbf{0}_{1\times n} & -\mu
    \end{pmatrix}
    \quad \text{and} \quad
    \tilde{M}_\mathbf{Y}(\lambda)=\begin{pmatrix}
         M(\lambda) & -\mu \bm{\sigma} \\ \mathbf{0}_{1 \times n} & -\mu
    \end{pmatrix},
\end{align*}
for $\lambda< \lambda^*$. We note that these matrices share eigenvalues, which are $\left\{A_j(\lambda)+i B_j(\lambda)\right\}_{j\in\{1,\dots,n\}}$ for $j\in\{1,\dots,n\}$ and $-\mu$. Their first $n$ eigenvectors are
\begin{align*}
    \left\{ \mathbf{e}_{j,\mathbf{X}}(\lambda) \right\}_{j\in\{1,\dots,n\}}
    =\left\{ \mathbf{e}_{j,\mathbf{Y}}(\lambda) \right\}_{j\in\{1,\dots,n\}}
    =\left\{\begin{pmatrix}
        \mathbf{e}_j(\lambda) \\ 0
    \end{pmatrix}\right\}_{j\in\{1,\dots,n\}}
\end{align*}
and the eigenvector corresponding to the eigenvalue $-\mu$ is
\begin{align*}
    \mathbf{e}_{n+1,\mathbf{X}}(\lambda)
    =\begin{pmatrix}
        -(\mu+M(\lambda))^{-1} \bm{\sigma}  \\ 1
    \end{pmatrix}
    \quad \text{and} \quad
    \mathbf{e}_{n+1,\mathbf{Y}}(\lambda)
    =\begin{pmatrix}
        \mu (\mu+M(\lambda))^{-1} \bm{\sigma}  \\ 1
    \end{pmatrix},
\end{align*}
respectively. The eigenvectors of the corresponding transpose matrices are
\begin{align*}
    \left\{ \mathbf{e}_{j,\mathbf{X}}^*(\lambda) \right\}_{j\in\{1,\dots,n\}}=\begin{pmatrix}
        \mathbf{e}_j^*(\lambda) \\ \\ \frac{\bm{\sigma}^\text{T} \mathbf{e}_j^*(\lambda)}{A_j(\lambda)-i B_j(\lambda) + \mu}
    \end{pmatrix}
    \quad \text{and} \quad
    \left\{ \mathbf{e}_{j,\mathbf{Y}}^*(\lambda) \right\}_{j\in\{1,\dots,n\}}=\begin{pmatrix}
        \mathbf{e}_j^*(\lambda) \\ \\ -\mu \frac{\bm{\sigma}^\text{T} \mathbf{e}_j^*(\lambda)}{A_j(\lambda)-i B_j(\lambda) + \mu}
    \end{pmatrix},
\end{align*}
which are associated to the eigenvalues $\left\{A_j(\lambda)-i B_j(\lambda)\right\}_{j\in\{1,\dots,n\}}$, and
\begin{align*}
    \mathbf{e}_{n+1,\mathbf{X}}^*(\lambda)
    = \mathbf{e}_{n+1,\mathbf{Y}}^*(\lambda)
    = \begin{pmatrix}
        \mathbf{0}_{n\times 1} \\ 1
    \end{pmatrix},
\end{align*}
matching the eigenvalue $-\mu$. Moreover, from \eqref{eq:DE coloured noise} and \eqref{eq:general noise fast subsystem}, we obtain the corresponding covariance matrices:
\begin{align*}
    \tilde{Q}_\mathbf{X}= \begin{pmatrix}
        \mathbf{0}_{n\times 1} \\ 1
    \end{pmatrix} \begin{pmatrix}
        \mathbf{0}_{1\times n} & 1
    \end{pmatrix}
    = \begin{pmatrix}
        \mathbf{0}_{n\times n} & \mathbf{0}_{n\times 1} \\ \mathbf{0}_{1\times n} & 1
    \end{pmatrix}
    \quad \text{and} \quad
    \tilde{Q}_\mathbf{Y}= \begin{pmatrix}
        \bm{\sigma} \\ 1
    \end{pmatrix} \begin{pmatrix}
        \bm{\sigma}^\text{T} & 1
    \end{pmatrix}
    = \begin{pmatrix}
        \bm{\sigma} \bm{\sigma}^\text{T} & \bm{\sigma} \\ \bm{\sigma}^\text{T} & 1
    \end{pmatrix} .
\end{align*}
From the collection of these objects, we can observe the rate of divergence of the time-asymptotic autocovariance of the solution of \eqref{eq: real red system} and \eqref{eq: false red system} in the settings \nameref{S1} and \nameref{S2} along general modes $\mathbf{v}_1,\mathbf{v}_2\in\R^n$. We note that this indicator corresponds to the time-asymptotic autocovariance of the solution of  \eqref{eq:general noise fast subsystem} in the corresponding settings along $\begin{pmatrix}\mathbf{v}_1 \\ 0\end{pmatrix}, \begin{pmatrix}\mathbf{v}_2 \\ 0\end{pmatrix}\in\R^{n+1}$. Depending on the system \eqref{eq: real red system} and \eqref{eq: false red system}, we label it as
\begin{align*}
    V_{\infty,\mathbf{X}}(\tau)\left[ \mathbf{w}_1, \mathbf{w}_2 \right]:= \lim_{t\to\infty}
    \mathbb{E}\left[\left\langle \begin{pmatrix}\mathbf{X}_t \\ Z_t\end{pmatrix}, \mathbf{w}_1 \right\rangle_{n+1} \overline{\left\langle \begin{pmatrix}\mathbf{X}_{t+\tau} \\ Z_{t+\tau}\end{pmatrix}, \mathbf{w}_2 \right\rangle_{n+1}} \right]
\end{align*}
and
\begin{align*}
    V_{\infty,\mathbf{Y}}(\tau)\left[ \mathbf{w}_1, \mathbf{w}_2 \right]:= \lim_{t\to\infty}
    \mathbb{E}\left[\left\langle \begin{pmatrix}\mathbf{Y}_t \\ Z_t\end{pmatrix}, \mathbf{w}_1 \right\rangle_{n+1} \overline{\left\langle \begin{pmatrix}\mathbf{Y}_{t+\tau} \\ Z_{t+\tau}\end{pmatrix}, \mathbf{w}_2 \right\rangle_{n+1}} \right]
\end{align*}
for any $\mathbf{w}_1,\mathbf{w}_2\in\C^{n+1}$ and $\tau\geq 0$, respectively. Setting then $\mathbf{w}_1 = \begin{pmatrix}\mathbf{v}_1 \\ 0\end{pmatrix}$ and $\mathbf{w}_2 = \begin{pmatrix}\mathbf{v}_2 \\ 0\end{pmatrix}$, we obtain in particular that
\begin{align*}
    V_{\infty,\mathbf{X}}(\tau)\left[ \begin{pmatrix}\mathbf{v}_1 \\ 0\end{pmatrix}, \begin{pmatrix}\mathbf{v}_2 \\ 0\end{pmatrix} \right]= \lim_{t\to\infty}
    \mathbb{E}\left[\left\langle \mathbf{X}_t, \mathbf{v}_1 \right\rangle \overline{\left\langle \mathbf{X}_{t+\tau}, \mathbf{v}_2 \right\rangle} \right]
\end{align*}
and
\begin{align*}
    V_{\infty,\mathbf{Y}}(\tau)\left[ \begin{pmatrix}\mathbf{v}_1 \\ 0\end{pmatrix}, \begin{pmatrix}\mathbf{v}_2 \\ 0\end{pmatrix} \right]= \lim_{t\to\infty}
    \mathbb{E}\left[\left\langle \mathbf{Y}_t, \mathbf{v}_1 \right\rangle \overline{\left\langle \mathbf{Y}_{t+\tau}, \mathbf{v}_2 \right\rangle} \right] .
\end{align*}
In the definitions above we use the dissipativity property of the systems for $\lambda<\lambda^*$ which implies that 
    $$\lim_{t\to\infty}
    \mathbb{E}\left[\left\langle \mathbf{X}_t, \mathbf{v}_1 \right\rangle\right]
    =\lim_{t\to\infty}
    \mathbb{E}\left[\left\langle \mathbf{Y}_t, \mathbf{v}_1 \right\rangle\right]
    = 0$$ 
    and simplifies the formulas. For $\mathbf{w}\not\in\text{Ker}\left(\tilde{Q}_\mathbf{X}\right)\subset \C^{n+1}$, we define the time-asymptotic autocorrelation function for \eqref{eq: bold_x} along $\mathbf{w}$ as
\begin{align*}
    AC_{\infty,\mathbf{X}}(\tau)[\mathbf{w}]:=\frac{V_{\infty,\mathbf{X}}(\tau)[\mathbf{w},\mathbf{w}]}{V_{\infty,\mathbf{X}}(0)[\mathbf{w},\mathbf{w}]},
\end{align*}
for any $\tau\geq 0$. Similarly, for $\mathbf{w}\not\in\text{Ker}\left(\tilde{Q}_\mathbf{Y}\right)\subset \C^{n+1}$ we define the time-asymptotic autocorrelation function associated with \eqref{eq: bold_y} along $\mathbf{w}$ as
\begin{align*}
    AC_{\infty,\mathbf{Y}}(\tau)[\mathbf{w}]:=\frac{V_{\infty,\mathbf{Y}}(\tau)[\mathbf{w},\mathbf{w}]}{V_{\infty,\mathbf{Y}}(0)[\mathbf{w},\mathbf{w}]},
\end{align*}
for $\tau\geq 0$. We note then that
\begin{align} \label{eq: breaking_x}
    \begin{pmatrix}\mathbf{v}_m \\ 0\end{pmatrix} =& \sum_{j\in\{1,\dots,n+1\}} \left\langle \begin{pmatrix}\mathbf{v}_m \\ 0\end{pmatrix}, \mathbf{e}_{j,\mathbf{X}}(\lambda) \right\rangle_{n+1} \mathbf{e}_{j,\mathbf{X}}^*(\lambda)\\
     =& \sum_{j\in\{1,\dots,n\}} \left\langle \mathbf{v}_m , \mathbf{e}_j(\lambda) \right\rangle \mathbf{e}_{j,\mathbf{X}}^*(\lambda) - \left\langle \mathbf{v}_m , (\mu+M(\lambda))^{-1} \bm{\sigma} \right\rangle \mathbf{e}_{j,\mathbf{X}}^*(\lambda) \nonumber
\end{align}
and that
\begin{align} \label{eq: breaking_y}
    \begin{pmatrix}\mathbf{v}_m \\ 0\end{pmatrix} =& \sum_{j\in\{1,\dots,n+1\}} \left\langle \begin{pmatrix}\mathbf{v}_m \\ 0\end{pmatrix}, \mathbf{e}_{j,\mathbf{Y}}(\lambda) \right\rangle_{n+1} \mathbf{e}_{j,\mathbf{Y}}^*(\lambda)\\
     =& \sum_{j\in\{1,\dots,n\}} \left\langle \mathbf{v}_m , \mathbf{e}_j(\lambda) \right\rangle \mathbf{e}_{j,\mathbf{Y}}^*(\lambda) + \mu \left\langle \mathbf{v}_m , (\mu+M(\lambda))^{-1} \bm{\sigma} \right\rangle \mathbf{e}_{j,\mathbf{Y}}^*(\lambda) \nonumber
\end{align}
for $m\in\{1,2\}$ and any $\lambda\leq \lambda^*$. Considering such modes, with a null proxy direction in the forcing $Z$, corresponds to the lack of data of a minor component or forcing, which is common in applications \cite{bernuzzi2024warning,liu2025influence}. However, vectors of this form may not be orthogonal to $\mathbf{e}_{n+1,\mathbf{X}}(\lambda^*)$ or $\mathbf{e}_{n+1,\mathbf{Y}}(\lambda^*)$. It follows that perturbations driven by fBm on $Z$ do not only affect the indicator through the forcing of the $\mathbf{X}$, or $\mathbf{Y}$, but are rather explicit components of the time-asymptotic autocovariance itself. In fact, we can follow similar methods to the proof of Theorem \ref{thm:autocov fast system} and Theorem \ref{thm:autocov fast system complex} to obtain
\begin{align*}
    &V_{\infty,\mathbf{X}}(\tau)\left[ \begin{pmatrix}\mathbf{v}_1 \\ 0\end{pmatrix}, \begin{pmatrix}\mathbf{v}_2 \\ 0\end{pmatrix} \right] \\
    =& \sum_{j_1,j_2\in\{1,\dots,n+1\}} \overline{\left\langle \begin{pmatrix}\mathbf{v}_1 \\ 0\end{pmatrix}, \mathbf{e}_{j_1,\mathbf{X}}(\lambda) \right\rangle_{n+1}} \left\langle \begin{pmatrix}\mathbf{v}_2 \\ 0\end{pmatrix}, \mathbf{e}_{j_2,\mathbf{X}}(\lambda) \right\rangle_{n+1} V_\infty(\tau)\left[ \mathbf{e}_{j_1,\mathbf{X}}^*(\lambda), \mathbf{e}_{j_2,\mathbf{X}}^*(\lambda) \right].
\end{align*}
As such, the rate of divergence is dictated by the scaling law along spectral modes. Specifically, it is driven by the corresponding masking, memory and mitigation terms, similarly to those introduced in \eqref{eq:autocov fast subsystem complex 2}. A similar approach can be followed for system \eqref{eq: bold_y}. We then define
\begin{align*}
    \rho_{j,\mathbf{X}}(\lambda) = \frac{ \bm{\sigma}^\text{T} \mathbf{e}_j^*(\lambda) }{A_j(\lambda)-i B_j(\lambda) + \mu}, \quad \text{for} \quad j\in\{1,\dots,n\}, \quad\text{and} \quad \rho_{n+1,\mathbf{X}}(\lambda) = 1
\end{align*}
along with
\begin{align*}
    \rho_{j,\mathbf{Y}}(\lambda) = \bm{\sigma}^\text{T} \mathbf{e}_j^*(\lambda) \left( 1 - \frac{\mu}{A_j(\lambda)-i B_j(\lambda)+\mu} \right), \quad \text{for} \quad j\in\{1,\dots,n\}, \quad\text{and} \quad \rho_{n+1,\mathbf{Y}}(\lambda) = 1,
\end{align*}
for $\lambda\leq\lambda^*$. For any $j_1,j_2\in\{1,\dots,n\}$, we find that the components of the time-asymptotic modal variance are as follows:
\begin{itemize}
    \item The masking terms for \eqref{eq: bold_x} and \eqref{eq: bold_y} are
    \begin{align*}
        \left\langle \mathbf{e}_{j_1,\mathbf{X}}^*(\lambda) , \tilde{Q}_\mathbf{X} \mathbf{e}_{j_2,\mathbf{X}}^*(\lambda) \right\rangle_{n+1} = \rho_{j_1,\mathbf{X}}(\lambda) \overline{\rho_{j_2,\mathbf{X}}(\lambda)}
    \end{align*}
    and
    \begin{align*}
        \left\langle \mathbf{e}_{j_1,\mathbf{Y}}^*(\lambda) , \tilde{Q}_\mathbf{Y} \mathbf{e}_{j_2,\mathbf{Y}}^*(\lambda) \right\rangle_{n+1} = \rho_{j_1,\mathbf{Y}}(\lambda) \overline{\rho_{j_2,\mathbf{Y}}(\lambda)}
    \end{align*}
    for any $j_1,j_2\in\{1,\dots,n\}$ and $\lambda\leq \lambda^*$. We find that their scaling laws depend on the system and setting addressed.
    \item The memory items are $P(A_{j_1}(\lambda)+i B_{j_1}(\lambda),A_{j_2}(\lambda)+i B_{j_2}(\lambda),H,\tau)$ and depend solely on the setting considered.
    \item The mitigation terms $\left(-(A_{j_1}(\lambda)+A_{j_2}(\lambda))+i (B_{j_2}(\lambda)-B_{j_1}(\lambda))\right)^{-1}$ diverge only on the sensible modes, i.e. $j_1=j_2=1$ in the setting \nameref{S1} and $j_1=j_2\in\{1,2\}$ in the setting \nameref{S2}.
\end{itemize}
We can then assemble our results to describe the rate of divergence of $V_{\infty,\mathbf{X}}(\tau)\left[ \begin{pmatrix}\mathbf{v}_1 \\ 0\end{pmatrix}, \begin{pmatrix}\mathbf{v}_2 \\ 0\end{pmatrix} \right]$ and $V_{\infty,\mathbf{Y}}(\tau)\left[ \begin{pmatrix}\mathbf{v}_1 \\ 0\end{pmatrix}, \begin{pmatrix}\mathbf{v}_2 \\ 0\end{pmatrix} \right]$ for almost every $\tau\geq0$. We first consider the solution of \eqref{eq: bold_x} in \nameref{S1}. The scaling laws of $V_{\infty,\mathbf{X}}(\tau)\left[ \mathbf{e}_{j_1,\mathbf{X}}^*(\lambda), \mathbf{e}_{j_2,\mathbf{X}}^*(\lambda) \right]$, for $j_1,j_2\in\{1,2\}$, and their components are shown in Table \ref{tab:X1}. Since $\bm{\sigma}$ and $\mathbf{e}_1^*(\lambda^*)$ are non-zero scalars, it follows that the orders in the first column of the table are exact for almost every $\tau$. As such, the time-asymptotic modal autocovariance in the first column is the only leading term in $V_{\infty,\mathbf{X}}(\tau)\left[ \begin{pmatrix}\mathbf{v}_1 \\ 0\end{pmatrix}, \begin{pmatrix}\mathbf{v}_2 \\ 0\end{pmatrix} \right]$. As shown in \eqref{eq: breaking_x}, the two indicators share a scaling law
\begin{align*}
    V_{\infty,\mathbf{X}}(\tau)\left[ \begin{pmatrix}\mathbf{v}_1 \\ 0\end{pmatrix}, \begin{pmatrix}\mathbf{v}_2 \\ 0\end{pmatrix} \right]
     \asymp
     V_{\infty,\mathbf{X}}(\tau)\left[ \mathbf{e}_{1,\mathbf{X}}^*(\lambda), \mathbf{e}_{1,\mathbf{X}}^*(\lambda) \right]
     \asymp
     |A_1(\lambda)|^{-2H}
\end{align*}
if
\begin{align} \label{eq: red_tool}
    \left\langle \begin{pmatrix}\mathbf{v}_m \\ 0\end{pmatrix}, \mathbf{e}_{1,\mathbf{X}}(\lambda) \right\rangle_{n+1} 
    = \left\langle \mathbf{v}_m, \mathbf{e}_1(\lambda) \right\rangle \neq 0
\end{align}
holds for every $m\in\{1,2\}$. Such a rate is equivalent to that reported in Theorem \ref{thm:autocov fast system} and in \eqref{eq:Autovar}. Since a unique time-asymptotic modal autocovariance assumes a leading scaling law, it follows from equivalent steps to Corollary \ref{cor:autocor fast system} and from \eqref{eq:I_ran_out_of_names} that
\begin{align*}
    AC_{\infty,\mathbf{X}}(\tau)\left[ \begin{pmatrix}\mathbf{v}_m \\ 0\end{pmatrix} \right]= 1+\mathcal{O}\left(|A_1(\lambda)|^{\text{min}\{1,2H\}}\right)
\end{align*}
for any $\tau>0$ if \eqref{eq: red_tool} holds.
\begin{center}
    \begin{tabular}{|c||c|c|c|}
        \hline
        \centering Terms in & \multicolumn{3}{c|}{$(j_1,j_2)$}  \\
        \cline{2-4}
        $V_{\infty,\mathbf{X}}(\tau)\left[ \mathbf{e}_{j_1,\mathbf{X}}^*(\lambda), \mathbf{e}_{j_2,\mathbf{X}}^*(\lambda) \right]$ & \multirow{1}{*}{$(1,1)$} & \multirow{1}{*}{$(1,2)$} & \multirow{1}{*}{$(2,2)$} \\
        \hline
        Masking & $0$ & $0$ & $0$ \\
        \hline
        Memory & $1-2H$ & $1-2H$ & $0$ \\
        \hline
        Mitigation & $-1$ & $0$ & $0$ \\
        \hline
        Modal & $-2H$ & $1-2H$ & $0$ \\
        \hline
    \end{tabular}
    \captionof{table}{Exponents $\nu$ in $\mathcal{O}\left(|A(\lambda)|^\nu\right)$ corresponding to the terms introduced in \eqref{eq:autocov fast subsystem complex 2} within the time-asymptotic modal autocovariance for \eqref{eq: bold_x} and setting \nameref{S1}. Since $\mathbf{e}_1^*(\lambda^*)$ is scalar, the scaling laws described in the first column are exact for almost every $\tau\geq0$. Consequently, no masking phenomenon occurs along $\mathbf{e}_1(\lambda)$ and the rate of divergence of the time-asymptotic autocovariance along general modes in asymptotically comparable to $|A(\lambda)|^{-2H}$ for almost every $\tau\geq 0$.}
    \label{tab:X1}
\end{center}

We then address the solution of \eqref{eq: bold_y} in the \nameref{S1} setting. The rates of the terms within $V_{\infty,\mathbf{Y}}(\tau)\left[ \mathbf{e}_{j_1,\mathbf{Y}}^*(\lambda), \mathbf{e}_{j_2,\mathbf{Y}}^*(\lambda) \right]$, for $j_1,j_2\in\{1,2\}$, are collected in Table \ref{tab:Y1}. Regardless of the value assumed by $\left\langle \bm{\sigma}, \mathbf{e}_1^*(\lambda^*)\right\rangle$ we note that $\rho_{1,\mathbf{Y}}(\lambda^*)=0$, since $B_1(\lambda)=0$ for any $\lambda\leq \lambda^*$. In contrast, $\rho_{2,\mathbf{Y}}(\lambda^*)\neq 0$ if and only if $\left\langle \bm{\sigma}, \mathbf{e}_2^*(\lambda^*)\right\rangle\neq 0$. It follows that $\left\langle \mathbf{e}_{j_1,\mathbf{Y}}^*(\lambda^*) , \tilde{Q}_\mathbf{Y} \mathbf{e}_{j_2,\mathbf{Y}}^*(\lambda^*) \right\rangle_{n+1}=0$ if $(j_1,j_2)\neq(2,2)$. Equivalently $\mathbf{e}_{1,\mathbf{Y}}^*(\lambda)$ enters in $\text{Ker}(\tilde{Q}_\mathbf{Y})$ in the limit $\lambda\to \lambda^*$, which induces a masking effect in the EWS. Along with the rates induced by the memory and mitigation terms, this results in the fact that 
\begin{align*}
    V_{\infty,\mathbf{Y}}(\tau)\left[ \mathbf{e}_{j_1,\mathbf{Y}}^*(\lambda), \mathbf{e}_{j_2,\mathbf{Y}}^*(\lambda) \right] \asymp \left| A_1(\lambda) \right|^{2-2H} \to 0
\end{align*}
in the limit $\lambda\to \lambda^*$ for $(j_1,j_2)\neq(2,2)$. Moreover, $V_{\infty,\mathbf{Y}}(\tau)\left[ \begin{pmatrix}\mathbf{v}_1 \\ 0\end{pmatrix}, \begin{pmatrix}\mathbf{v}_2 \\ 0\end{pmatrix} \right]$ is not null in the critical limit only if $\left\langle \bm{\sigma}, \mathbf{e}_2^*(\lambda^*)\right\rangle\neq 0$ and, by \eqref{eq: breaking_y}, if $\left\langle \mathbf{v}_m, \mathbf{e}_2(\lambda) \right\rangle \neq 0$ for every $m\in\{1,2\}$. Consequently, the masking effect is associated with the convergence of the indicator across general modes in the critical limit, which in turn leads to a reduced rate of increase as the system approaches the critical threshold. In such a setting, it follows directly from \eqref{eq:I_ran_out_of_names} and
\begin{align*}
    \frac{V_{\infty,\mathbf{Y}}(\tau)\left[ \mathbf{e}_{2,\mathbf{Y}}^*(\lambda), \mathbf{e}_{2,\mathbf{Y}}^*(\lambda) \right]}{V_{\infty,\mathbf{Y}}(0)\left[ \mathbf{e}_{2,\mathbf{Y}}^*(\lambda), \mathbf{e}_{2,\mathbf{Y}}^*(\lambda) \right]}
    = \frac{P(-\mu,-\mu,H,\tau)}{P(-\mu,-\mu,H,0)} ,
\end{align*}
for all $\tau\geq 0$ and $\lambda \leq \lambda^*$, that
\begin{align} \label{eq: plot twist}
    AC_{\infty,\mathbf{Y}}(\tau)\left[ \begin{pmatrix}\mathbf{v}_m \\ 0\end{pmatrix} \right]= \frac{P(-\mu,-\mu,H,\tau)}{P(-\mu,-\mu,H,0)} +\mathcal{O}\left(|A_1(\lambda)|^{2-2H}\right)
\end{align}
for any $\tau>0$. The rate is given by the fact that all time-asymptotic modal autocovariances converge and only one does not approach zero.

\begin{center}
    \begin{tabular}{|c||c|c|c|}
        \hline
        \centering Terms in & \multicolumn{3}{c|}{$(j_1,j_2)$}  \\
        \cline{2-4}
        $V_{\infty,\mathbf{Y}}(\tau)\left[ \mathbf{e}_{j_1,\mathbf{Y}}^*(\lambda), \mathbf{e}_{j_2,\mathbf{Y}}^*(\lambda) \right]$ & \multirow{1}{*}{$(1,1)$} & \multirow{1}{*}{$(1,2)$} & \multirow{1}{*}{$(2,2)$} \\
        \hline
        Masking & $2$ & $1$ & $0$ \\
        \hline
        Memory & $1-2H$ & $1-2H$ & $0$ \\
        \hline
        Mitigation & $-1$ & $0$ & $0$ \\
        \hline
        Modal & $2-2H$ & $2-2H$ & $0$ \\
        \hline
    \end{tabular}
    \captionof{table}{Exponents $\nu$ in $\mathcal{O}\left(|A(\lambda)|^\nu\right)$ associated with the items in \eqref{eq:autocov fast subsystem complex 2} within the time-asymptotic modal autocovariance for \eqref{eq: bold_y} in setting \nameref{S1}. The time-asymptotic modal autocovariance appears to be masked along critical and mixed modes since the corresponding exponents are positive. It follows that such terms vanish in the critical limit and $V_{\infty,\mathbf{Y}}(\tau)\left[ \mathbf{v}_1, \mathbf{v}_2 \right]$ converges in $\lambda\to\lambda^*$ for any $\tau\geq 0$ and $\mathbf{v}_1, \mathbf{v}_2 \in \R^2$ as a result.}
    \label{tab:Y1}
\end{center}

As we study the solution of \eqref{eq: bold_x} in the \nameref{S2} setting and $V_{\infty,\mathbf{X}}(\tau)\left[ \begin{pmatrix}\mathbf{v}_1 \\ 0\end{pmatrix}, \begin{pmatrix}\mathbf{v}_2 \\ 0\end{pmatrix} \right]$, we note that the masking terms $\left\langle \mathbf{e}_{j_1,\mathbf{X}}^*(\lambda) , \tilde{Q}_\mathbf{X} \mathbf{e}_{j_2,\mathbf{X}}^*(\lambda) \right\rangle_{n+1}$ are well defined and not null in $\lambda=\lambda^*$ if $\left\langle \bm{\sigma}, \mathbf{e}_{j_1}^*(\lambda^*)\right\rangle\neq 0$ and $\left\langle \bm{\sigma}, \mathbf{e}_{j_2}^*(\lambda^*)\right\rangle\neq 0$. Furthermore, as described in Lemma \ref{lem:autocov fast system complex lemma} and Theorem \ref{thm:autocov fast system complex}, the memory term is convergent in the critical limit for almost every $\tau\geq 0$ and any $j_1,j_2\in\{1,2,3\}$. In conclusion, the only diverging term is the mitigation term for $j_1=j_2\in\{1,2\}$. The scaling laws of these items are shown in Table \ref{tab:X2}. Upon fixing $H\in(0,1)$ and assuming $\left\langle \bm{\sigma}, \mathbf{e}_1^*(\lambda^*)\right\rangle\neq 0$, the scaling law of $\left|V_{\infty,\mathbf{X}}(\tau)\left[ \begin{pmatrix}\mathbf{v}_1 \\ 0\end{pmatrix}, \begin{pmatrix}\mathbf{v}_2 \\ 0\end{pmatrix} \right]\right|$ is then asymptotically comparable to $|A_1(\lambda)|^{-1}$ for $\tau\geq 0$ and $\mathbf{v}_1$ and $\mathbf{v}_2$ in a dense subset of $\R^2$. Such a statement can be proven equivalently to Theorem \ref{thm:autocov fast system complex}. We finally notice that since only two time-asymptotic modal autocovariance indicators diverge in Table \ref{tab:X2}, we can follow a similar approach to Corollary \ref{cor:autocor fast system complex}. It results then that we can find $q\in\C$ such that
\begin{align*}
    AC_{\infty,\mathbf{X}}(\tau)\left[ \begin{pmatrix}\mathbf{v}_m \\ 0\end{pmatrix} \right]
    =& q + \mathcal{O}\left(-A_1(\lambda)\right) + \cO\left(\left|\left| \mathbf{e}_{1,\mathbf{X}}^*(\lambda) - \mathbf{e}_{1,\mathbf{X}}^*(\lambda^*) \right|\right|_{n+1} \right)\\
    &+ \cO\left(\left|\left| \mathbf{e}_{2,\mathbf{X}}^*(\lambda) - \mathbf{e}_{2,\mathbf{X}}^*(\lambda^*) \right|\right|_{n+1} \right) + \cO(|B_1(\lambda)-B_1(\lambda^*)|)
\end{align*}
for $m\in\{1,2\}$, almost all $\tau>0$ and $\mathbf{v}_1$ and $\mathbf{v}_2$ in dense subsets of $\R^2$.

\begin{center}
    \begin{tabular}{|c||c|c|c|c|c|c|}
        \hline
        \centering Terms in & \multicolumn{6}{c|}{$(j_1,j_2)$}  \\
        \cline{2-7}
        $V_{\infty,\mathbf{X}}(\tau)\left[ \mathbf{e}_{j_1,\mathbf{X}}^*(\lambda), \mathbf{e}_{j_2,\mathbf{X}}^*(\lambda) \right]$ & \multirow{1}{*}{$(1,1)$} & \multirow{1}{*}{$(1,2)$} & \multirow{1}{*}{$(1,3)$} & \multirow{1}{*}{$(2,2)$} & \multirow{1}{*}{$(2,3)$} & \multirow{1}{*}{$(3,3)$} \\
        \hline
        Masking & $0$ & $0$ & $0$ & $0$ & $0$ & $0$ \\
        \hline
        Memory & $0$ & $0$ & $0$ & $0$ & $0$ & $0$ \\
        \hline
        Mitigation & $-1$ & $0$ & $0$ & $-1$ & $0$ & $0$ \\
        \hline
        Modal & $-1$ & $0$ & $0$ & $-1$ & $0$ & $0$ \\
        \hline
    \end{tabular}
    \captionof{table}{Exponents $\nu$ in $\mathcal{O}\left(|A(\lambda)|^\nu\right)$ corresponding to the items in \eqref{eq:autocov fast subsystem complex 2} within the time-asymptotic modal autocovariance for \eqref{eq: bold_x}, or equivalently for \eqref{eq: bold_y}, and \nameref{S2}. The fact that the critical eigenvalues do not reach the imaginary axis at $0$ in $\lambda\to \lambda^*$ avoids the masking effect in \eqref{eq: bold_y}. In particular, for $\left\langle \bm{\sigma}, \mathbf{e}_1^*(\lambda^*)\right\rangle\neq 0$, it follows that $\rho_{1,\mathbf{X}}\neq 0$ and $\rho_{1,\mathbf{Y}}\neq 0$. The non-uniqueness of the critical modes implies that the leading time-asymptotic modal autocovariances could in principle assume a higher scaling law than equivalent indicators along general modes. Following the same steps as in the proof of Theorem \ref{thm:autocov fast system complex}, such an observable assumes however the rate $\mathcal{O}\left(|A(\lambda)|^{-1}\right)$ for almost every mode and $\tau\geq0$.}
    \label{tab:X2}
\end{center}

The scaling law of $\left|V_{\infty,\mathbf{Y}}(\tau)\left[ \begin{pmatrix}\mathbf{v}_1 \\ 0\end{pmatrix}, \begin{pmatrix}\mathbf{v}_2 \\ 0\end{pmatrix} \right]\right|$ associated to the solution of \eqref{eq: bold_y} in the setting \nameref{S2} is equivalent to the one shown above. In fact, the terms within its time-asymptotic modal autocovariance indicators assume rates described in Table \ref{tab:X2}. This is given by the fact that in the setting \nameref{S2} the masking terms $\left\langle \mathbf{e}_{j_1,\mathbf{Y}}^*(\lambda) , \tilde{Q}_\mathbf{Y} \mathbf{e}_{j_2,\mathbf{Y}}^*(\lambda) \right\rangle_{n+1}$ are well defined in $\lambda=\lambda^*$. Moreover, similarly to the previous case, they are not null if $\left\langle \bm{\sigma}, \mathbf{e}_{j_1}^*(\lambda^*)\right\rangle\neq 0$ and $\left\langle \bm{\sigma}, \mathbf{e}_{j_2}^*(\lambda^*)\right\rangle\neq 0$. The memory and mitigation terms depend solely on the eigenvalues, $H$ and $\tau$. As such, they assume identical values as the ones described in \eqref{eq: bold_x} in the setting \nameref{S2}. Finally, since the critical mode is not unique, the modes $\mathbf{v}_1,\mathbf{v}_2\in\R^2$ and the lag times $\tau\geq 0$ for which $\left|V_{\infty,\mathbf{Y}}(\tau)\left[ \begin{pmatrix}\mathbf{v}_1 \\ 0\end{pmatrix}, \begin{pmatrix}\mathbf{v}_2 \\ 0\end{pmatrix} \right]\right|$ assumes a rate of order $|A_1(\lambda)|^{-1}$, can be found following the steps of the proof of Theorem \ref{thm:autocov fast system complex} for any $H\in(0,1)$. Similarly to the previous case, we can find $q\in\C$ such that
\begin{align*}
    AC_{\infty,\mathbf{Y}}(\tau)\left[ \begin{pmatrix}\mathbf{v}_m \\ 0\end{pmatrix} \right]
    =& q + \mathcal{O}\left(-A_1(\lambda)\right) + \cO\left(\left|\left| \mathbf{e}_{1,\mathbf{Y}}^*(\lambda) - \mathbf{e}_{1,\mathbf{Y}}^*(\lambda^*) \right|\right|_{n+1} \right)\\
    &+ \cO\left(\left|\left| \mathbf{e}_{2,\mathbf{Y}}^*(\lambda) - \mathbf{e}_{2,\mathbf{Y}}^*(\lambda^*) \right|\right|_{n+1} \right) + \cO(|B_1(\lambda)-B_1(\lambda^*)|)
\end{align*}
holds for $m\in\{1,2\}$, almost all $\tau>0$ and $\mathbf{v}_1$ and $\mathbf{v}_2$ in a dense subset of $\R^2$.

\subsection{Spectral density}
We study the spectral density for the red noise system \eqref{eq: real red system} and the fOU noise system \eqref{eq: false red system}. We start by deriving the SD and afterwards discuss what happens when we approach a bifurcation.
\begin{lemma}\label{lem:PSD for real red system}
    Let $(Z_t)_{t\geq 0}$ be a stationary fOU process \eqref{eq: ou forcing}. Then, the stationary solution of \eqref{eq: real red system} has the spectral density
    $$S_\mathbf{X}=C_H\frac{|\omega|^{1-2H}}{\mu^2+\omega^2}(M(\lambda)-i\omega I)^{-1} \bm\sigma\bm\sigma^*(M(\lambda)^*+i\omega I)^{-1}.$$
\end{lemma}
\begin{proof}
    The SD for $(Z_t)_{t\geq 0}$ derived in Lemma \ref{lem:fBm 1d PSD} is given by
    $$S_Z(\omega)=\sigma^2C_H\frac{|\omega|^{1-2H}}{\mu^2+\omega^2}.$$
    We rewrite \eqref{eq: real red system} into
    $$\dot{\mathbf{X}} = M(\lambda) \mathbf{X} + \bm{\sigma} Z_t.$$
    Applying the Fourier transformation elementwise, we get
    $$i\omega \mathcal{F}(\mathbf{X})(\omega)=M(\lambda) \mathcal{F}(\mathbf{X})(\omega) + \bm{\sigma} \mathcal{F}(Z)(\omega).$$
    Rearranging and taking the absolute value squared, we derive
    $$|\mathcal{F}(\mathbf{X})(\omega)|^2 = (M(\lambda)-i\omega I)^{-1} \bm{\sigma} \mathcal{F}(Z)(\omega)\overline{\mathcal{F}(Z)(\omega)}\bm{\sigma}^*(M(\lambda)^*+i\omega I)^{-1}.$$
    Using the alternate representation for the SD \eqref{eq:PSD via Fourier} we get
    \begin{align*}
        S_{\mathbf{X}}(\omega) 
        &= (M(\lambda)-i\omega I)^{-1}\bm{\sigma} S_Z(\omega)I\bm{\sigma}^* (M(\lambda)^*+i\omega I)^{-1} \\
        &=(M(\lambda)-i\omega I)^{-1}\bm{\sigma}\bm{\sigma}^* (M(\lambda)^*+i\omega I)^{-1} S_Z(\omega).
    \end{align*}
    \qed\\
\end{proof}
\begin{lemma}
    Let $(Z_t)_{t\geq 0}$ be a stationary fOU process \eqref{eq: ou forcing}. Then the stationary solution $(\mathbf{Y}_t)_{t\geq 0}$ of \eqref{eq: false red system} has the spectral density
    $$S_\mathbf{Y}=C_H\frac{|\omega|^{3-2H}}{\mu^2+\omega^2}(M(\lambda)-i\omega I)^{-1} \bm\sigma\bm\sigma^*(M(\lambda)^*+i\omega I)^{-1}.$$
\end{lemma}
\begin{proof}
    First, we derive the SD of $dZ$. Recall from \cite[Chapter 4.12.5]{Priestley_Spectral_Analysis} that the differentiation is a linear transformation and we therefore informally obtain  
    $$S_{\txtd Z}(\omega) = \omega^2 S_Z(\omega).$$
    The rigorous derivation of the previous formula uses the spectral distribution function and defines $S_{\txtd Z}(\omega)$ only as a generalized spectral density. Applying the linear transformation from the SDE entails the SD. We omit the details for brevity since the methodology is similar to the fBm case in Lemma \ref{lem:PSD S2 fBm}. We compute the autocovariance $V(\tau)$ using the spectral representation and then apply the Fourier transform to get the SD.
    We remind the reader that the spectral density of the increments exists if the autocovariance $V(\tau)$ is integrable as stated in Definition~\ref{def:PSD}. Because $Z$ is a zero-mean weakly stationary increment process, there exists a spectral representation \cite[Theorem 2.7]{Basse_stationary_increments} for the increments of $Z$. Therefore for all $s,t\in\R$ with $s<t$ we have 
    $$Z_t-Z_s = \int_{\R} \mathcal{F}(\mathbbm{1}_{(s,t]})(z)~\txtd \mathcal{Z}(z),$$
    and for $s_1\leq t_1 \leq s_2 \leq t_2$ 
    \begin{align*}
        \textnormal{Cov}((Z_{t_1}-Z_{s_1}),(Z_{t_2}-Z_{s_2}))
        &=\E[(Z_{t_1}-Z_{s_1})\overline{(Z_{t_2}-Z_{s_2})}]\\
        &=\int_\R\mathcal{F}(\mathbbm{1}_{(s_1,t_1]})(z)\overline{\mathcal{F}(\mathbbm{1}_{(s_2,t_2]})(z)}~\txtd F(z),
    \end{align*}
    where the measures $F$ and $\mathcal{Z}$ are uniquely determined by $Z$. We call $\mathcal{Z}$ an orthogonal-increment process and $F$ the spectral distribution. As we assume the increments of $Z$ to have a spectral density, we know that $F$ is absolutely continuous with respect to the Lebesgue measure. Hence, we have 
    $$\textnormal{Cov}((Z_{t_1}-Z_{s_1}),(Z_{t_2}-Z_{s_2}))=\int_\R\mathcal{F}(\mathbbm{1}_{(s_1,t_1]})(z)\overline{\mathcal{F}(\mathbbm{1}_{(s_2,t_2]})(z)} S_{\txtd Z}(z)~\txtd z.$$
    Moreover, $\mathcal{Z}$ and $F$ are employed in the construction of integrals against $Z$. By \cite[Theorem 3.5]{Basse_stationary_increments} we have 
    $$\int_\R \varphi(t)~\txtd Z_t = \int_\R \mathcal{F}(\varphi)(z)~\txtd \mathcal{Z}(z),$$
    for $\varphi\in\Lambda_{\textnormal{func}}$ where
    $$\Lambda_{\textnormal{func}}=\left\{\varphi\in L^2(\R): \int_\R |\mathcal{F}(\varphi)(z)|^2 ~\txtd F(z)<\infty \right\}.$$
    With this representation, the covariance between the integrals of $\varphi,\psi\in\Lambda_{\textnormal{func}}$ is given by
    \begin{align*}
        \text{Cov}\left(\int_\R \varphi(u)~\txtd Z_u, \int_\R \psi(u)~\txtd Z_u\right)
        &=\E\left[\int_\R \varphi(u)~\txtd Z_u \overline{\int_\R \psi(u)~\txtd Z_u} \right]\\
        &=\E\left[\int_\R \mathcal{F}(\varphi)(z)~\txtd \mathcal{Z}(z) \overline{\int_\R \mathcal{F}(\psi)(z')~\txtd \mathcal{Z}(z')} \right] \\
        &= \int_\R \mathcal{F}(\varphi)(z)\overline{\mathcal{F}(\psi)(z)}~\txtd F(z)\\
        &= \int_\R \mathcal{F}(\varphi)(z)\overline{\mathcal{F}(\psi)(z)} S_{\txtd Z}(z)~\txtd z.
    \end{align*}
    We consider now the stationary solution of \eqref{eq: false red system} given by 
    $$\mathbf{Y}_t = \int_{-\I}^t e^{-M(\lambda)(t-u)}\bm{\sigma}~\txtd Z_u.$$
   In order to compute the covariance, we need to check that $\mathbbm{1}_{\{u\le\tau\}}e^{M(\lambda)u}\bm{\sigma}\in \Lambda_{\textnormal{func}}$ for $\tau\in\R$. We first drop $\bm{\sigma}$ as it is not dependent on $u$ and use again the Fourier transform computed in \eqref{eq:Fourier transform of exponential}
    \begin{align*}
        \int_\R \left|e^{iz\tau}e^{M(\lambda)\tau}\left(M(\lambda)-izI \right)^{-1} \right|^2S_{\txtd Z}(z)~\txtd z
        \le\left|e^{M(\lambda)\tau}\right|^2 \int_\R \left|\left(M(\lambda)-izI \right)^{-1} \right|^2\frac{|z|^{3-2H}}{\mu^2+z^2}~\txtd z<\infty.
    \end{align*}
    Now we have the necessary tools to compute the autocovariance for $\mathbf{Y}_t$ as
    \begin{align*}
        V(\tau) 
        &= \text{Cov}\left(\int_\R \mathbbm{1}_{\{u\le 0\}} e^{M(\lambda)u} \bm{\sigma} ~\txtd Z_u, \int_\R e^{-M(\lambda)\tau}\mathbbm{1}_{\{u\le \tau\}} e^{M(\lambda)u} \bm{\sigma}~\txtd Z_u\right) \\
        &= \E\left[\int_\R \mathbbm{1}_{\{u\le 0\}} e^{M(\lambda)u} \bm{\sigma} ~\txtd Z_u \left(\int_\R e^{-M(\lambda)\tau}\mathbbm{1}_{\{u\le \tau\}} e^{M(\lambda)u} \bm{\sigma}~\txtd Z_u\right)^*\right] \\
        &= \int_\R (M(\lambda)-izI)^{-1} \bm{\sigma} \bm{\sigma}^* (M(\lambda)^*+izI)^{-1} e^{iz\tau}e^{M(\lambda)^*\tau}e^{-M(\lambda)^*\tau}S_{\txtd Z}(z) ~\txtd z\\
        &= \int_\R (M(\lambda)-izI)^{-1} \bm{\sigma} \bm{\sigma}^* (M(\lambda)^*+izI)^{-1} e^{iz\tau}S_{\txtd Z}(z) ~\txtd z.
    \end{align*} 
    Applying the Fourier transform we get by the definition that \eqref{eq:def spectral density}
    \begin{align*}
        S_\mathbf{Y}(\omega) = (M(\lambda)-i\omega I)^{-1} \bm{\sigma} \bm{\sigma}^* (M(\lambda)^*+i\omega I)^{-1} S_{\txtd Z}(\omega),
    \end{align*}
    which concludes the proof.
    \qed \\
\end{proof}

We can clearly distinguish the two main influences on the SD. First, we have the term from the linear part
$$(M(\lambda)-i\omega I)^{-1} \bm{\sigma} \bm{\sigma}^* (M(\lambda)^*+i\omega I)^{-1}.$$
This term is the equivalent to the mitigation term appearing in the autocovariance formula. It causes peaks at $\pm B(\lambda)$, which increase with power-law $-2$ as $A(\lambda)\to0$. Second, we have the SD of the noise. The SD of the noise typically introduces either a maximum peak at some frequency (fBm with $H=\frac{1}{2}$), a singularity at some frequency (fBm with $H>\frac{1}{2}$) or is 0 at some frequency (fBm with $H<\frac{1}{2}$). For all the noise terms considered in this paper, the SD is mainly affected at $\omega=0$. There are examples where this is not the case, the most prominent one being the stationary solution of \eqref{eq:fast system simplified hopf 2}. 
The structure of the SD of the stationary solution depends on the interplay between the linear part and the noise. If both act on the same frequency, as in the setting \nameref{S1}, the noise can obstruct the divergence of the peak which warns us of the approaching bifurcation. For this case, we used in Section \ref{sec:SD S1 fBm} the SD at a frequency close to zero $S_x(\delta)$ for $\delta>0$ to observe the divergence. If both components act at different frequencies as in the setting \nameref{S2}, we can observe each separately. The divergence due to the approaching bifurcation can be observed from the frequencies $\omega=\pm B(\lambda)$ and the noise effects can be inferred considering the frequency $\omega=0$.\\
In the one-dimensional case the SDs simplify to
\begin{align*}
    S_X(\omega) = C_H  \frac{|\omega|^{1-2H}}{\mu^2+\omega^2}\frac{\sigma^2 }{M(\lambda)^2+\omega^2},
\end{align*}
and
\begin{align*}
    S_Y(\omega) = C_H \frac{|\omega|^{3-2H}}{\mu^2+\omega^2}\frac{\sigma^2 }{M(\lambda)^2+\omega^2}.
\end{align*}
From these expressions, the slopes of $S^{\max}$ listed in Table~\ref{tab:fold overview} follow by direct computation. Furthermore, $S_X(\delta)$ and $S_Y(\delta)$ have a power-law behavior with exponent $-2$ in the regime $|M(\lambda)|>\delta$. This can be shown by an analogous proof to Lemma \ref{lem:PSD0 for fBm}.

\section{Applications} \label{sec:applications} 

In this section, we provide insights on the practical implementation of the EWSs introduced in the sections above.~First, we discuss the approximations adopted in \eqref{eq:fast system simplified} and the time-asymptotic limit, such as in \eqref{eq:PSD via Fourier}, \eqref{eq:time-asymptotic autocov} and \eqref{eq:time-asymptotic autocorr}. Moreover, we discuss their role on the use and construction of statistical estimators that allow for the study of real-life data. Finally, we implement the estimators to trajectories obtained from an ocean model and from a theoretical system. Consequently, we cross-validate our analytic results through the warning of impending fold and Hopf bifurcations.

\subsection{Practical use of early-warning signs}

We discuss the implementation of the EWSs studied in the previous sections and assess the observability of the corresponding indicators and their estimators in the collected data. The insights presented are therefore not restricted to a specific application. Rather, they are broadly relevant to a wide range of tipping points and their associated critical transitions. First, we set $T\gg0$ and label $\left(\mathbf{x}_t^{(1)}\right)_{t\in[0,T]}$ that solves \eqref{eq:fast system} for a given initial condition at $t=0$ and $0<\eps\ll1$, which we interpret as the model associated with a real-life occurrence. Then, we discretize the time interval $[0,T]$ into an equidistant partition
$$\cP=\left\{ t_j| \; t_j = j T/N \right\}_{j\in\{0,\dots, N\}}$$ 
and consider $\left\{\mathbf{x}_t^{(1)}\right\}_{t\in\cP}$ as a data sample collected over a large time interval. As we aim to predict critical transition events, we analyze portions of data in windows $\left\{\cW_j\right\}_j\subset[0,T]$ prior to the critical threshold. However, the application of the EWSs presented in the paper requires considerations on the approximations in Section \ref{sec:preliminaries}:
\begin{description}
    \assumptionitem{A1}{A1} The model \eqref{eq:fast system} refers to an initial mathematical reconstruction of a phenomenon. While function $f$ indicates more abrupt dynamics of the trajectory, function $g$ defines its most common motions in time. The role of this functions is defined by $\eps$, which dictates the time regimes of the components within the solution. While in our analysis we consider $\eps\ll1$ and then $\eps=0$ in the fast time regime \eqref{eq:fast system simplified}, the actual order of $\eps$ deeply affects the critical transitions. A sufficiently large $\eps$ can increase the difficulty in finding a reliable estimate of the bifurcation threshold. In fact, a fast approach in the critical limit can force the trajectory out of proximity to the stable branch and also deprive the time required to effectively register critical slowing down. Nonetheless, such a setting can be in contrast with the definition of a tipping phenomenon such as intended. In this view, a parameter $\eps\ll1$ that is sufficiently small to register critical slowing down is a required property for a tipping event.
    
    \assumptionitem{A2}{A2} The linearization approximation in \eqref{eq:fast system simplified} is justified for dynamics in a neighborhood of the stable branch $h_0^\text{a}(\lambda)$. It provides a simplified setting that allows to explicitate the behavior of the indicator in the approach to the the critical regime. Nonetheless, it overrides the dynamics outside of the basin of attraction. In particular, it does not consider the possibility of escape from a saddle and the existence of further stable states. Such an approximation therefore weakens in a regime where the nonlinear components within $f$ dominate during the registered time interval. A prime example is in the proximity to the critical transition. In such a setting, either metastable jumps are more likely to occur or nonlinear dynamics can shut down the divergence of the indicators \cite{bernuzzi2026early}. In this vision, the correct implementation of observables as EWSs becomes the registration of an increase in the rates as presented in Table \ref{tab:fold overview} and Table \ref{tab:Hopf overview}. While such a growth is eventually halted by forcings induced by nonlinear terms, it hints at a loss of resilience, or diminished stability of the branch. In fact, the EWSs do not only precede a critical transition, but warn of the increasing likelihood of an escape from the basin of attraction.

    \assumptionitem{A3}{A3} All analytic observables studied in the previous sections require time-asymptotic knowledge of the trajectory. This clearly becomes an idealistic assumption in applications and in the fast-slow regime for $\eps>0$. From this perspective, the time-asymptotic limit serves as an approximation of indicators estimated from observations collected over long time intervals. Our results are therefore particularly relevant in settings where data are available over sufficiently large observation windows. Although this approximation influences the quantitative behavior of the indicators, it does not alter their qualitative properties when the system operates within the appropriate regime. In particular, under sufficient resilience, or equivalently enough distance from the critical threshold, a proper approximation of the time-asymptotic indicators for the linearized fast model requires only a finite time interval of data observation \cite[Remark 3.2]{bernuzzi2025early}.
\end{description}

\begin{figure}[h!]
    \centering
    \subfloat{\begin{overpic}[scale=0.63]{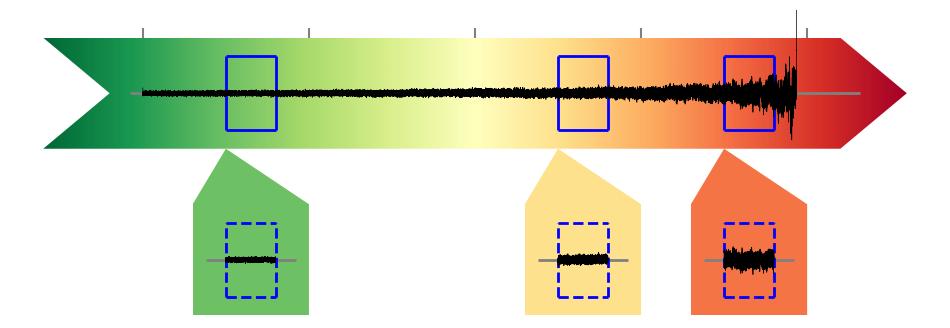}
    \put(158,313){\small{$-20$}}
    \put(333,313){\small{$-15$}}
    \put(508,313){\small{$-10$}}
    \put(683,313){\small{$-5$}}
    \put(858,313){\small{$0=\lambda^*$}}
    \put(-7,225){\shortstack{Collected\\data}}
    \put(-15,50){\shortstack{Analytic\\trajectories}}
    \end{overpic}}
    
    \caption{The upper part of the figure shows a stochastic process approaching a subcritical pitchfork bifurcation, solving \eqref{eq: sub_pitch} with $\eps = 10^{-3}$, $\sigma = 10^{-1}$, and $H = 0.75$. The slow variable increases linearly toward $\lambda^* = 0$ with initial condition $y_0 = -20$, and its value is encoded by the color of the arrow. The fast component (black) evolves in the direction of the arrow; its noise is generated via the Davies-Harte method and integrated using an implicit Euler scheme \cite{dieker2004simulation}. Blue windows indicate intervals $[s_0, s_0 + 1500]$ for $s_0 \in \{2500, 12500, 17500\}$, chosen to compare different magnitudes of $y_{s_0}^{-1}$ and stochastic oscillation amplitudes.
    The lower part shows solutions of \eqref{eq:fast system simplified} with corresponding parameters, starting from $\mathbf{x}_{s_0}$ at $t = s_0$ and $M(\lambda) = \lambda = y_{s_0}$. The color of each shape encodes this value. Trajectories in vertically aligned blue regions are similar away from criticality but diverge near the critical regime, where the upper trajectory departs from the stable branch (gray).}
    \label{fig:Illustration_2}
\end{figure}
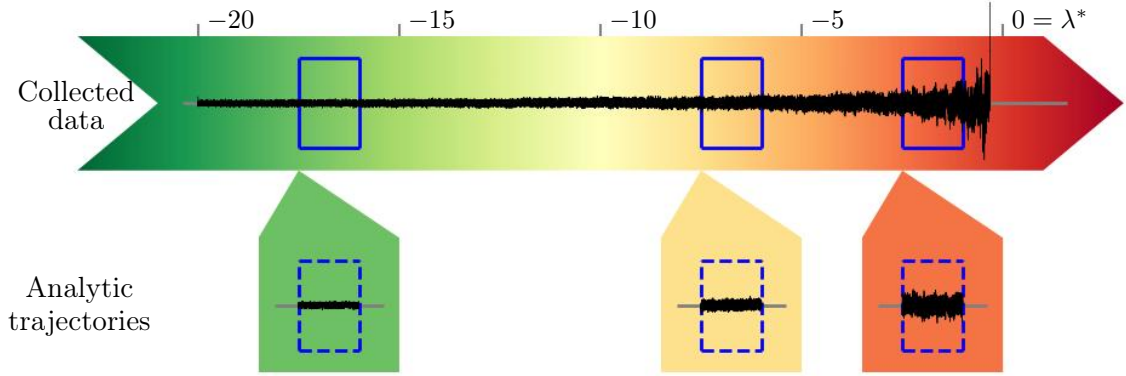 

Having discussed the role of the addressed approximations on system \eqref{eq:fast system}, we study their use within the implementation of the EWSs. An example of a considered data sample, is shown in the first row of Figure \ref{fig:Illustration_2} for
\begin{align} \label{eq: sub_pitch}
    \begin{cases}
        \txtd \mathbf{x}_t^{(1)} = \left(y_t\mathbf{x}_t^{(1)} + \left(\mathbf{x}_t^{(1)}\right)^3\right)~\txtd t + \sigma~\txtd W_t^H, \\
        \txtd y_t = \eps ~\txtd t,
    \end{cases}
\end{align}
with $n=1$, $H=0.75$ and $y_0<0$. Consequently, $\left(\mathbf{x}_t^{(1)}\right)_{t\in[0,T]}$ solves the normal form corresponding to the subcritical pitchfork bifurcation for a variable that slowly approaches the critical threshold in $[0,T]$. The values of $(y_t)_{t\in[0,T]}$ are indicated on the arrow in the figure and associated to a color. The time series $\left(\mathbf{x}_t^{(1)}\right)_{t\in\mathcal{P}}$ is divided into portions within closed windows, indicated as blue solid squares, for a sufficiently large fixed time interval. We then set a window $\mathcal{W}\subset [0,T]$ and label the discrete set $\mathcal{P}'=\mathcal{P}\cap \mathcal{W}$ of size $0\ll N_1\ll N$. Since $\eps\ll1$, the variable $(y_t)_{t\in\mathcal{P}'}$ and the eigenvalues of $(M(y_t)^{-1})_{t\in\mathcal{P}'}$, for $M(y_t)=\partial_\mathbf{x}f(h_0^a(y_t),y_t,0)$, remain almost constant if $\max(y_t)_{t\in\mathcal{P}'}$ is distant from $\lambda^*$. Consequently, we define $\left(\mathbf{x}_t^{(2)}\right)_{t\in\mathcal{W}}$ that solves \eqref{eq:fast system simplified} for initial condition $\mathbf{x}_{s_0}^{(2)}= \mathbf{x}_{s_0}^{(1)}- h_0^a(y_{s_0})$ on $s_0=\min(\mathcal{W})$ and with equivalent noise realization as in \eqref{eq:fast system}. As indicated by \nameref{A1} and \nameref{A2}, the difference between the trajectories $\left(\mathbf{x}_t^{(1)}\right)_{t\in\mathcal{W}}$ and $\left(\mathbf{x}_t^{(2)}\right)_{t\in\mathcal{W}}$ is small and provides a reliable approximation under the mentioned conditions. For our example, the trajectories $\left(\mathbf{x}_t^{(2)}\right)_{t\in\mathcal{P}'}$ are shown in Figure \ref{fig:Illustration_2} in blue dashed squares and within rectangles whose color indicates the value of $\lambda=M(\lambda)$ in \eqref{eq:fast system simplified}.

The application of statistical estimators associated to the observables studied in the paper on $\left(\mathbf{x}_t^{(1)}\right)_{t\in\mathcal{W}}$ is then justified by its resemblance to the solution of the linearized system in $\mathcal{W}$. However, the finiteness of the collected data and the boundedness of $[0,T]$ require the finite-time approximation \nameref{A3} within their construction. The statistical estimators are listed below along with a description of their implementation in the elements of the trajectories addressed.
\begin{itemize}
    \item The time-asymptotic autocovariance \eqref{eq:time-asymptotic autocov} for lag time $\tau= d/N$, with $d\in \{0,\dots, N_1-1\}$, is approximated as
    \begin{align*}
        \mathtt{V}^{(k)}(\tau)[\mathbf{w}_1,\mathbf{w}_2]
        :=& \left( N_1-d \right)^{-1} \sum_{j\in\cP_{d}'} \left\langle \mathbf{x}_j^{(k)} , \mathbf{w}_1 \right\rangle \overline{\left\langle \mathbf{x}_{j+d}^{(k)} , \mathbf{w}_2 \right\rangle}\\
        &- \left( N_1-d \right)^{-2} \left(\sum_{j\in\cP_{d}'} \left\langle \mathbf{x}_j^{(k)} , \mathbf{w}_1 \right\rangle \right)
        \left( \sum_{j\in\cP_{d}'} \overline{\left\langle \mathbf{x}_{j+d}^{(k)} , \mathbf{w}_2 \right\rangle} \right) ,
    \end{align*}
    for $k\in\{1,2\}$, any $\mathbf{w}_1,\mathbf{w}_2\in\C^n$ and
    \begin{align*}
        \cP_{d}'=\{j\in \cP'| \; j+d \in \cP'\},
    \end{align*}
    i.e. by the discrete covariance with lag index $d$ of $\left\{\left\langle \mathbf{x}_j^{(k)} , \mathbf{w}_1 \right\rangle\right\}_{j\in\cP'}$ and $\left\{\left\langle \mathbf{x}_j^{(k)} , \mathbf{w}_2 \right\rangle\right\}_{j\in\cP'}$. The choice of $\mathbf{w}_1$ and $\mathbf{w}_2$ is associated to the components of the observed trajectory. For instance, for proxy vectors in the canonical basis $\{\mathbf{b}_j\}_{j\in\{1,\dots,n\}}$ of $\R^n$, then $\left(\mathtt{V}^{(k)}(\tau)[\mathbf{b}_{j_1},\mathbf{b}_{j_2}]\right)_{j_1,j_2\in\{1,\dots,n\}}$ corresponds to the covariance matrix of $\left\{ \mathbf{x}_j^{(k)}\right\}_{j\in\cP'}$. By definition, it collects the time covariance with lag time $\tau$ of each of the elements. Consequently, we indicate $\mathtt{V}^{(k)}(\tau)_{j_1j_2}:= \mathtt{V}^{(k)}(\tau)[\mathbf{b}_{j_1},\mathbf{b}_{j_2}]$ for any $j_1,j_2 \in\{1, \dots, n\}$ to simplify the notation;

    \item the time-asymptotic autocorrelation \eqref{eq:time-asymptotic autocorr} for lag time $\tau= d/N$ is estimated with
    \begin{align*}
        \mathtt{AC}^{(k)}(\tau)[\mathbf{w}]:=\frac{\mathtt{V}^{(k)}(\tau)[\mathbf{w},\mathbf{w}]}{\mathtt{V}^{(k)}(0)[\mathbf{w},\mathbf{w}]} ,
    \end{align*}
    for $k\in\{1,2\}$ and any $\mathbf{w}\in\C^n$, i.e., the discrete autocorrelation with lag index $d$ of $\left\{\left\langle \mathbf{x}_j^{(k)} , \mathbf{w} \right\rangle\right\}_{j\in\cP'}$. Similar to the estimator above, the choice of $\mathbf{w}\in\C^n$ is related to the elements of the process $\left\{ \mathbf{x}_j^{(k)}\right\}_{j\in\cP'}$. While for $\mathbf{w}\in\{\mathbf{b}_j\}_{j\in\{1,\dots,n\}}$ the estimator corresponds to the time autocorrelation of a single coordinate of the studied trajectory, other vectors refer to projections along further modes;

    \item we approximate the spectral density \eqref{eq:PSD via Fourier} using Welch's method \cite{ChenPSDforHopf}. For fixed $k\in\{1,2\}$, this consists of dividing the time series $\left\{\mathbf{x}_t^{(k)}\right\}_{t\in\cP'}$ into $K\in\N_{>0}$, possibly overlapping, segments of length $N_2\ll N_1$, applying a Hamming window to each segment, computing the periodogram of each windowed segment, and then averaging. More precisely, for a segment starting at index $s$, we define the windowed data $\left\{ h_j f_{s+j}\right\}_{j\in\{0,\dots,N_2-1\}}$, where the $\left\{ h_j \right\}_{j\in\{0,\dots,N_2-1\}}$ is a Hamming window \cite{ChenPSDforHopf}.
    The discrete Fourier transform on the windowed segment is indicated then by $\hat{\mathbf{x}}_s^{(k)}(\omega)$ for any $\omega \in \{0,\dots, N_2-1\}$. The corresponding periodogram is then the matrix $\frac{1}{N_2}\hat{\mathbf{x}}_s^{(k)}(\omega) \left( \hat{\mathbf{x}}_s^{(k)}(\omega) \right)^\text{T}$. Finally, Welch's estimator of the spectral density, labeled $\mathtt{S}^{(k)}(\omega)$, is obtained by averaging over all segments. In contrast to the previous estimators, the study of the elements of the estimator corresponds directly to the observation of single coordinates of the trajectory, as opposed to using a spectral perspective;
\end{itemize}

The estimators listed above are consistent \cite{ChenPSDforHopf,kubilius2017parameter} from the ergodic property \cite{kubilius2017parameter} of the solutions of \eqref{eq:fast system simplified}, i.e., for $k=2$. In fact, let us set $N=\left \lfloor c_0 T \right \rfloor $ and $N_1=\left \lfloor c_1 T \right \rfloor$ with $0<c_1 \ll c_2$. For $N_1\to \infty$, the estimator $\mathtt{V}^{(2)}(\tau)[\mathbf{w}_1,\mathbf{w}_2]$ displays equivalent scaling laws as indicated in Theorem \ref{thm:autocov fast system} and Theorem \ref{thm:autocov fast system complex} under general $\tau\geq 0$ and $\mathbf{w}_1,\mathbf{w}_2\in\C^n$. It follows that $\mathtt{AC}^{(2)}(\tau)[\mathbf{w}_1,\mathbf{w}_2]$ shows the rate of convergence stated in Corollary \ref{cor:autocor fast system} and Corollary \ref{cor:autocor fast system complex} for almost every $\tau\geq 0$ and $\mathbf{w}\in\C^n$. Finally, the rate of divergence of $\mathtt{S}^{(2)}(\omega)$, for appropriate $\omega$, is shown in Lemma \ref{lem:fBm 1d PSD} and Theorem \ref{thm:psd_on_elements} in the limit $N_1\to \infty$, $N_2=N_2(N_1)\to \infty$ and $K=K(N_1)\to \infty$. \\

We note that, in practice, the estimators applied to $\left(\mathbf{x}_t^{(1)}\right)_{t\in\mathcal{P}'}$ are affected by the assumptions taken in our analysis. We consider first $\mathtt{V}^{(1)}(\tau)[\mathbf{w}_1,\mathbf{w}_2]$ and $\mathtt{S}^{(1)}(\omega)$ for $\tau\geq0$, $\mathbf{w}_1,\mathbf{w}_2\in\C^n$ and $\omega\in\R$ such that $\mathtt{V}^{(2)}(\tau)[\mathbf{w}_1,\mathbf{w}_2]$ and $\mathtt{S}^{(2)}(\omega)$ are divergent in the limit $\lambda\to\lambda^*$ with an equivalent rate as discussed above. The estimators behave as follows: they can be influenced by the change of the slow variable, \nameref{A1}; they could be restrained or incur into a metastable jump by nonlinear forcings, \nameref{A2}; they are damped by finite-time samples, \nameref{A3}. Such effects are more present in the proximity to the bifurcation threshold, as displayed in Figure \ref{fig:Illustration_2}. For a collection of non-overlapping windows $\left\{\cW_j\right\}_j$ of increasing extremes, the approximations may result then in three stages in the collection of subtrajectories $\left\{\left(\mathbf{x}_t^{(1)}\right)_{t\in\cW_j\cap\mathcal{P}}\right\}_j$. First, a similar rate as shown in Section \ref{sec:fbm 1d} is adopted by the corresponding estimator, regardless of the bifurcation approached. For Hopf bifurcations, the real part of the sensible eigenvalue becomes eventually smaller in magnitude than the imaginary part. This induces a second stage where the increase stated in Section \ref{sec:fbm 2d} is recorded. Finally, the nonlinear terms are dominant in the proximity to the critical threshold and the growth of the estimators is affected or hindered. We note that the existence of these stages depends on parameters such as the noise intensity, which can force the trajectory far from the stable branch prior to the approach of the critical threshold, and $N_1$, that indicates the amount of data considered in the construction of the statistics. The EWSs are therefore effective if the first stage, or second stage if the bifurcation is Hopf, are detected prior to the eventual damping by nonlinearities.

\subsection{AMOC}
We consider a stochastic Stommel-Cessi \cite{CessiBoxModel} box model for the Atlantic Meridional Overturning Circulation (AMOC) driven by fractional Brownian motion with Hurst index $H\in(0,1)$, which is given by 
\begin{align}\label{eq:AMOC fBm}
    \txtd x_t &= y_t - x_t(1+\eta^2(1-x_t)^2)~\txtd t + \sigma \txtd W_t^H, \nonumber\\
    \txtd y_t &= -\eps~\txtd t,
\end{align}
where $x$ represents the salinity difference in two ocean regions, $y$ is proportional to the temperature difference, and $\eta^2=7.5$ is the ratio between the diffusive and advective time scale as set in \cite{CessiBoxModel}. In our simulation, we consider a small noise regime $\sigma=0.001$. For this model we compute the autocovariance, autocorrelation and SD as EWSs in the fast subsystem regime, i.e. $\eps=0$. Consequently, we plot them on logarithmic scales against the difference of the bifurcation parameter $y=\lambda$ to the critical value $\lambda^*$. This threshold corresponds to the AMOC transition from an on state to an off state, or AMOC collapse \cite{bernuzzi2024warning}.
\begin{figure}[h!]
    \centering
    \begin{subfigure}{0.45\textwidth}
        \centering
        \begin{overpic}[width=\linewidth]{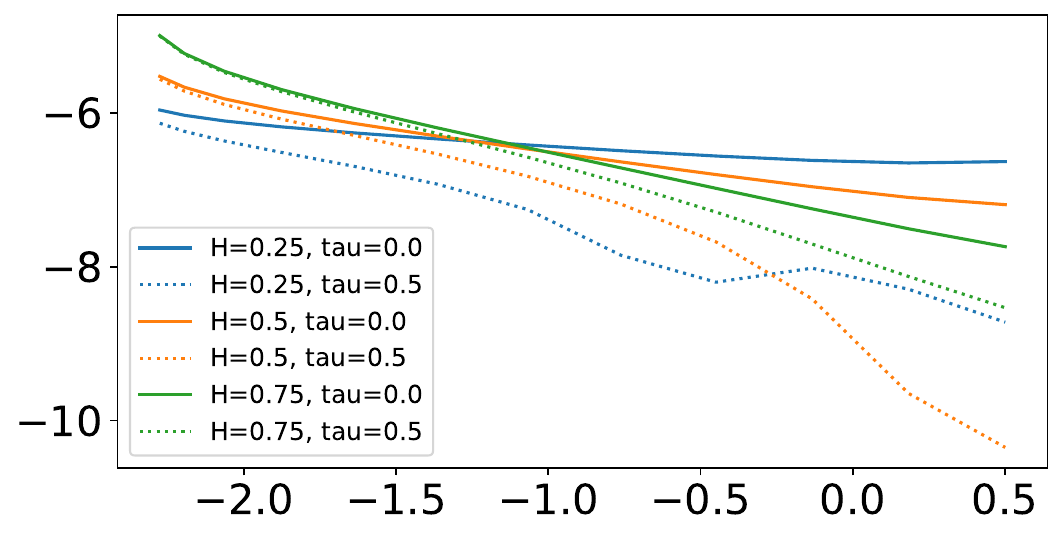}
            \put(-50,520){\large (a)}
            \put(430,-35){\scriptsize $\log_{10}(\lambda-\lambda^*)$}
            \put(-40,250){\rotatebox[]{90}{\scriptsize $\log_{10}(\mathtt{V}(\tau))$}}
        \end{overpic}
    \end{subfigure}
    \hfill
    \begin{subfigure}{0.45\textwidth}
        \centering
        \begin{overpic}[width=\linewidth]{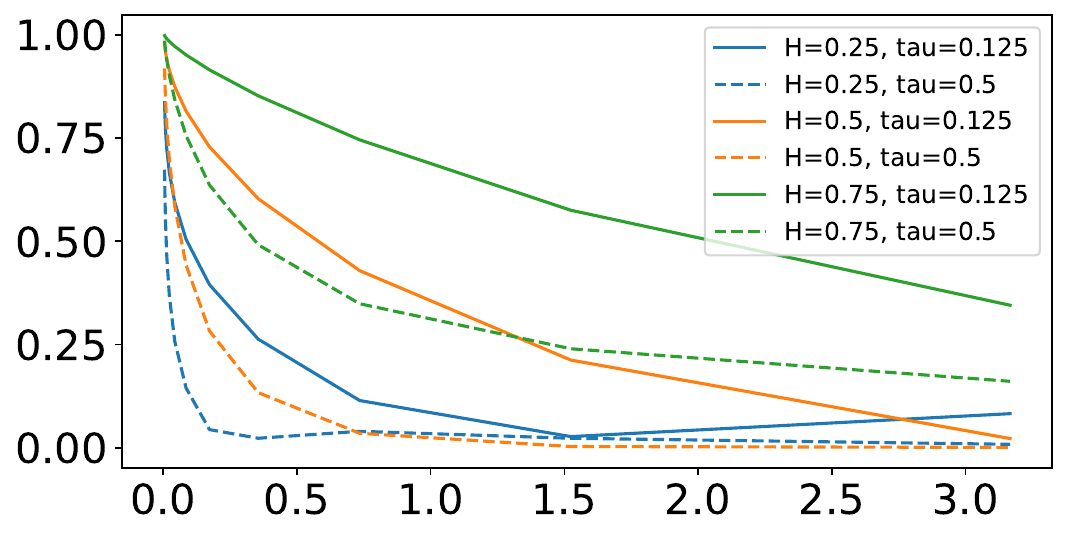}
            \put(-55,520){\large (b)}
            \put(470,-35){\scriptsize $\lambda-\lambda^*$}
            \put(-40,280){\rotatebox[]{90}{\scriptsize $\mathtt{AC}(\tau)$}}
        \end{overpic}
    \end{subfigure}
    \par \vspace{0.5cm}
    \begin{subfigure}{0.45\textwidth}
        \centering
        \begin{overpic}[width=\linewidth]{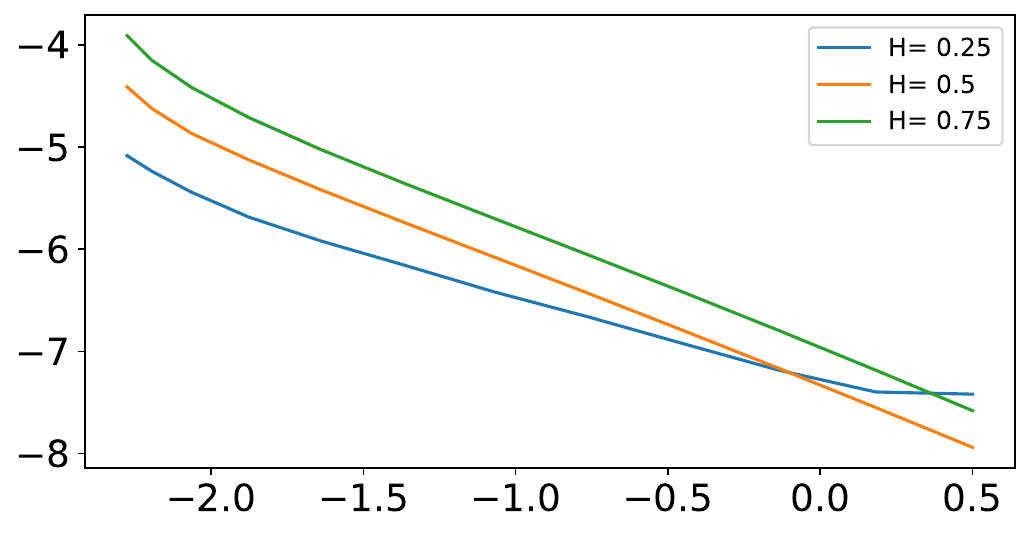}
            \put(-50,520){\large (c)}
            \put(430,-35){\scriptsize $\log_{10}(\lambda-\lambda^*)$}
            \put(-40,250){\rotatebox[]{90}{\scriptsize $\log_{10}(\mathtt{S}^{\max})$}}
        \end{overpic}
    \end{subfigure}
    \hfill
    \begin{subfigure}{0.45\textwidth}
        \centering
        \begin{overpic}[width=\linewidth]{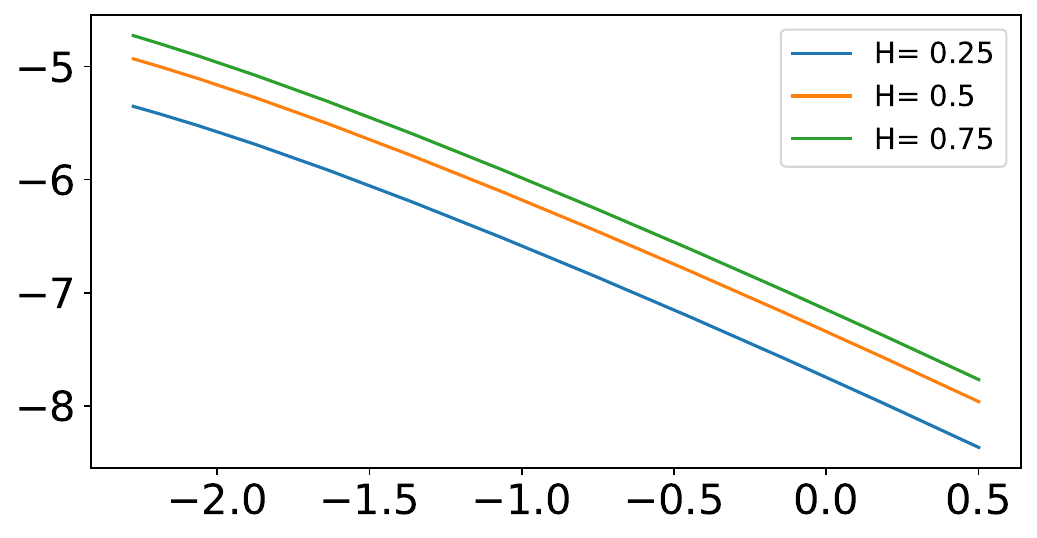}
            \put(-55,520){\large (d)}
            \put(430,-35){\scriptsize $\log_{10}(\lambda-\lambda^*)$}
            \put(-40,250){\rotatebox[]{90}{\scriptsize $\log_{10}(\mathtt{S}(\delta))$}}
        \end{overpic}
    \end{subfigure}
    \caption{(Simulation parameters $T=2^{11},N=2^{15},\text{samples}=20$) Comparing different EWSs for the fBM noise AMOC system \eqref{eq:AMOC fBm}. In each plot we have $H=0.25$ in blue, $H=0.5$ in orange and $H=0.75$ in green. In (a), we see the log-variance, i.e. $\log_{10} \left(\mathtt{V}(0) \right)$, as the solid line and the log-autocovariance with lag time $\tau = 0.5$, i.e. $\log_{10} \left( \mathtt{V}(0.5) \right)$, as the dotted line against the log distance of the bifurcation parameter $\lambda$ from the critical value $\lambda^*$. In (b), we have the autocorrelation for lag $\tau=0.125$, i.e. $\mathtt{AC}(0.125)$, as the solid line and the autocorrelation for lag time $\tau=0.5$, i.e. $\mathtt{AC}(0.5)$, as the dotted line against the distance of the bifurcation parameter $\lambda$ from the critical value $\lambda^*$. In (c), we plot the log of the maximum of the SD, i.e. $\log_{10} \left( \mathtt{S}^{\max}\right)$, against the log distance from the bifurcation point. In (d), we see the log of the SD at a frequency close to $\omega = 0 \approx \delta$, i.e. $\log_{10} \left(\mathtt{S}(\delta)\right)$, against the log distance from the bifurcation point.}
    \label{fig:AMOC fBm EWS}
\end{figure}

For the variance and autocovariance we see in Figure \ref{fig:AMOC fBm EWS} (a) that they all diverge as we approach the bifurcation. As in \cite{kuehn2022warning} the rate of divergence of the variance depends on the Hurst index $H$. For the autocorrelation in (b) we see a convergence to 1 at different rates for different $H$, corroborating the statement of Corollary \ref{cor:autocor fast system}. In both (a) and (b) an increase in the lag time $\tau$ leads to a later increase of the EWS. For the autocovariance, the dotted line in (a), this leads to different slopes compared to those indicated in \eqref{eq:autocov fast subsystem}. Nevertheless, the autocovariance increases significantly and therefore can be used as an EWS. In (c), we see the divergence of $\mathtt{S}^{\max}$ for all $H$. As shown in Lemma \ref{lem:PSD_max for fBM} the rate of divergence decreases for smaller $H$. As mentioned in the Remark~\ref{remark:sd}, the maximum can be computed numerically for $H>1/2$ as well, even though it is not attained. We get a divergence rate of slope $-2$ just as for $H=1/2$. For the SD at a fixed frequency close to zero ($\delta=0.4$), which is shown in (d), we have a power-law of slope $-2$ for all $H\in(0,1)$. We emphasize again that $\mathtt{S}(\delta)$ does not diverge, but increases corresponding to a power-law of slope $-2$ before starting to converge to a fixed value at around $\lambda-\lambda^*=\delta$. By choosing $\delta$ small enough, the convergence part starts very late and we have a large region that displays a slope $-2$. The choice of a frequency closer to zero, while legitimate, might incur into the risk of estimation errors from Welch's method interfering with the right slope. The impact of the estimation method on the SD is shown below. \\
Besides $\mathtt{S}^{\max}$ and $\mathtt{S}(\delta)$ we also plot the structure of the SD and observe the reddening of the noise in Figure \ref{fig:SD AMOC fBm}. 
\begin{figure}[h!]
    \subfloat{\begin{overpic}[scale=0.29]{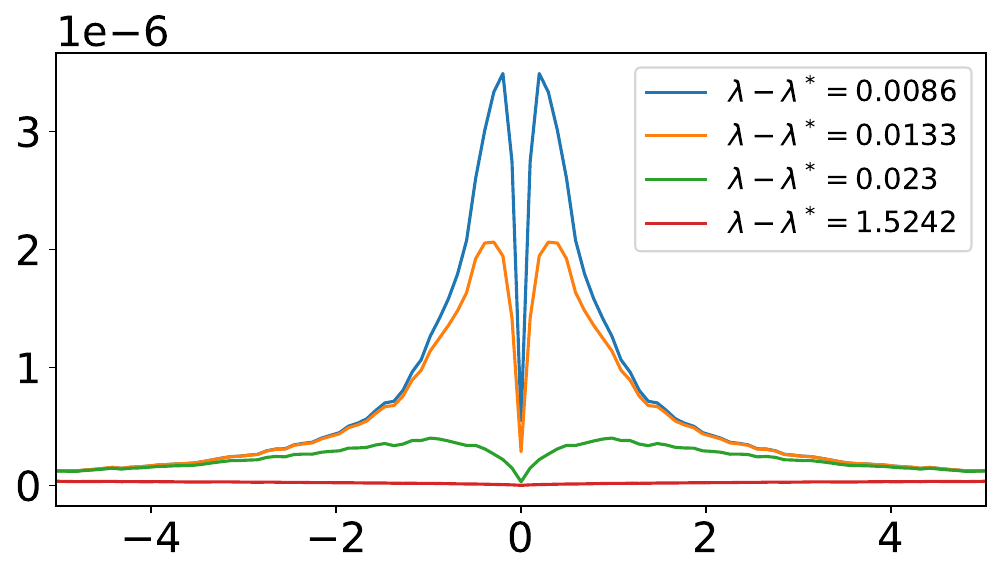}
    \put(-40,-30){(a)}
    \put(370,-35){\scriptsize frequency $\omega$}
    \put(-50,250){\rotatebox[]{90}{\scriptsize $\mathtt{S}(\omega)$}}
    \end{overpic}}
    \hspace{1mm}
    \subfloat{\begin{overpic}[scale=0.29]{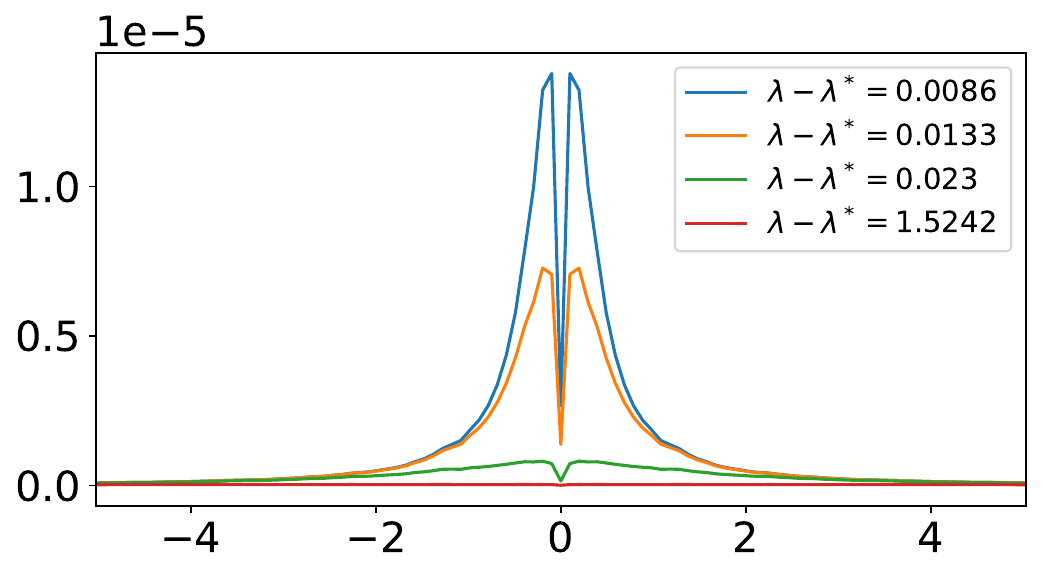}
    \put(-40,-30){(b)}
    \put(380,-35){\scriptsize frequency $\omega$}
    \end{overpic}}
    \hspace{1mm}
    \subfloat{\begin{overpic}[scale=0.29]{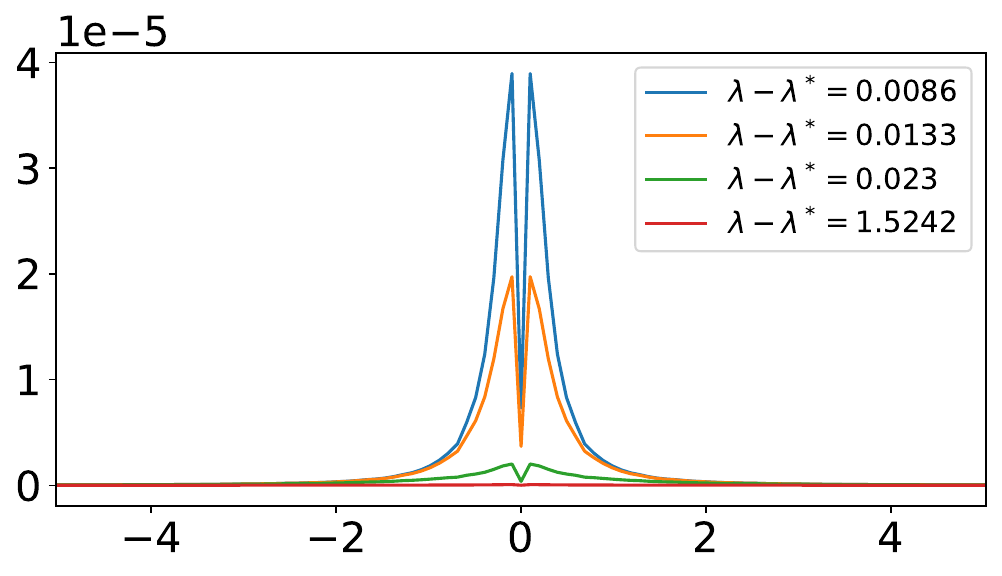}
    \put(-40,-30){(c)}
    \put(370,-35){\scriptsize frequency $\omega$}
    \end{overpic}}
    \caption{Spectral densities for $H=0.25$ (a), $H=0.5$ (b) and $H=0.75$ (c) for \eqref{eq:AMOC fBm}.}
    \label{fig:SD AMOC fBm}
\end{figure}
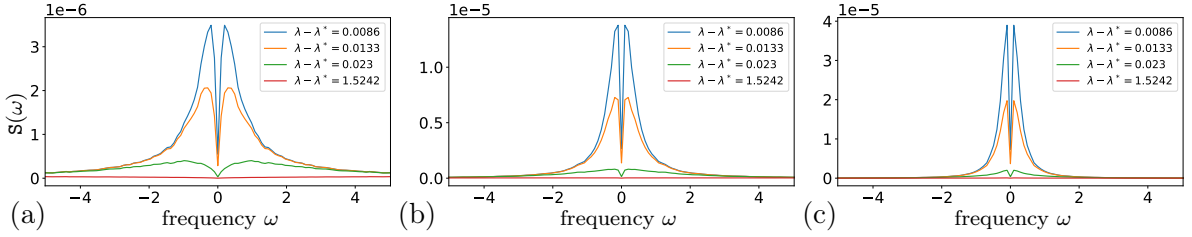
As expected, the SD has more mass at the lower frequencies as we approach the bifurcation. Before discussing the structure of the SD, we note that due to the numerical computation with Welch's method~\cite{ChenPSDforHopf} we have a dip at the center frequency. This is caused by the averaging over the Hamming windows. Such an effect can also be observed in \cite{ChenPSDforHopf}. As we study further the structure of the SD, we see for $H<1/2$ the two humps close to zero that come closer to zero and increase as we approach the bifurcation. These are the maxima computed in Lemma \ref{lem:PSD_max for fBM}. For $H=1/2$ we have the single peak in the middle that increases. This maximum is also obtained in Lemma \ref{lem:PSD_max for fBM}. For $H>1/2$ we have a singularity at 0 and the neighboring frequencies increase as we approach the bifurcation. As the value at each frequency is numerically finite, this provides a valid EWS in $\mathtt{S}^{\max}$ for $H>1/2$. Furthermore, these graphs demonstrate why $\mathtt{S}(\delta)$ is a suitable EWS for all $H\in(0,1)$. We see in all three graphs a clear increase of the SD as we approach the bifurcation at frequencies close to zero.
\\\\
Next, we consider a similar AMOC model driven by red noise. We consider the following fast subsystem
\begin{align}\label{eq:AMOC red noise}
    \txtd x_t &= \left( \lambda - x_t(1+\eta^2(1-x_t)^2) + \sigma Z_t \right)~\txtd t
\end{align}
where $Z$ is the stationary fOU process \eqref{eq: ou forcing}. 
\begin{figure}[h!]
    \centering
    \begin{subfigure}{0.45\textwidth}
        \centering
        \begin{overpic}[width=\linewidth]{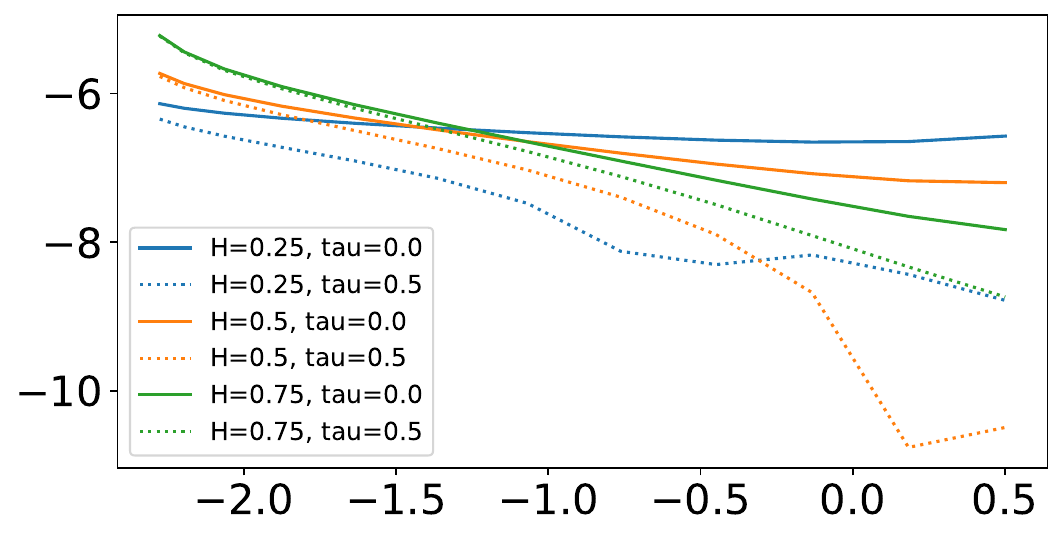}
            \put(-50,520){\large (a)}
            \put(430,-35){\scriptsize $\log_{10}(\lambda-\lambda^*)$}
            \put(-40,250){\rotatebox[]{90}{\scriptsize $\log_{10}(\mathtt{V}(\tau))$}}
        \end{overpic}
    \end{subfigure}
    \hfill
    \begin{subfigure}{0.45\textwidth}
        \centering
        \begin{overpic}[width=\linewidth]{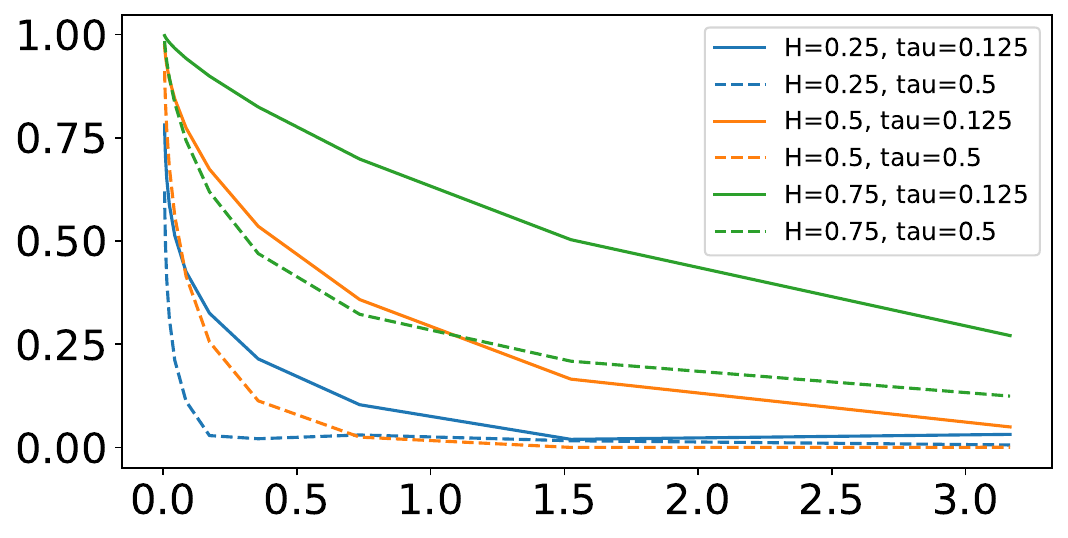}
            \put(-55,520){\large (b)}
            \put(470,-35){\scriptsize $\lambda-\lambda^*$}
            \put(-40,280){\rotatebox[]{90}{\scriptsize $\mathtt{AC}(\tau)$}}
        \end{overpic}
    \end{subfigure}
    \par \vspace{0.5cm}
    \begin{subfigure}{0.45\textwidth}
        \centering
        \begin{overpic}[width=\linewidth]{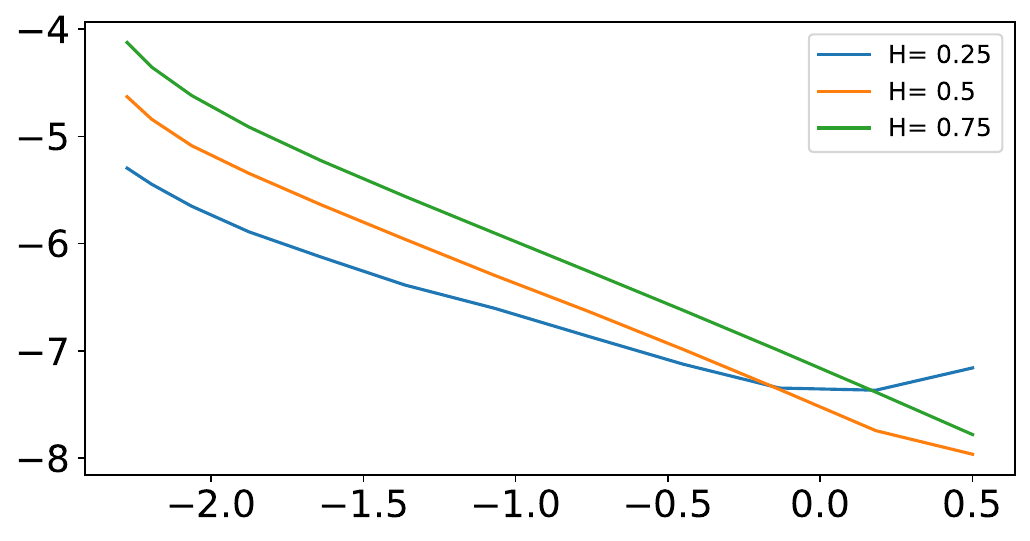}
            \put(-50,520){\large (c)}
            \put(430,-35){\scriptsize $\log_{10}(\lambda-\lambda^*)$}
            \put(-40,250){\rotatebox[]{90}{\scriptsize $\log_{10}(\mathtt{S}^{\max})$}}
        \end{overpic}
    \end{subfigure}
    \hfill
    \begin{subfigure}{0.45\textwidth}
        \centering
        \begin{overpic}[width=\linewidth]{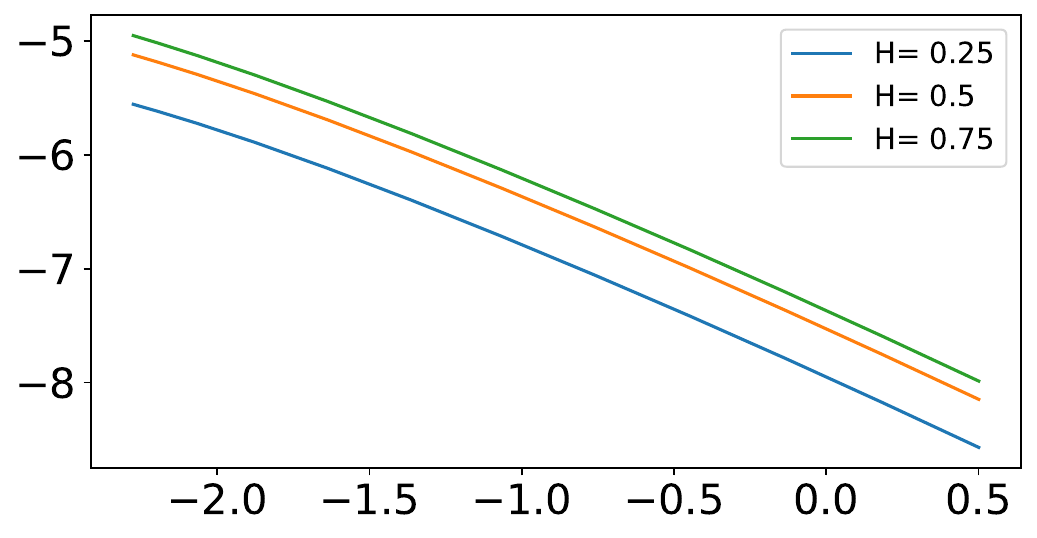}
            \put(-55,520){\large (d)}
            \put(430,-35){\scriptsize $\log_{10}(\lambda-\lambda^*)$}
            \put(-40,250){\rotatebox[]{90}{\scriptsize $\log_{10}(\mathtt{S}(\delta))$}}
        \end{overpic}
    \end{subfigure}
    \caption{(Simulation parameters $T=2^{11},N=2^{15},\text{samples}=20$) Comparing different EWSs for the red noise AMOC system \eqref{eq:AMOC red noise}. In each plot we have $H=0.25$ in blue, $H=0.5$ in orange and $H=0.75$ in green. In (a), we see the log-variance, i.e $\log_{10} \left( \mathtt{V}(0) \right)$ as the solid line and the log-autocovariance with lag $\tau=0.5$, i.e. $\log_{10} \left( \mathtt{V}(0.5) \right)$, as the dotted line against the log distance of the bifurcation parameter $\lambda$ from the critical value $\lambda^*$. In (b), we have the autocorrelation for lag $\tau=0.125$, i.e. $\mathtt{AC}(0.125)$, as the solid line and the autocorrelation for lag $\tau=0.5$, i.e. $\mathtt{AC}(0.5)$, as the dotted line against the distance of the bifurcation parameter $\lambda$ from the critical value $\lambda^*$. In (c), we plot the log of the maximum of the SD, i.e. $\log_{10} (\mathtt{S}^{\max})$, against the log distance from the bifurcation point. In (d), we see the log of the SD at a frequency close to $\omega=0\approx \delta$, i.e. $\log_{10} \left( \mathtt{S}(\delta) \right)$, against the log distance from the bifurcation point.}
    \label{fig:AMOC red noise EWS}
\end{figure}
In Figure \ref{fig:AMOC red noise EWS} we see a very similar behavior to Figure \ref{fig:AMOC fBm EWS}. Qualitatively, the EWSs behave equivalently for both systems. As the SD is also very similar to Figure \ref{fig:SD AMOC fBm}, we omit the plots.\\\\ 
Lastly, we consider the AMOC fast subsystem with fOU noise
\begin{align}\label{eq:AMOC fOU noise}
    \txtd x_t &= \lambda - x_t(1+\eta^2(1-x_t)^2) ~\txtd t + \sigma \txtd Z_t,
\end{align}
where $Z$ is again the stationary fOU process \eqref{eq: ou forcing}.
\begin{figure}[h!]
    \centering
    \begin{subfigure}{0.45\textwidth}
        \centering
        \begin{overpic}[width=\linewidth]{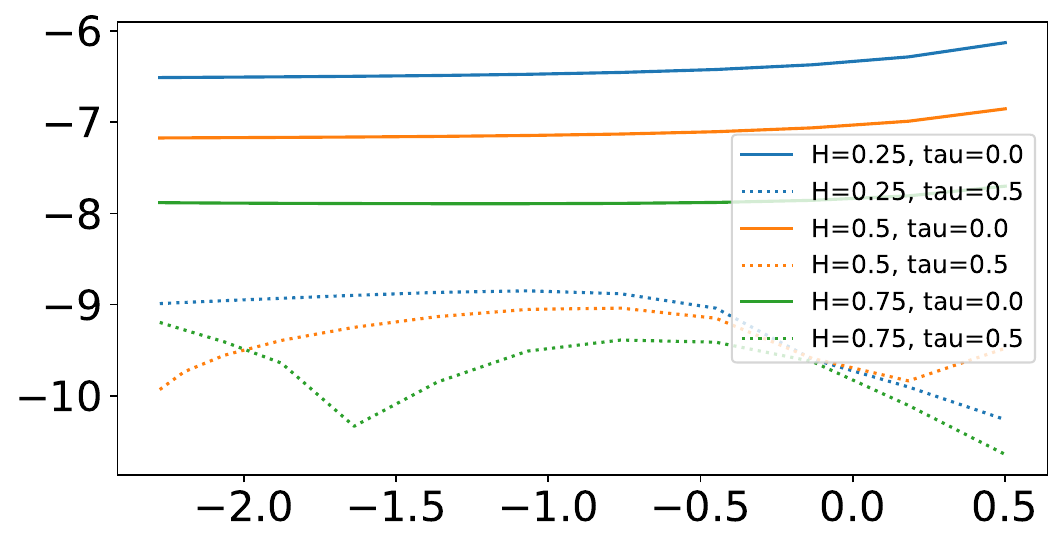}
            \put(-50,520){\large (a)}
            \put(430,-35){\scriptsize $\log_{10}(\lambda-\lambda^*)$}
            \put(-40,250){\rotatebox[]{90}{\scriptsize $\log_{10}(\mathtt{V}(\tau))$}}
        \end{overpic}
    \end{subfigure}
    \hfill
    \begin{subfigure}{0.45\textwidth}
        \centering
        \begin{overpic}[width=\linewidth]{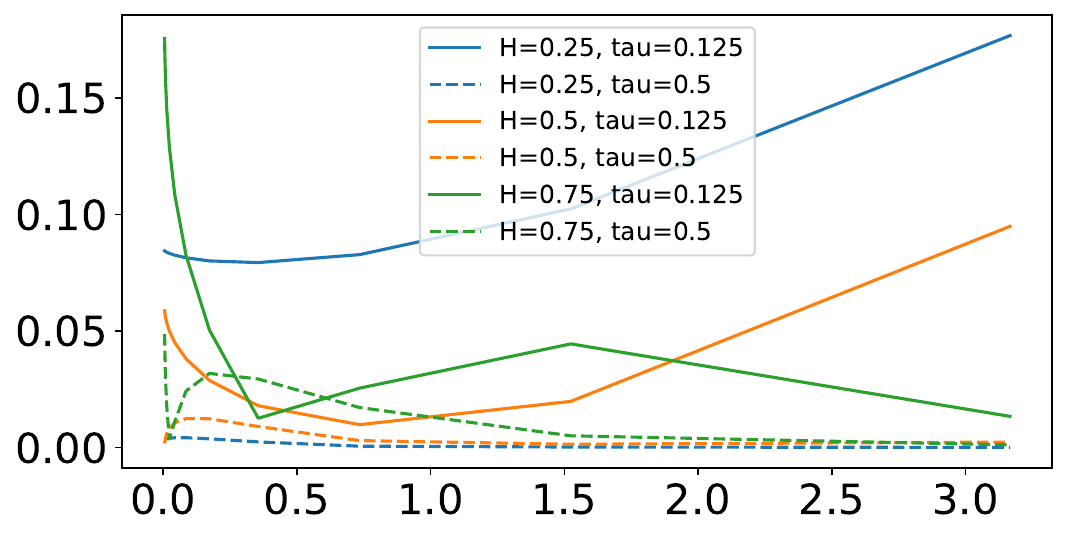}
            \put(-55,520){\large (b)}
            \put(470,-35){\scriptsize $\lambda-\lambda^*$}
            \put(-40,280){\rotatebox[]{90}{\scriptsize $AC(\tau)$}}
        \end{overpic}
    \end{subfigure}
    \par \vspace{0.5cm}
    \begin{subfigure}{0.45\textwidth}
        \centering
        \begin{overpic}[width=\linewidth]{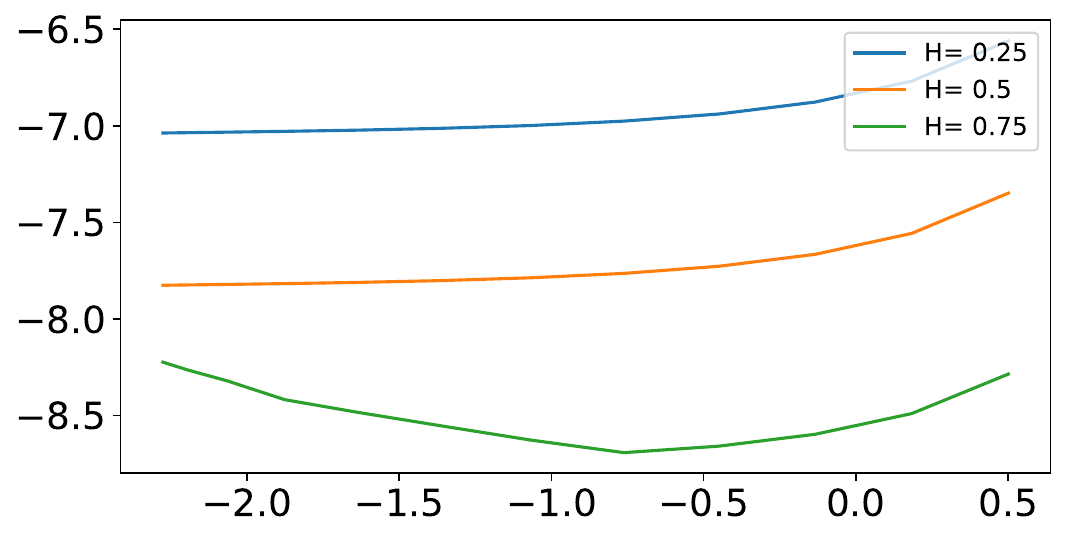}
            \put(-50,520){\large (c)}
            \put(430,-35){\scriptsize $\log_{10}(\lambda-\lambda^*)$}
            \put(-40,250){\rotatebox[]{90}{\scriptsize $\log_{10}(\mathtt{S}^{\max})$}}
        \end{overpic}
    \end{subfigure}
    \hfill
    \begin{subfigure}{0.45\textwidth}
        \centering
        \begin{overpic}[width=\linewidth]{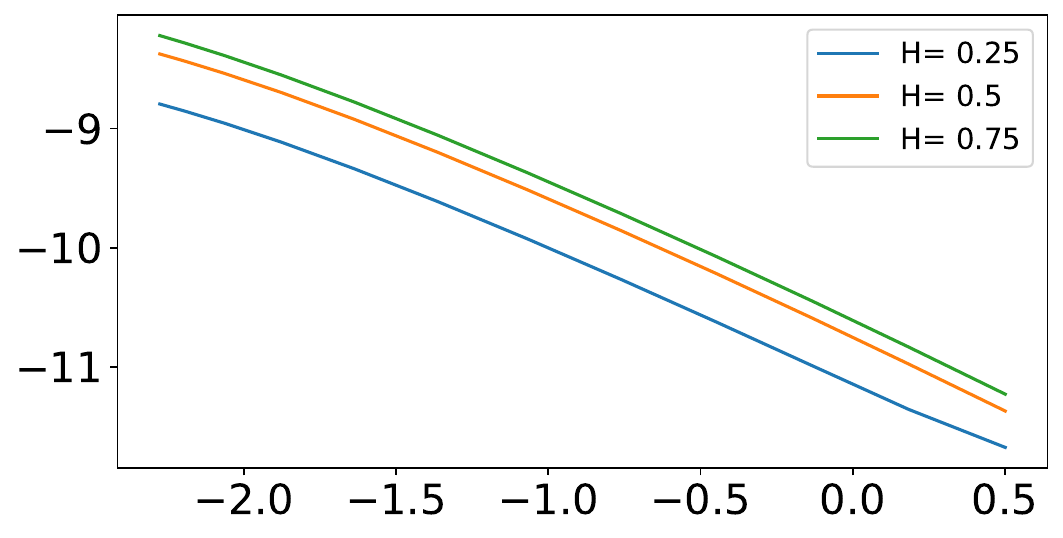}
            \put(-55,520){\large (d)}
            \put(430,-35){\scriptsize $\log_{10}(\lambda-\lambda^*)$}
            \put(-40,250){\rotatebox[]{90}{\scriptsize $\log_{10}(\mathtt{S}(\delta))$}}
        \end{overpic}
    \end{subfigure}
    \caption{(Simulation parameters $T=2^{11},N=2^{15},\text{samples}=20$) Comparing different EWSs for the fOU noise AMOC system \eqref{eq:AMOC fOU noise}. In each plot we have $H=0.25$ in blue, $H=0.5$ in orange and $H=0.75$ in green. In (a), we see the log-variance $\log_{10} (\mathtt{V}(0))$ as the solid line and the log-autocovariance with lag $\tau=0.5$, i.e. $\log_{10} (\mathtt{V}(0.5))$, as the dotted line against the log distance of the bifurcation parameter $\lambda$ from the critical value $\lambda^*$. In (b), we have the autocorrelation for lag $\tau=0.125$, i.e. $\mathtt{AC}(0.125)$, as the solid line and the autocorrelation for lag $\tau=0.5$, i.e $\mathtt{AC}(0.5)$, as the dotted line against the distance of the bifurcation parameter $\lambda$ from the critical value $\lambda^*$. In (c), we plot the log of the maximum of the SD $\log_{10} (\mathtt{S}^{\max})$ against the log distance from the bifurcation point. In (d), we see the log of the SD at a frequency close to 0, i.e. $\log_{10} (\mathtt{S}(\delta))$, against the log distance from the bifurcation point.}
    \label{fig:AMOC fOU noise EWS}
\end{figure}

The correlations introduced by the fOU noise in \eqref{eq:AMOC fOU noise} have a drastic impact on the EWSs. As discussed in \cite{kuehn2022warning}, this type of noise induces the color blindness effect in the sense that the variance does not diverge and therefore cannot be used as an EWS for the approaching bifurcation. This is consistent with Figure \ref{fig:AMOC fOU noise EWS} (a). In corroboration of \eqref{eq: plot twist}, the autocorrelation in (b) does not converge to 1. Likewise, the maximum of the SD $\mathtt{S}^{\max}$ in (c) does not diverge. All these statistics experience masking due to the fOU noise. In contrast, the SD at a fixed frequency $\delta$ close to zero in (d) shows a power-law behavior that can be used as an EWS. Since $S(0)=0$ for all $H$, choosing a frequency close to zero for e.g. $\delta=0.4$ entails the expected slope of $-2$. 
For the system \eqref{eq:AMOC fOU noise} the SD undergoes reddening as well, but we do not observe a single peak at zero but rather two peaks in its proximity (Figure \ref{fig:SD AMOC fOU noise}).
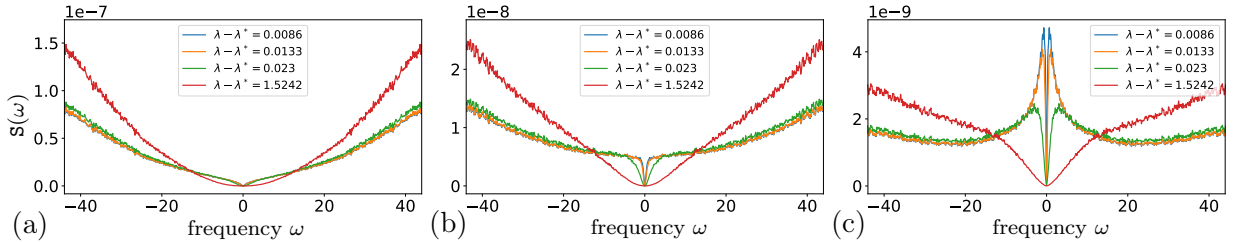
\begin{figure}[h!]
    \centering
    \subfloat{\begin{overpic}[scale=0.30]{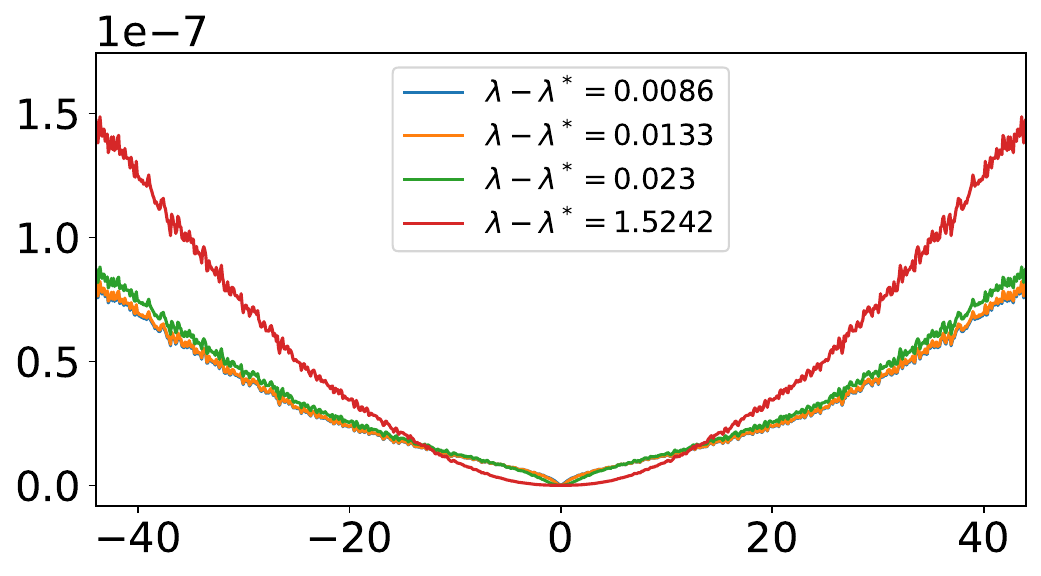}
    \put(-40,-30){(a)}
    \put(400,-35){\scriptsize frequency $\omega$}
    \put(-50,250){\rotatebox[]{90}{\scriptsize $\mathtt{S}(\omega)$}}
    \end{overpic}}
    \hspace{1mm}
    \subfloat{\begin{overpic}[scale=0.30]{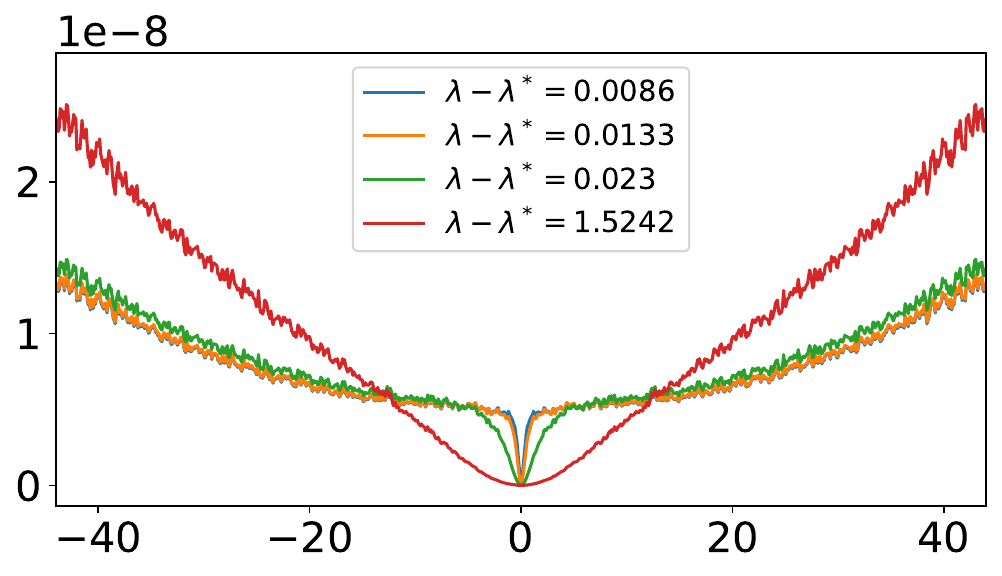}
    \put(-40,-30){(b)}
    \put(380,-35){\scriptsize frequency $\omega$}
    \end{overpic}}
    \hspace{1mm}
    \subfloat{\begin{overpic}[scale=0.30]{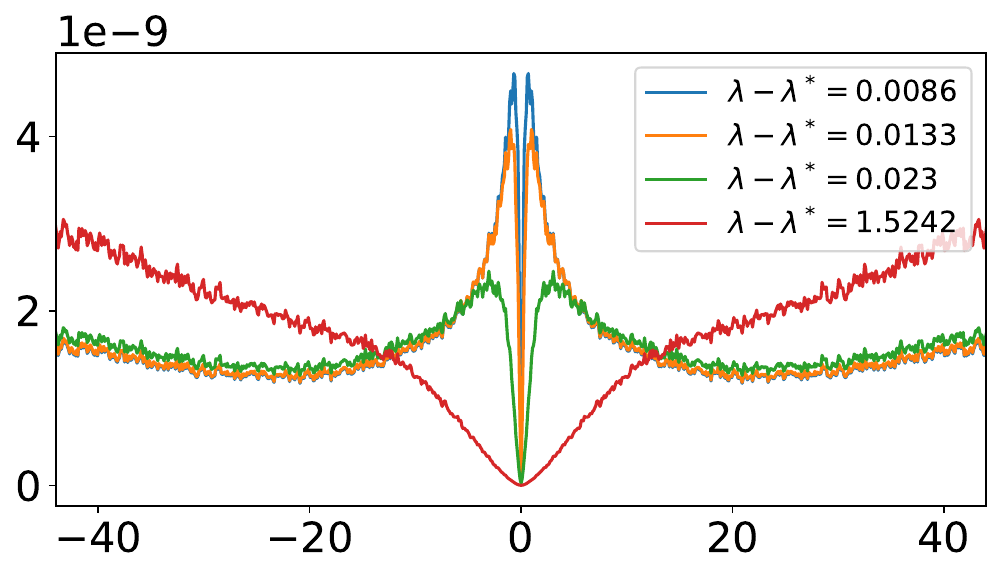}
    \put(-40,-30){(c)}
    \put(380,-35){\scriptsize frequency $\omega$}
    \end{overpic}}
    \caption{Spectral densities for $H=0.25$ (a), $H=0.5$ (b) and $H=0.75$ (c) for \eqref{eq:AMOC fOU noise}.}
    \label{fig:SD AMOC fOU noise}
\end{figure}

These plots are very instructive to understand why we observe the color blindness. The SD is given by
$$S(\omega)=C_H\sigma^2\frac{|\omega|^3}{\mu^2+\omega^2}|\omega|^{-2H}\frac{1}{|A(\lambda)|^2+\omega^2},$$
where $A(\lambda)$ is the linear part from the linearization.
The mitigation part of the SD $\frac{1}{|A(\lambda)|^2+\omega^2}$ leads to a peak at zero as we approach the bifurcation. However, due to the fOU noise the SD is zero in zero. Then, as the frequency at which the mitigation is acting the strongest becomes null, this leads to the color blindness effect. The maximum of the SD is also not directly applicable as an EWS, since the peaks are caused by noise and not the bifurcation itself. In order to capture the approaching bifurcation, a frequency close to zero, but not equal to zero, has to be observed.

\subsection{Hopf bifurcations \& limit cycles} \label{sec:applications_hopf}
 We study a Hopf bifurcation in the following system driven by fBm
\begin{align}\label{eq:Hopf fBm}
    \begin{split}
        \txtd \mathbf{x}_t &= \begin{pmatrix}
            y_t & -\omega_0 \\ \omega_0 & y_t
        \end{pmatrix} \mathbf{x}_t - \left| \mathbf{x}_t \right|^2 \mathbf{x}_t~\txtd t + \begin{pmatrix}
            \sigma_1 & \sigma_2
        \end{pmatrix} ~\txtd \textbf{W}_t^{H}, \\
        \txtd y_t &= \eps~\txtd t,~~
        \textbf{x}_0=\begin{pmatrix}
            0.01 \\ 0.01
        \end{pmatrix},~~ y_0<0,
    \end{split}
\end{align}
for $t\geq 0$, $\eps\ll1$ and $n=2$. In Figure \ref{fig:Hopf trajectory different sigma}, we show examples of trajectories for $\eps>0, \sigma_1=0.01,\sigma_2=0.05$ depending on the Hurst index. The trajectory remains in the vicinity of the null state until $y$ crosses the Hopf bifurcation threshold $\lambda^*=0$. Then, they depart from the origin following a limit cycle, whose radius increases with $y$. In fact, setting $\eps=0$, the system has the unique stable equilibrium point in the origin for negative $y=\lambda$. For $\lambda>0$ the system has a stable limit cycle with radius $\lambda$.
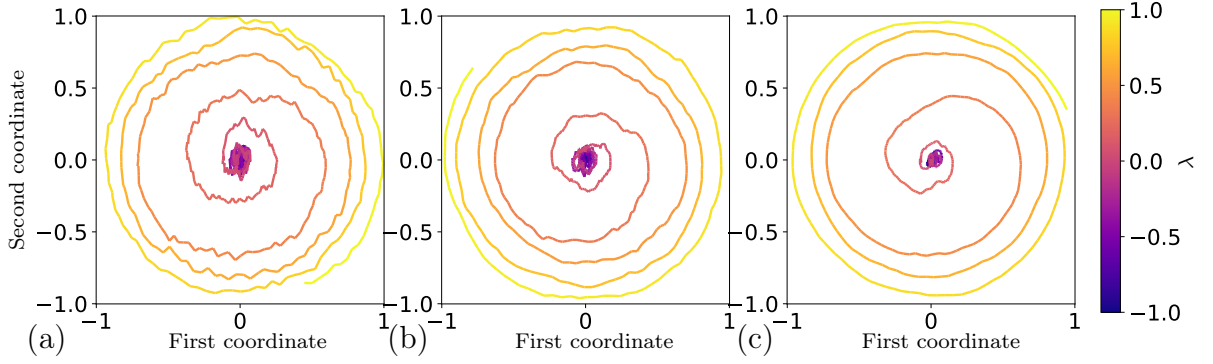
\begin{figure}[h!]
    \centering
    \subfloat{\begin{overpic}[scale=0.45]{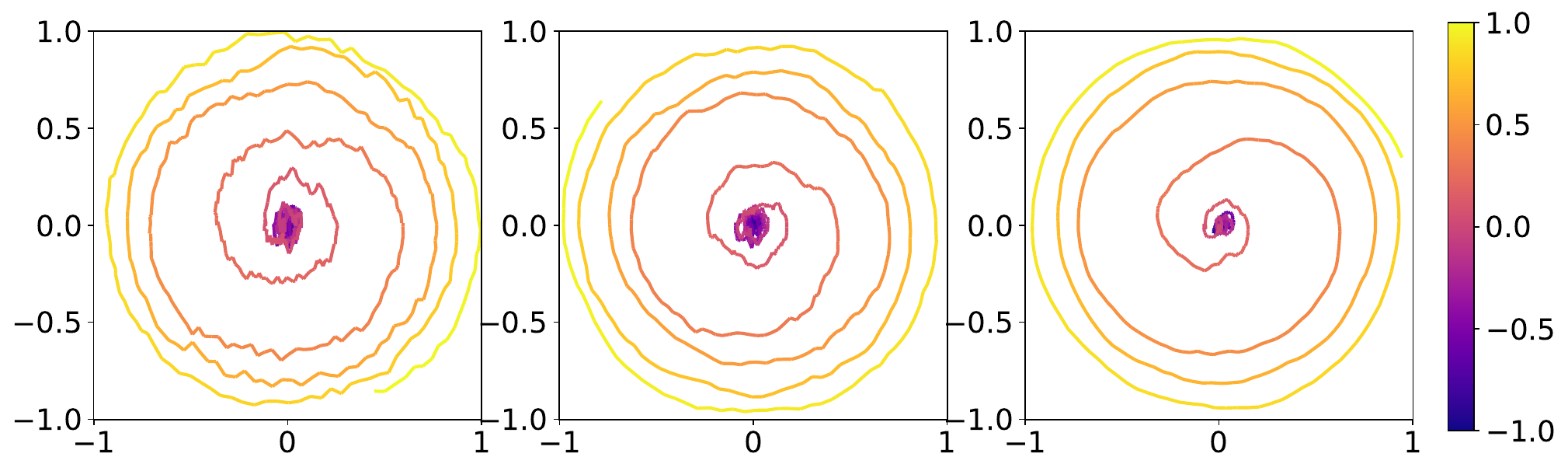}
    \put(0,-5){\large (a)}
    \put(122,-5){\scriptsize First coordinate}
    \put(-15,150){\rotatebox[]{90}{\scriptsize Second coordinate}}
    \put(310,-5){\large (b)}
    \put(420,-5){\scriptsize First coordinate}
    \put(610,-5){\large (c)}
    \put(718,-5){\scriptsize First coordinate}
    \put(990,148){\rotatebox[]{90}{\scriptsize $\lambda$}}
    \end{overpic}}
    \caption{Example trajectories of $\mathbf{x}$ that solves \eqref{eq:Hopf fBm} with $\eps=2^{-5},y_0=-1,$ $\sigma_1=0.01, \sigma_2=0.05$ for $H=0.25$ (a), $H=0.5$ (b) and $H=0.75$ (c). Each axis indicates a single coordinate.}
    \label{fig:Hopf trajectory different sigma}
\end{figure}
In the remainder of the subsection, we study the EWSs on the fast subsystem, i.e. \eqref{eq:Hopf fBm} with $\eps=0$.
\begin{figure}[h!]
\centering
\begin{subfigure}{0.31\textwidth}
    \centering
    \begin{overpic}[width=\linewidth]{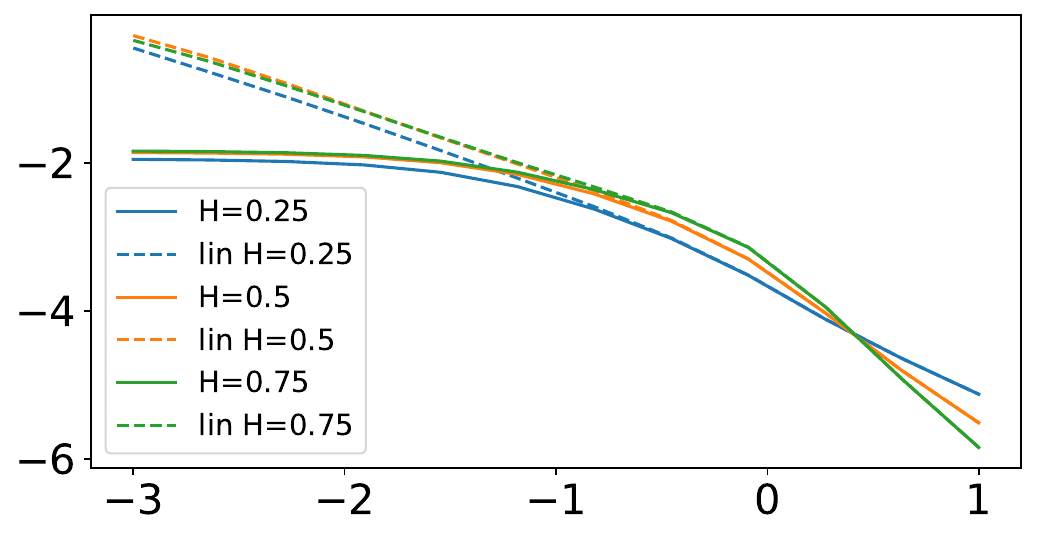}
        \put(-50,525){\large (a)}
        \put(380,-35){\scriptsize $\log_{10}(\lambda-\lambda^*)$}
        \put(-75,250){\rotatebox[]{90}{\scriptsize $\log_{10}(\mathtt{V}(0)_{11})$}}
    \end{overpic}
\end{subfigure}
\hfill
\begin{subfigure}{0.31\textwidth}
    \centering
    \begin{overpic}[width=\linewidth]{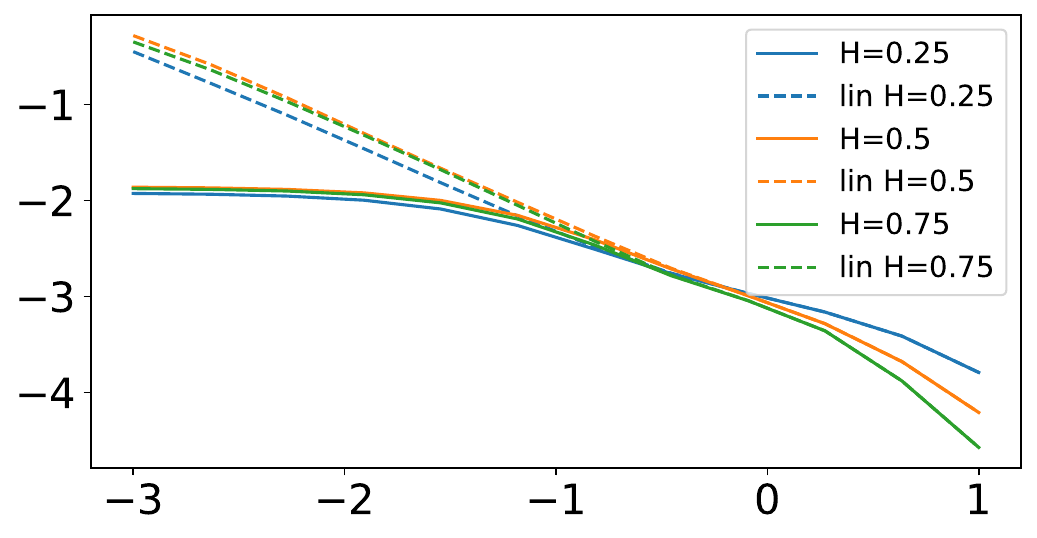}
        \put(-50,525){\large (b)}
        \put(380,-35){\scriptsize $\log_{10}(\lambda-\lambda^*)$}
        \put(-75,250){\rotatebox[]{90}{\scriptsize $\log_{10}(\mathtt{V}(0)_{22})$}}
    \end{overpic}
\end{subfigure}
\hfill
\begin{subfigure}{0.31\textwidth}
    \centering
    \begin{overpic}[width=\linewidth]{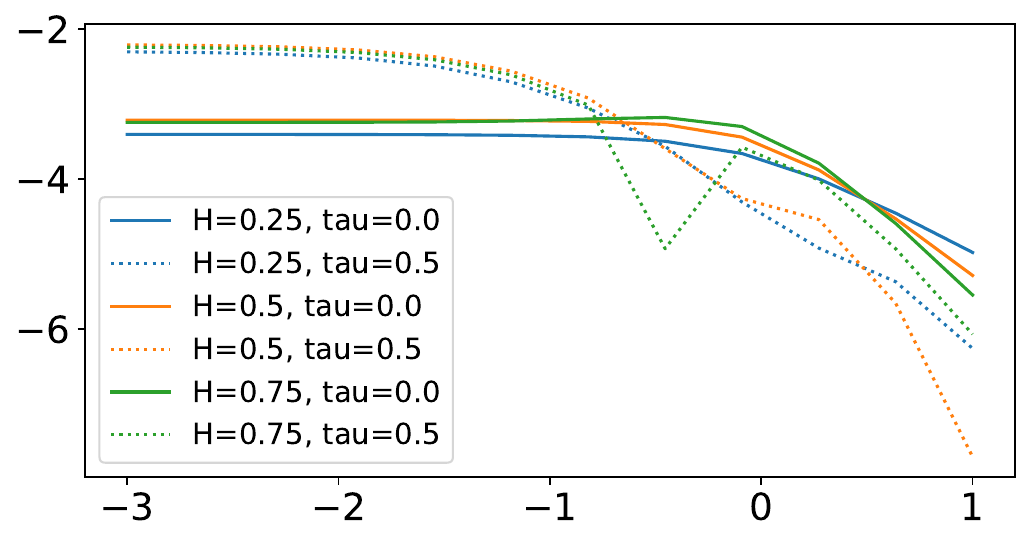}
        \put(-50,525){\large (c)}
        \put(380,-35){\scriptsize $\log_{10}(\lambda-\lambda^*)$}
        \put(-75,250){\rotatebox[]{90}{\scriptsize $\log_{10}(\mathtt{V}(\tau)_{12})$}}
    \end{overpic}
\end{subfigure}

\par \vspace{0.3cm}

\begin{subfigure}{0.31\textwidth}
    \centering
    \begin{overpic}[width=\linewidth]{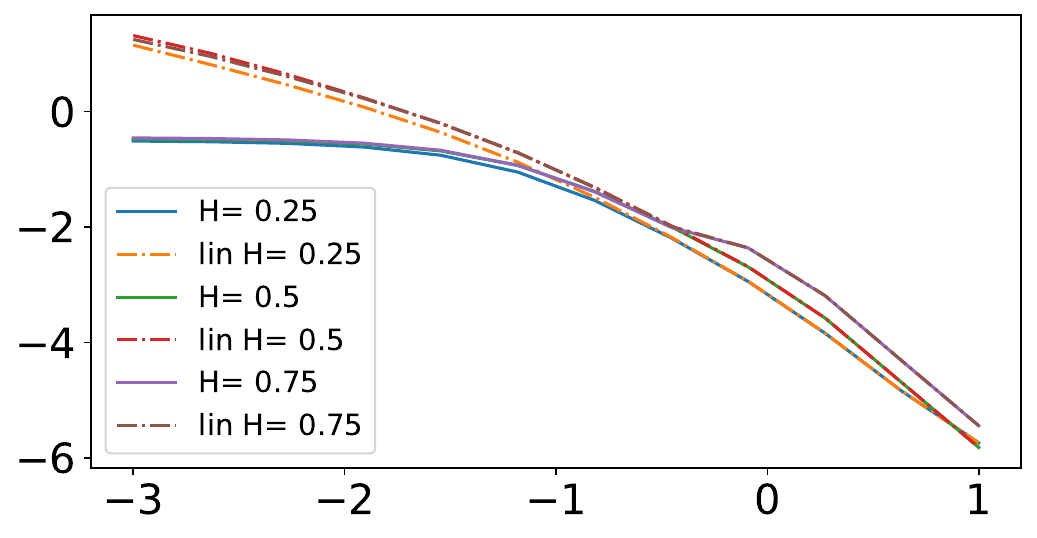}
        \put(-50,525){\large (d)}
        \put(380,-35){\scriptsize $\log_{10}(\lambda-\lambda^*)$}
        \put(-75,250){\rotatebox[]{90}{\scriptsize $\log_{10}(\max\limits_{\omega} \mathtt{S}^{11})$}}
    \end{overpic}
\end{subfigure}
\hfill
\begin{subfigure}{0.31\textwidth}
    \centering
    \begin{overpic}[width=\linewidth]{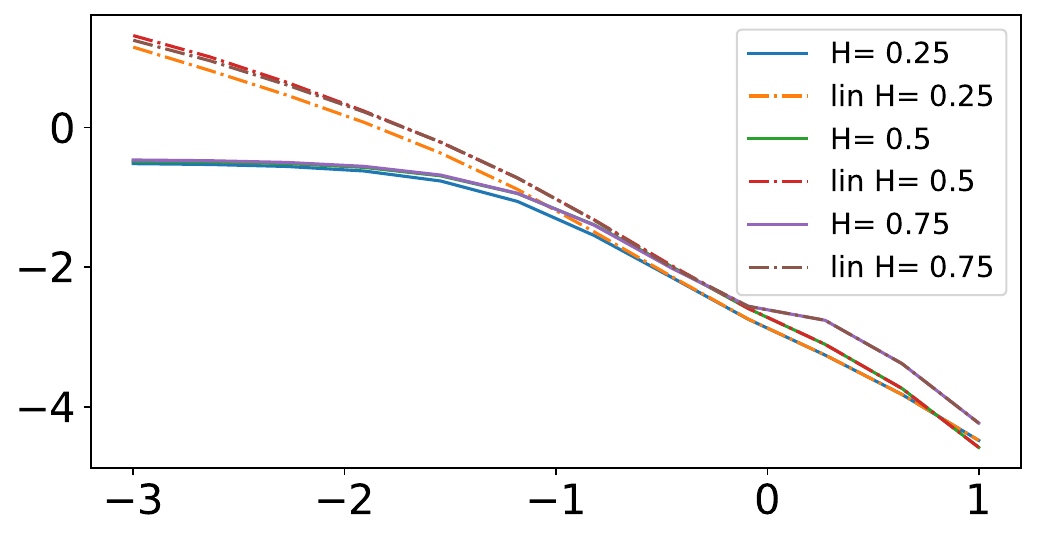}
        \put(-50,525){\large (e)}
        \put(380,-35){\scriptsize $\log_{10}(\lambda-\lambda^*)$}
        \put(-75,250){\rotatebox[]{90}{\scriptsize $\log_{10}(\max\limits_{\omega} \mathtt{S}^{22})$}}
    \end{overpic}
\end{subfigure}
\hfill
\begin{subfigure}{0.31\textwidth}
    \centering
    \begin{overpic}[width=\linewidth]{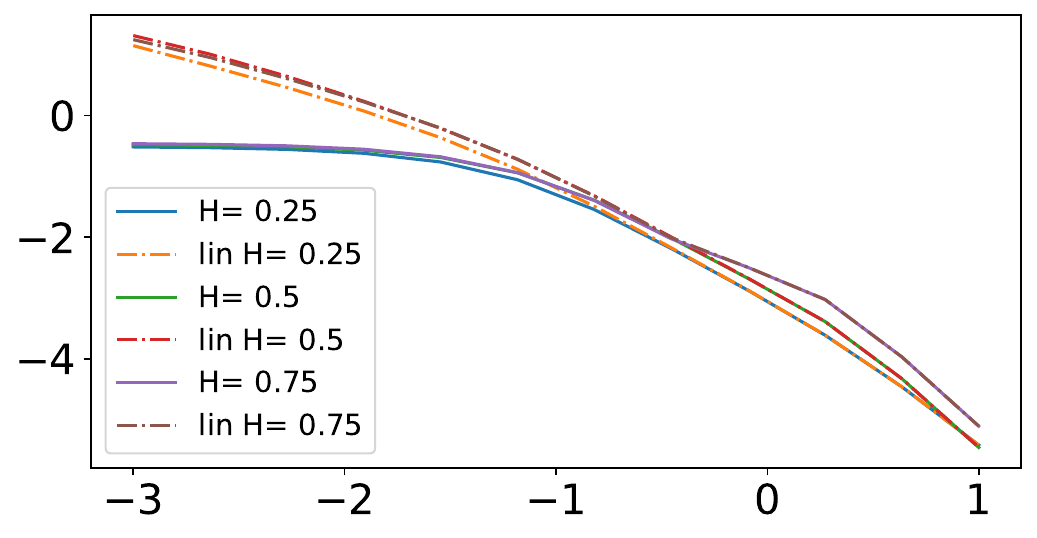}
        \put(-50,525){\large (f)}
        \put(380,-35){\scriptsize $\log_{10}(\lambda-\lambda^*)$}
        \put(-75,250){\rotatebox[]{90}{\scriptsize $\log_{10}(\max\limits_{\omega} \mathtt{S}^{12})$}}
    \end{overpic}
\end{subfigure}

\par \vspace{0.3cm}

\begin{subfigure}{0.31\textwidth}
    \centering
    \begin{overpic}[width=\linewidth]{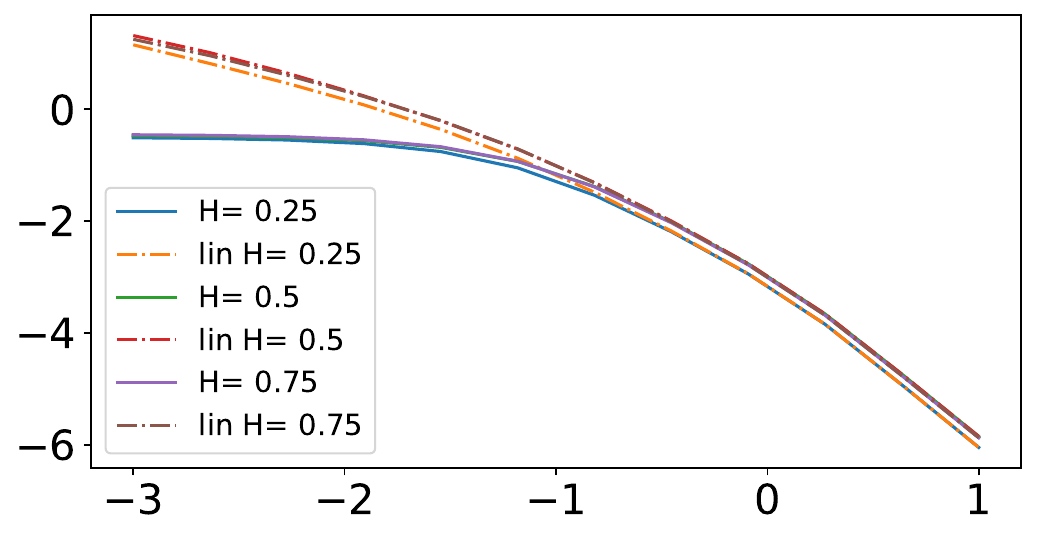}
        \put(-50,525){\large (g)}
        \put(380,-35){\scriptsize $\log_{10}(\lambda-\lambda^*)$}
        \put(-75,250){\rotatebox[]{90}{\scriptsize $\log_{10}(\mathtt{S}^{11}(\omega_0))$}}
    \end{overpic}
\end{subfigure}
\hfill
\begin{subfigure}{0.31\textwidth}
    \centering
    \begin{overpic}[width=\linewidth]{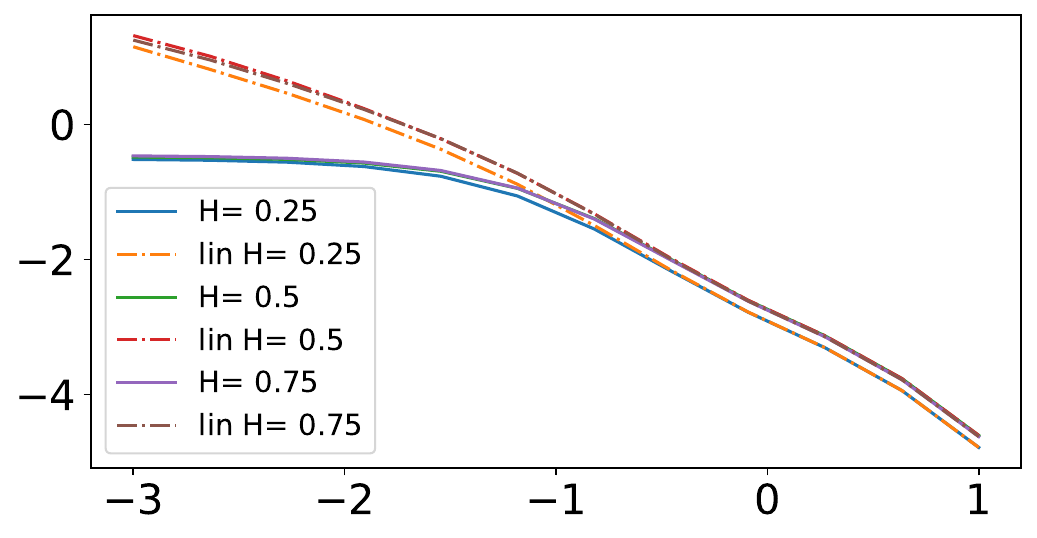}
        \put(-50,525){\large (h)}
        \put(380,-35){\scriptsize $\log_{10}(\lambda-\lambda^*)$}
        \put(-75,250){\rotatebox[]{90}{\scriptsize $\log_{10}(\mathtt{S}^{22}(\omega_0))$}}
    \end{overpic}
\end{subfigure}
\hfill
\begin{subfigure}{0.31\textwidth}
    \centering
    \begin{overpic}[width=\linewidth]{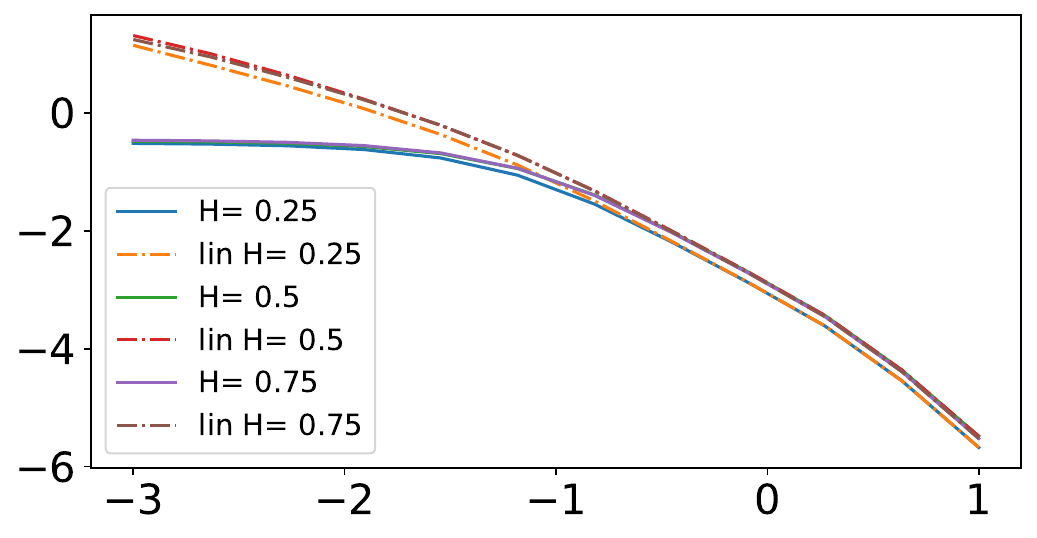}
        \put(-50,525){\large (i)}
        \put(380,-35){\scriptsize $\log_{10}(\lambda-\lambda^*)$}
        \put(-75,250){\rotatebox[]{90}{\scriptsize $\log_{10}(\mathtt{S}^{12}(\omega_0))$}}
    \end{overpic}
\end{subfigure}

\caption{
(Simulation parameters $T=2^{12},N=2^{16},\text{samples}=8$) Plots of the different EWSs for the Hopf system driven by fBm \eqref{eq:Hopf fBm} with $\sigma_1=0.01,\sigma_2=0.05$. Each column correspond respectively to the first component, second component and mixed component. (a)--(c): Autocovariance in a loglog plot against the bifurcation parameter. For (a) and (b) the solid lines are the variances for the nonlinear case and the dashed lines for the linearization. For (c) the solid lines refer to lag $\tau=0$ and the dotted lines correspond to lag $\tau=0.5$.
(d)--(f): The maximum of the SD in a loglog plot against the bifurcation parameter. Solid lines are from the nonlinear simulation and dashed lines from the linearized system.
(g)--(i): The SD at the peak frequency $\omega_0$, i.e. $\mathtt{S}(\omega_0)$, in a loglog plot against the bifurcation parameter. Solid lines are from the nonlinear simulation and dashed lines from the linearization.
}
\label{fig:Hopf fBm EWSs}
\end{figure}
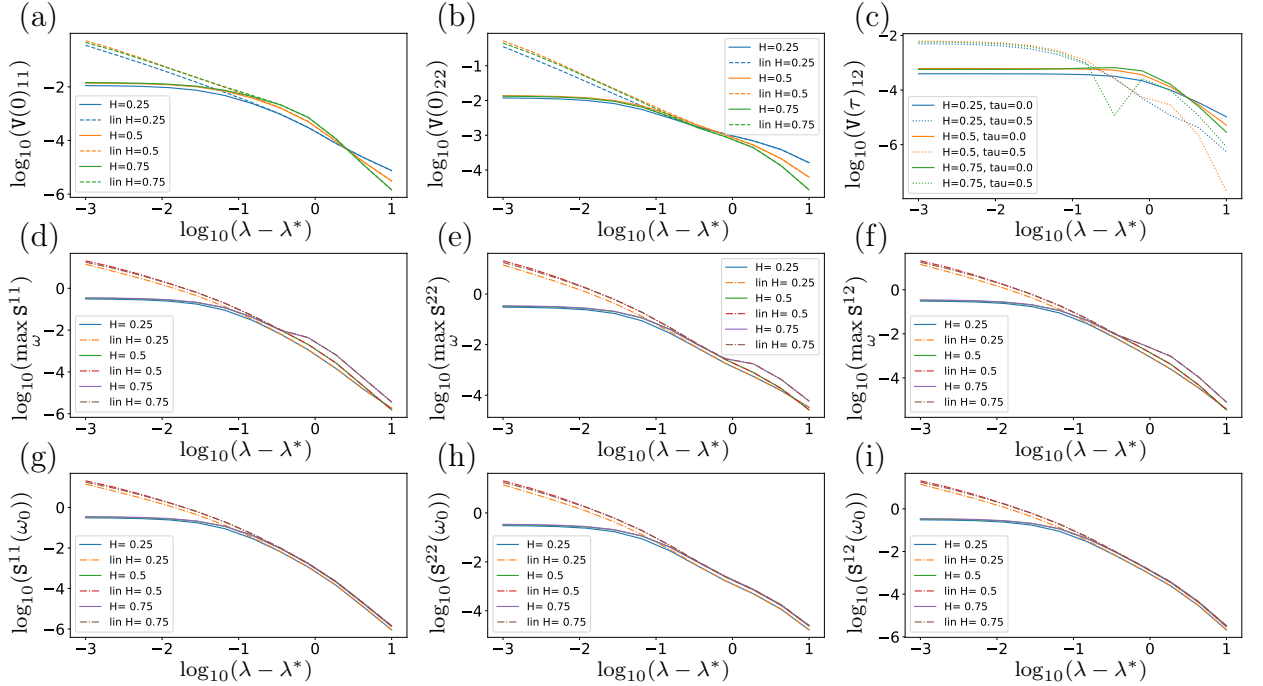
We start with the autocovariance, which can be seen in the first row of Figure \ref{fig:Hopf fBm EWSs} as loglog plots. As discussed in Section~\ref{sec:fbm 2d} and summarized in Table \ref{tab:X2}, the components along all modes of the autocovariance diverge with rate of slope $-1$ for all $H\in(0,1)$ and almost all $\tau\ge0$. However, we can differentiate three regions. We start with plot (a). Furthest away from the bifurcation the slope behaves as in the one-dimensional case, meaning the slope is $-2H$. In this region, the system does not detect the rotation. As the real part of the eigenvalue becomes of order of magnitude similar to the imaginary part ($\omega_0=1$) the slopes become parallel and are equal to $-1$. Last, the nonlinear forcing affects and dampens the slope. This can be clearly seen as the linearized slope still increases and the difference between the nonlinear slope and the linear slope gets larger. \\
In plot (b) we find the same regions, but the region with slope $-1$ is smaller. Hence, the noise intensity influences the change between the three regimes. Lastly, we discuss the mixed mode in (c). Here, the middle regime with slope $-1$ is completely missing. The reason is that $\tau=0$ is one of the values where the two modes cancel the autocovariance \eqref{eq: civil war}. For the dashed line ($\tau=0.5$)
we can spot the intermediate regime with slope $-1$.\\
Even though we do not consider the maximum of the SD analytically, we can see in the second row that it is still a good EWS along all modes. This becomes much clearer after seeing that the peaks are also the local maxima for the fBm case (Figure \ref{fig:Hopf fBm SD}). Finally, in the third row the SD at the peak frequency has a clear divergence rate as computed in Theorem \ref{thm:psd_on_elements}. In the proximity to the bifurcation threshold, the curve flattens for both the maximum of SD and the SD at the peak frequency due to the nonlinearity of the system. \\
To help with the visualization, we plot the SD along the first component in Figure \ref{fig:Hopf fBm SD} to observe the two clear peaks at the frequencies $\pm\omega_0$ for all $H\in(0,1)$ and how the different Hurst indices $H$ change its behavior at zero.
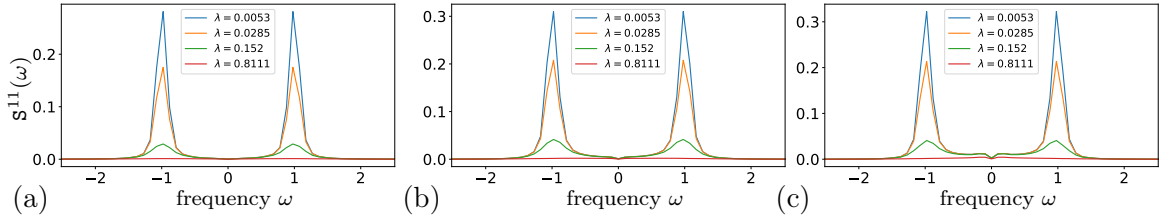
\begin{figure}[h!]
    \centering
    \subfloat{\begin{overpic}[scale=0.28]{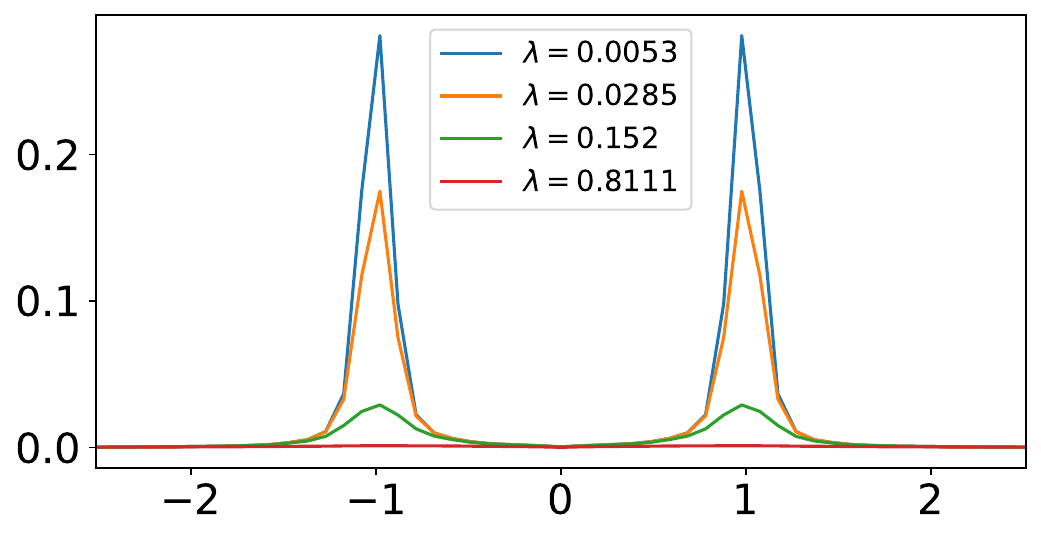}
    \put(-40,-40){(a)}
    \put(400,-35){\scriptsize frequency $\omega$}
    \put(-50,260){\rotatebox[]{90}{\scriptsize $\mathtt{S}^{11}(\omega)$}}
    \end{overpic}}
    \hspace{1mm}
    \subfloat{\begin{overpic}[scale=0.28]{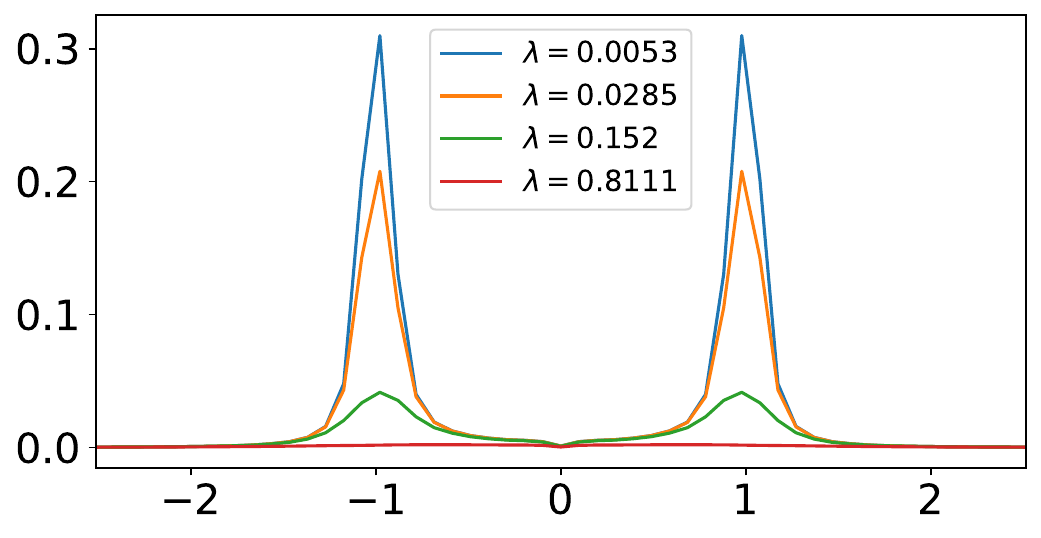}
    \put(-40,-40){(b)}
    \put(400,-35){\scriptsize frequency $\omega$}
    \end{overpic}}
    \subfloat{\begin{overpic}[scale=0.28]{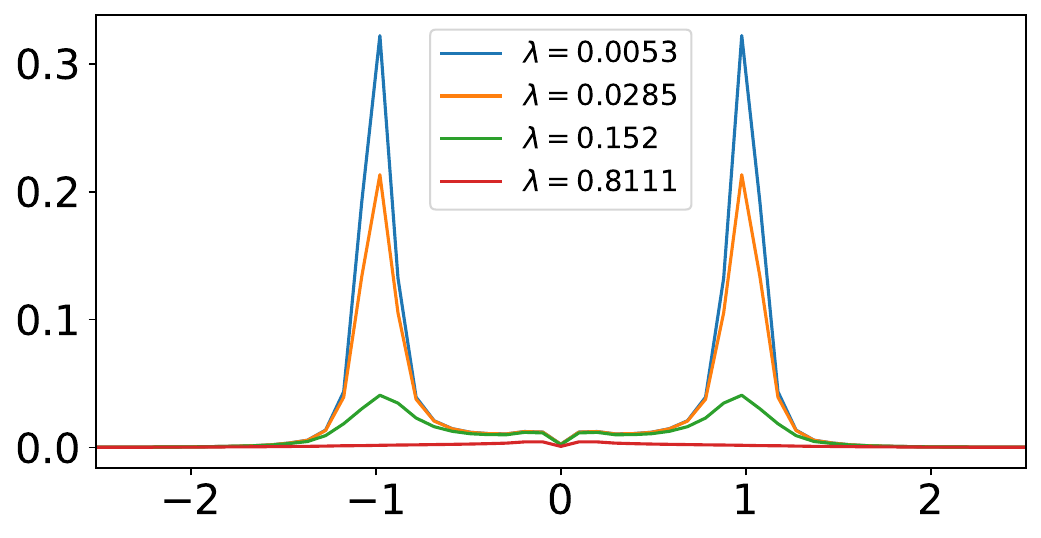}
    \put(-40,-40){(c)}
    \put(400,-35){\scriptsize frequency $\omega$}
    \end{overpic}}
    \caption{The first component of SD $\mathtt{S}(\omega)$ of the Hopf system driven by fBm \eqref{eq:Hopf fBm} for (a): $H=0.25$, (b): $H=0.5$ and (c): $H=0.75$.}
    \label{fig:Hopf fBm SD}
\end{figure}

After studying the Hopf system driven by a fBm, we consider the Hopf fast subsystem driven by fOU noise
\begin{align}\label{eq:Hopf fOU}
    \begin{split}
        \txtd \mathbf{x}_t &= \begin{pmatrix}
            \lambda & -\omega_0 \\ \omega_0 & \lambda
        \end{pmatrix} \mathbf{x}_t - \left| \mathbf{x}_t \right|^2 \mathbf{x}_t~\txtd t + \begin{pmatrix}
            \sigma_1 \\ \sigma_2
        \end{pmatrix} ~\txtd Z_t^{H}, \\
        \lambda&<0, \quad \textbf{x}_0=\begin{pmatrix}
            0.01 \\ 0.01
        \end{pmatrix},
    \end{split}
\end{align}
where $Z$ is a stationary solution of \eqref{eq: ou forcing}.
\begin{figure}[h!]
\centering
\begin{subfigure}{0.31\textwidth}
    \centering
    \begin{overpic}[width=\linewidth]{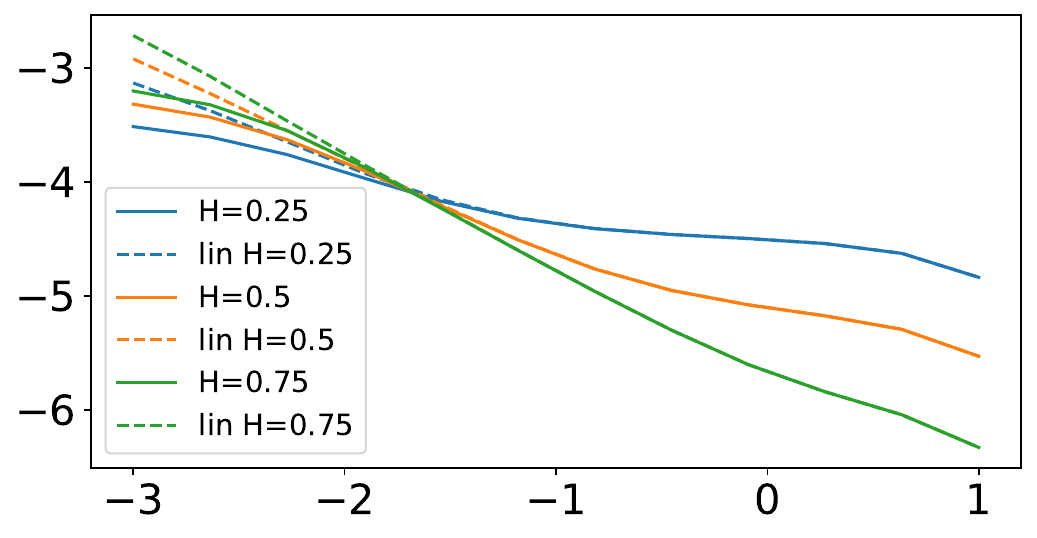}
        \put(-50,525){\large (a)}
        \put(380,-35){\scriptsize $\log_{10}(\lambda-\lambda^*)$}
        \put(-75,250){\rotatebox[]{90}{\scriptsize $\log_{10}(\mathtt{V}(0)_{11})$}}
    \end{overpic}
\end{subfigure}
\hfill
\begin{subfigure}{0.31\textwidth}
    \centering
    \begin{overpic}[width=\linewidth]{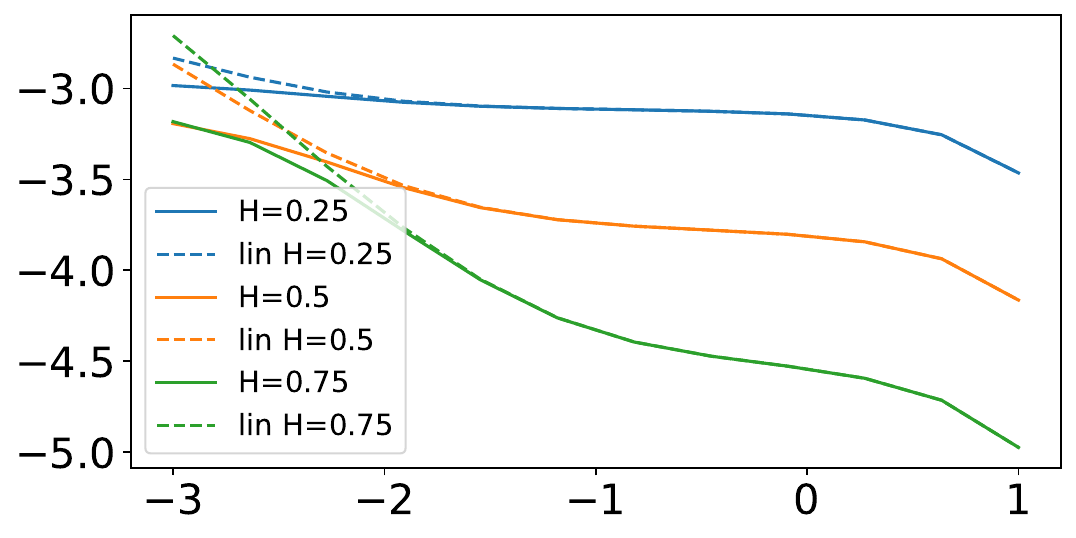}
        \put(-50,525){\large (b)}
        \put(380,-35){\scriptsize $\log_{10}(\lambda-\lambda^*)$}
        \put(-75,250){\rotatebox[]{90}{\scriptsize $\log_{10}(\mathtt{V}(0)_{22})$}}
    \end{overpic}
\end{subfigure}
\hfill
\begin{subfigure}{0.31\textwidth}
    \centering
    \begin{overpic}[width=\linewidth]{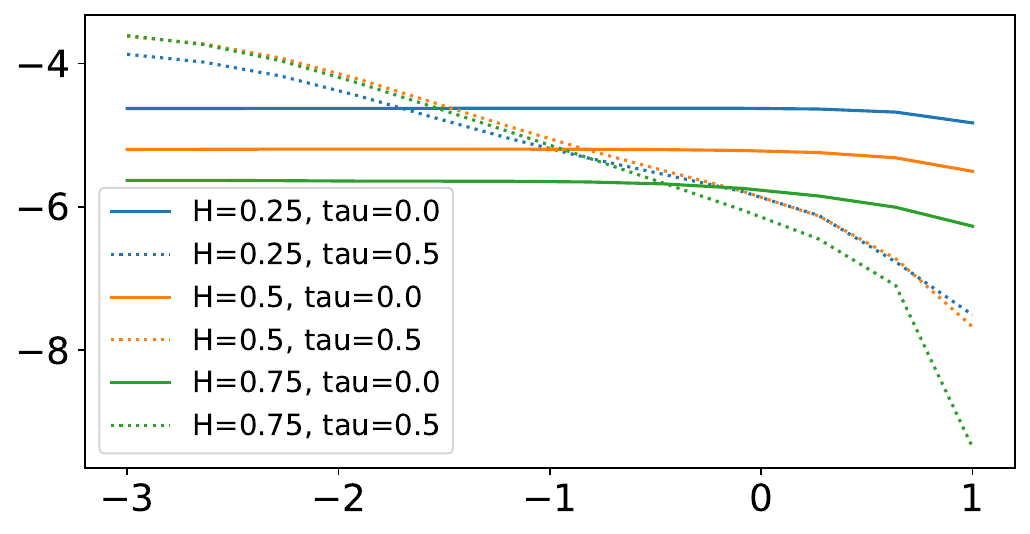}
        \put(-50,525){\large (c)}
        \put(380,-35){\scriptsize $\log_{10}(\lambda-\lambda^*)$}
        \put(-75,250){\rotatebox[]{90}{\scriptsize $\log_{10}(\mathtt{V}(\tau)_{12})$}}
    \end{overpic}
\end{subfigure}

\par \vspace{0.3cm}

\begin{subfigure}{0.31\textwidth}
    \centering
    \begin{overpic}[width=\linewidth]{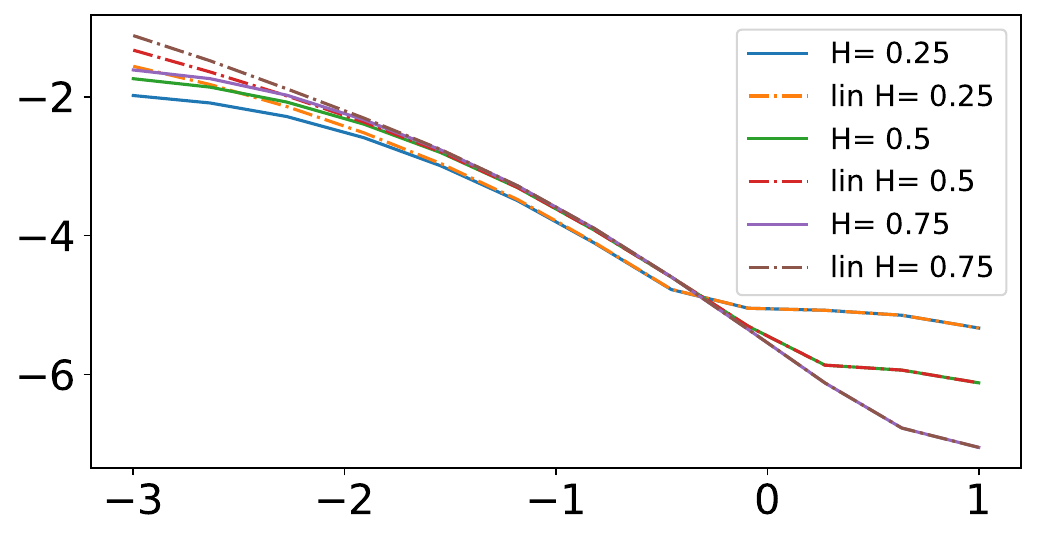}
        \put(-50,525){\large (d)}
        \put(380,-35){\scriptsize $\log_{10}(\lambda-\lambda^*)$}
        \put(-75,250){\rotatebox[]{90}{\scriptsize $\log_{10}(\max\limits_{\omega} \mathtt{S}^{11})$}}
    \end{overpic}
\end{subfigure}
\hfill
\begin{subfigure}{0.31\textwidth}
    \centering
    \begin{overpic}[width=\linewidth]{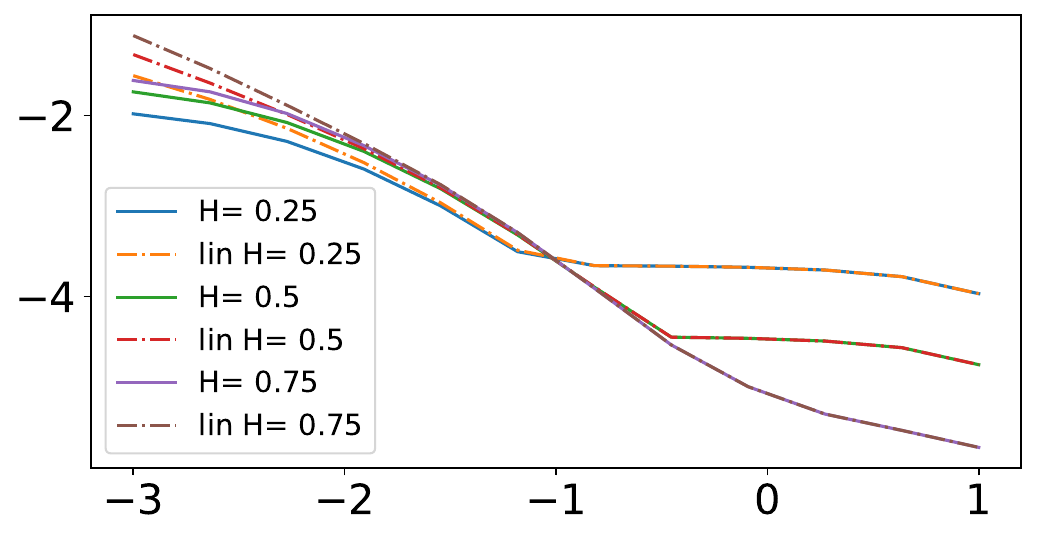}
        \put(-50,525){\large (e)}
        \put(380,-35){\scriptsize $\log_{10}(\lambda-\lambda^*)$}
        \put(-75,250){\rotatebox[]{90}{\scriptsize $\log_{10}(\max\limits_{\omega} \mathtt{S}^{22})$}}
    \end{overpic}
\end{subfigure}
\hfill
\begin{subfigure}{0.31\textwidth}
    \centering
    \begin{overpic}[width=\linewidth]{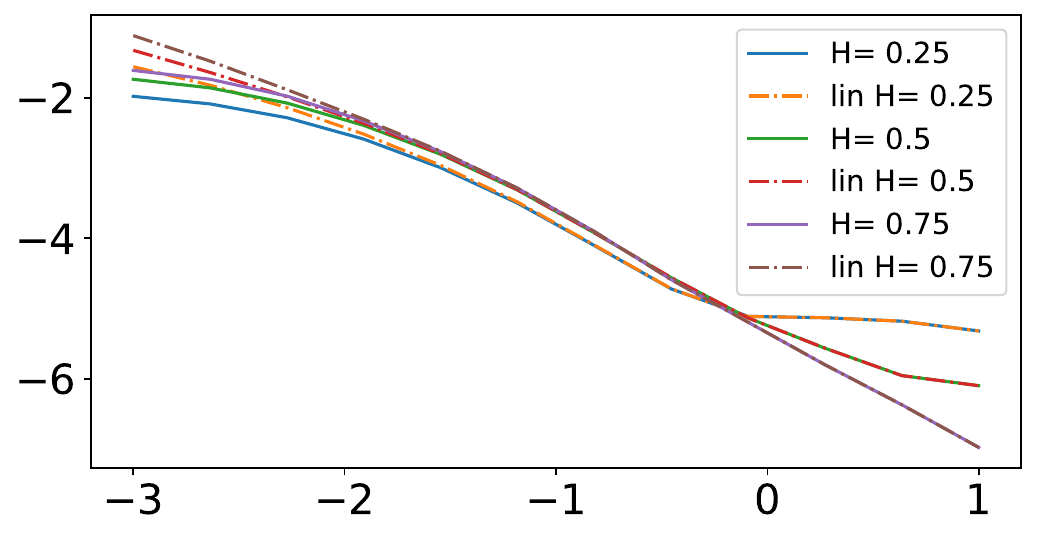}
        \put(-50,525){\large (f)}
        \put(380,-35){\scriptsize $\log_{10}(\lambda-\lambda^*)$}
        \put(-75,250){\rotatebox[]{90}{\scriptsize $\log_{10}(\max\limits_{\omega} \mathtt{S}^{12})$}}
    \end{overpic}
\end{subfigure}

\par \vspace{0.3cm}

\begin{subfigure}{0.31\textwidth}
    \centering
    \begin{overpic}[width=\linewidth]{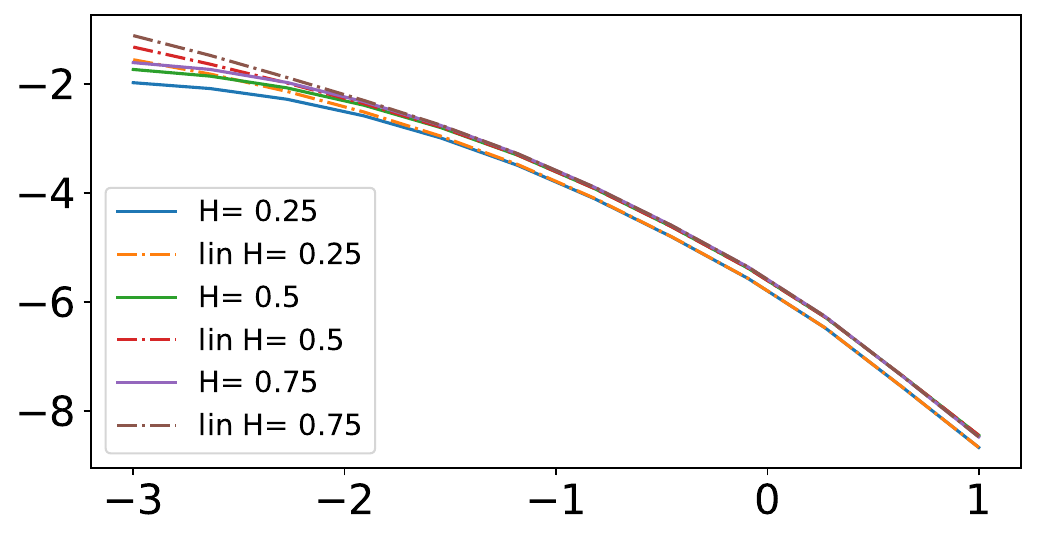}
        \put(-50,525){\large (g)}
        \put(380,-35){\scriptsize $\log_{10}(\lambda-\lambda^*)$}
        \put(-75,250){\rotatebox[]{90}{\scriptsize $\log_{10}(\mathtt{S}^{11}(\omega_0))$}}
    \end{overpic}
\end{subfigure}
\hfill
\begin{subfigure}{0.31\textwidth}
    \centering
    \begin{overpic}[width=\linewidth]{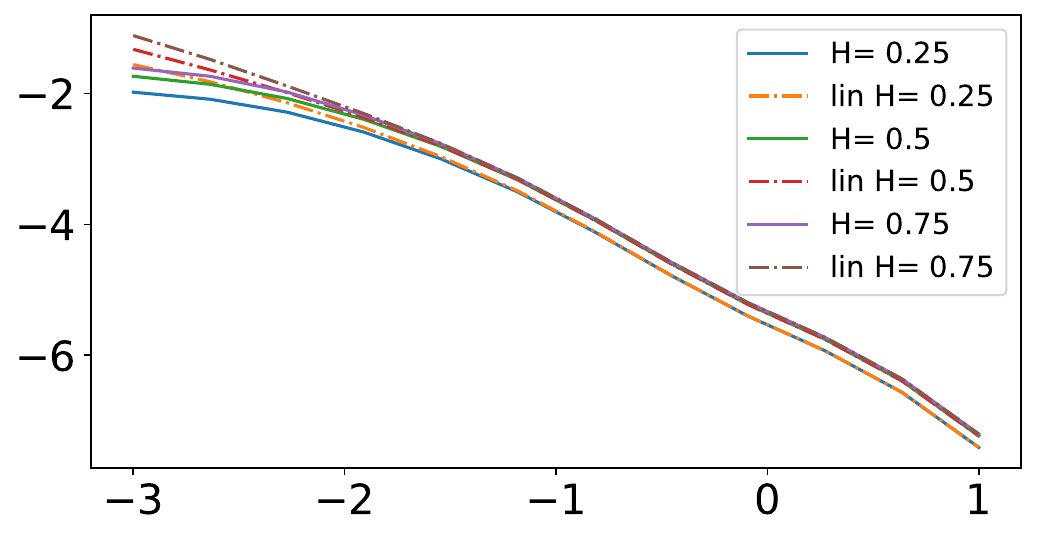}
        \put(-50,525){\large (h)}
        \put(380,-35){\scriptsize $\log_{10}(\lambda-\lambda^*)$}
        \put(-75,250){\rotatebox[]{90}{\scriptsize $\log_{10}(\mathtt{S}^{22}(\omega_0))$}}
    \end{overpic}
\end{subfigure}
\hfill
\begin{subfigure}{0.31\textwidth}
    \centering
    \begin{overpic}[width=\linewidth]{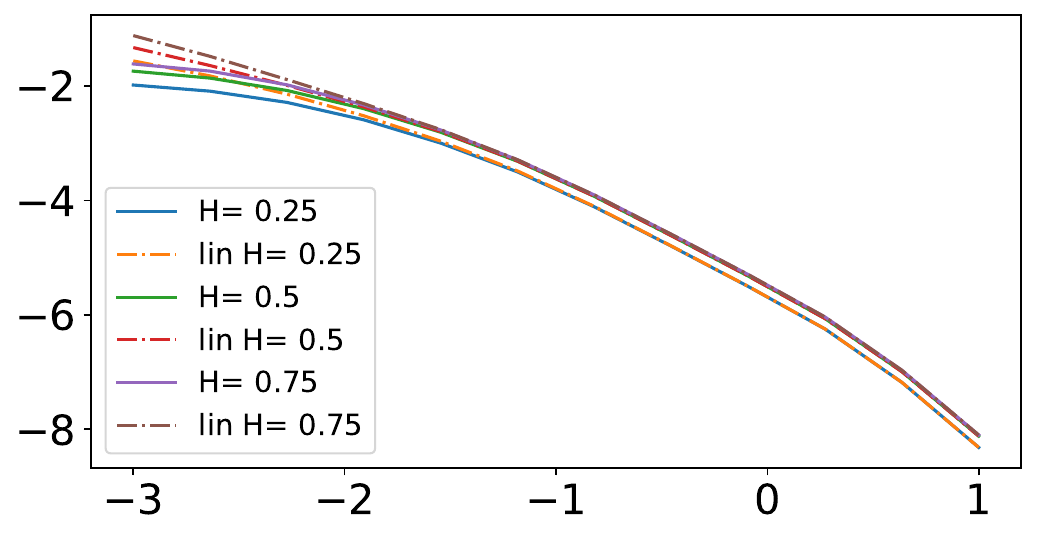}
        \put(-50,525){\large (i)}
        \put(380,-35){\scriptsize $\log_{10}(\lambda-\lambda^*)$}
        \put(-75,250){\rotatebox[]{90}{\scriptsize $\log_{10}(\mathtt{S}^{12}(\omega_0))$}}
    \end{overpic}
\end{subfigure}

\caption{
(Simulation parameters $T=2^{12},N=2^{16},\text{samples}=8$) Plots of the different EWSs for the Hopf system driven by fOU noise \eqref{eq:Hopf fOU} with $\sigma_1=0.01,\sigma_2=0.05$ and $\omega_0=1$. Each column corresponds respectively to the first component, second component and mixed component. (a)--(c): Autocovariance in a loglog plot against the bifurcation parameter. The solid lines are associated with lag $\tau=0$ and the dotted lines correspond to lag $\tau=0.5$.
(d)--(f): The maximum of the SD in a loglog plot against the bifurcation parameter.
(g)--(i): The SD at the peak frequency $\omega_0$, $\mathtt{S}(\omega_0)$ in a loglog plot against the bifurcation parameter.
}
\label{fig:Hopf fOU EWSs}
\end{figure}
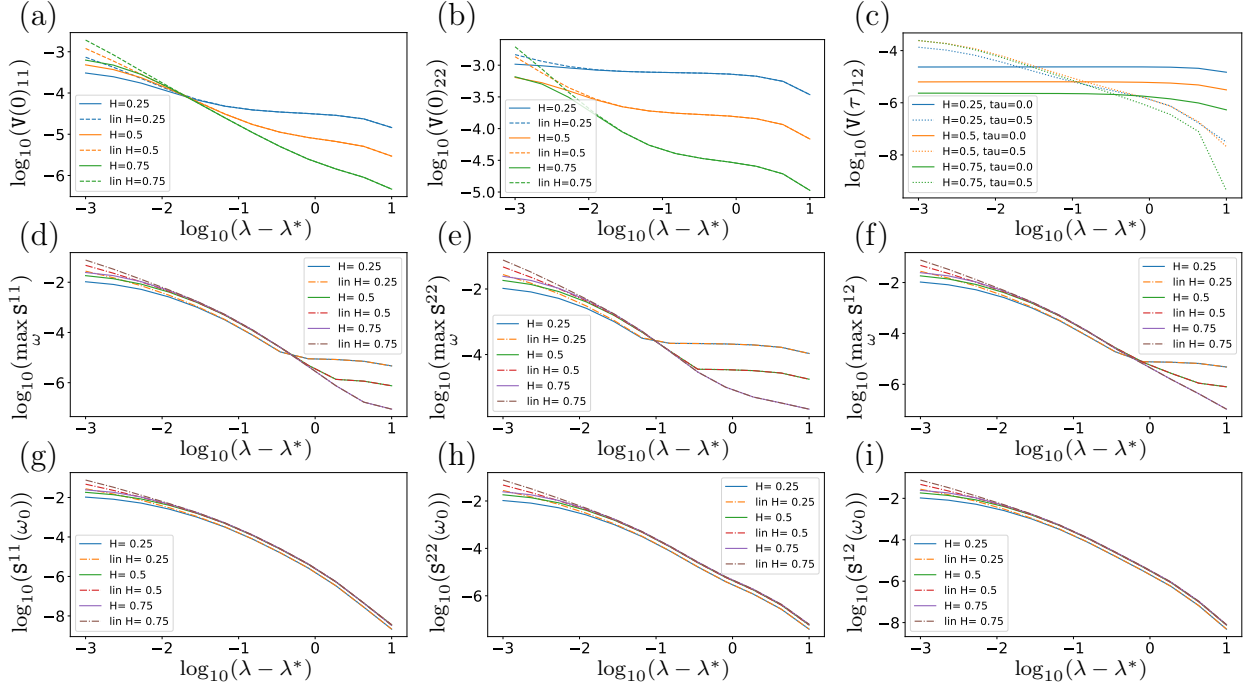
In contrast to the AMOC example above (Figure \ref{fig:AMOC fOU noise EWS}) here we do not have complete masking due to the fOU noise, as shown in Figure \ref{fig:Hopf fOU EWSs}. In plot (a), we distinguish again three regimes. In the first regime, we see masking as in the AMOC example. After that, we can spot the $-1$ slope associated by memory loss and lastly the damping due to the nonlinearity. In (b), we register a similar behavior but the regime with the slope $-1$ is much smaller. In (c), we see again the flat variance (solid line) due to the canceling of the terms in \eqref{eq: civil war}, but for $\tau=0.5$ (dotted line) we see the slope $-1$ before the damping in the last regime.\\
The maximum of the SD in the second row can still be used as an EWS, but carries the risk of warning relatively late, as can be seen in (e). This is analogous to the masking in the autocovariance. Last, we consider the SD at the peak frequency $\omega_0$. As in the fBm case (Figure \ref{fig:Hopf fBm EWSs}), $\mathtt{S}(\omega_0)$ has again a very clear divergence before the nonlinearity starts to damp the slope. In conclusion, $\mathtt{S}(\omega_0)$ seems to be very resistant to any kind of noises in the Hopf setting.\\
Finally, we consider in Figure \ref{fig:Hopf fOU SD} the plots of the SD along the first component for a better visualization of the effects of the fOU noise Figure.
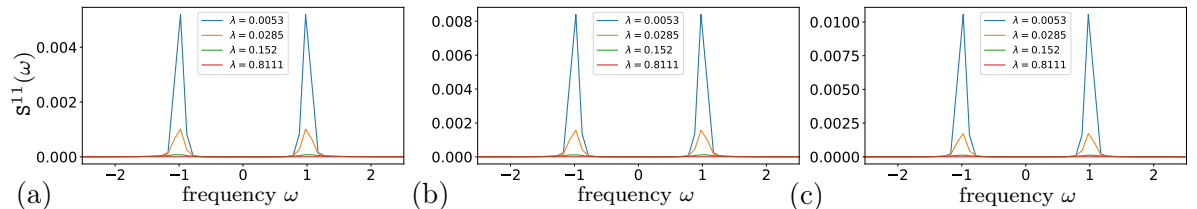
\begin{figure}[h!]
    \centering
    \subfloat{\begin{overpic}[scale=0.27]{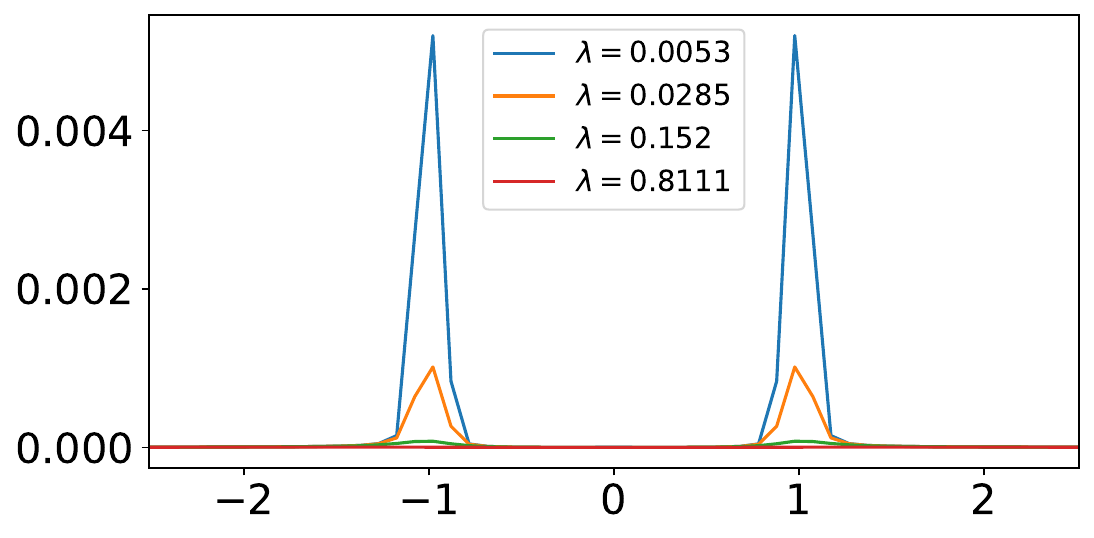}
    \put(-40,-40){(a)}
    \put(400,-35){\scriptsize frequency $\omega$}
    \put(-50,260){\rotatebox[]{90}{\scriptsize $\mathtt{S}^{11}(\omega)$}}
    \end{overpic}}
    \hspace{1mm}
    \subfloat{\begin{overpic}[scale=0.27]{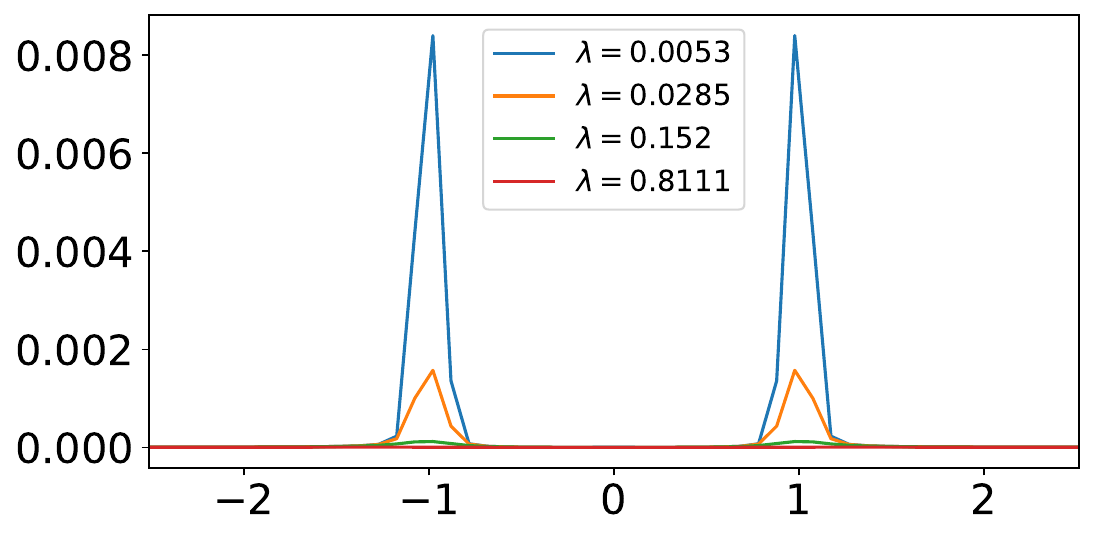}
    \put(-40,-40){(b)}
    \put(400,-35){\scriptsize frequency $\omega$}
    \end{overpic}}
    \subfloat{\begin{overpic}[scale=0.27]{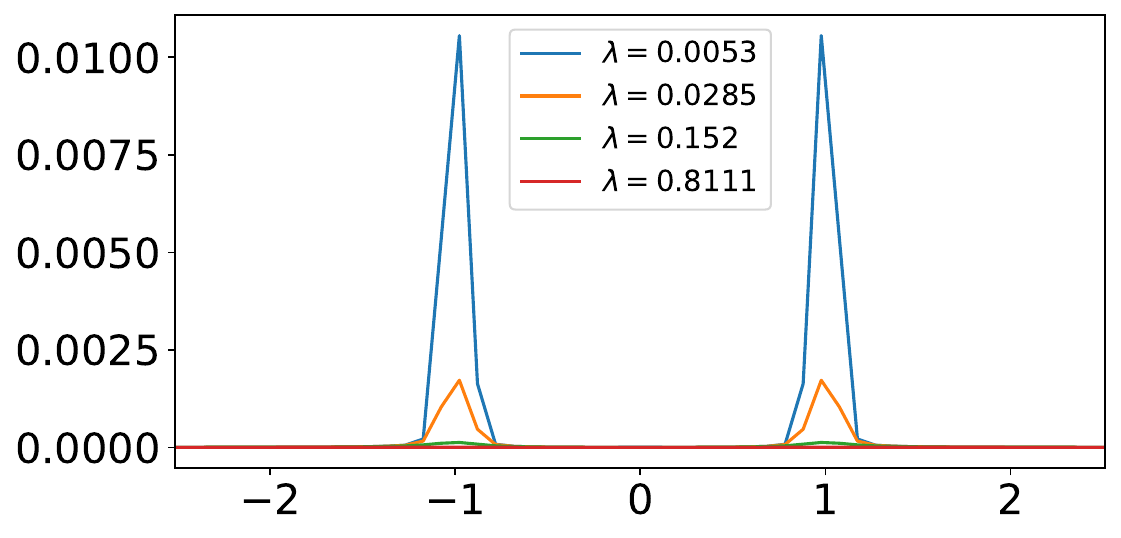}
    \put(-40,-40){(c)}
    \put(400,-35){\scriptsize frequency $\omega$}
    \end{overpic}}
    \caption{The first component of SD $\mathtt{S}(\omega)$ of the Hopf system driven by fOU noise \eqref{eq:Hopf fOU} for (a): $H=0.25$, (b): $H=0.5$ and (c): $H=0.75$.}
    \label{fig:Hopf fOU SD}
\end{figure}
Again, we can see the clear peaks due to the approaching bifurcation.\\
We omit the presentation of the red noise case, as it behaves very similarly to the fBm case.

\section{Conclusion}

In this work, we develop EWSs in fast-slow systems driven by non-Markovian stochastic forcing. By deriving scaling laws for the variance, autocorrelation, and power spectral density in the vicinity of critical transitions, we quantified how memory effects influence the detectability of approached bifurcations. For one-dimensional codimension-$1$ bifurcations, the scaling behavior was shown to depend explicitly on the Hurst parameter, reflecting a direct retention of the long-range dependence present in the driving noise. In contrast, for Hopf bifurcations, the rotational dynamics introduce a mixing mechanism that suppresses or reshapes memory effects, leading to scaling laws that are qualitatively less sensitive to the noise persistence. This distinction reveals a fundamental difference between dissipative and oscillatory critical transitions in non-Markovian environments. Moreover, by extending the analysis beyond fractional Brownian motion to red noise and fractional Ornstein-Uhlenbeck processes, we demonstrated that spectral observables remain effective EWSs even in situations where classical variance-based approaches suffer from masking or color-blindness effects.

The results provide a cohesive perspective on the practical use of EWSs in systems with temporally correlated fluctuations and suggest a broad range of potential applications. Since memory effects and colored noise are ubiquitous in a large variety of applications \cite{Oana,Flandoli,kuehn2022warning,Morr_Red_Noise}, the developed indicators offer a flexible methodology for detecting critical transitions under more realistic stochastic forcing than the classical white-noise setting. In particular, the applicability of the framework was illustrated through its implementation in ocean circulation models exhibiting tipping behavior as well as in oscillatory systems undergoing Hopf bifurcations.~More generally, the scaling laws derived here provide quantitative guidelines for interpreting observational data and selecting suitable indicators according to both the underlying bifurcation structure and the temporal characteristics of the noise.~This paves the way for more robust monitoring and prediction strategies in complex systems affected by memory and stochastic forcing terms.

\printbibliography

\appendix
\section{Time-asymptotic autocovariance and autocorrelation in the one-dimensional model}\label{app:A}
In this appendix, we prove the results in Section \ref{sec:auto 1d} and obtain the rate of the time-asymptotic autocovariance and autocorrelation in linearized models. Specifically, we consider the loss of stability in the proximity to a fold, pitchfork or transcritical bifurcation in \eqref{eq:fast system simplified} under \nameref{S1} conditions. We start by discussing the proof of Lemma \ref{lem:autocov fast system complex lemma}.
\begin{refproof}[Proof of Lemma~\ref{lem:autocov fast system complex lemma}]\label{proof:autocov fast system complex lemma}
First, we notice that equation \eqref{eq:autocov fast subsystem} is proven for $H=\frac{1}{2}$ in \cite{bernuzzi2024warning}. For the remainder of the proof, we consider then the assumption $H\in(0,1)\setminus\left\{\frac{1}{2}\right\}$. The next properties are described for $\left(\psi_t^{(1)}\right)_{t\geq 0}$, nonetheless they hold equivalently for $\left(\psi_t^{(2)}\right)_{t\geq 0}$. We write $\left(\psi_t^{(1)}\right)_{t\geq 0}$ using the mild solution formula
$$\psi_t^{(1)} = e^{\alpha t}\psi_0^{(1)} + \int_0^t e^{\alpha (t-r)}~\left\langle\txtd \mathbf{W}_r^H, \bm{\sigma}_1\right\rangle$$
with initial condition $\psi_0^{(1)}\in\C$. 
Moreover, we employ the Mandelbrot-van Ness representation of fBm. To this aim, we introduce the kernel $K_H:\R\times\R\to\R$, which is defined for $H\in(0,1)\setminus\left\{\frac{1}{2}\right\}$ as
$$K_H(t,s):= c_H s^{\frac{1}{2}-H}\int_s^t (u-s)^{H-\frac{3}{2}}u^{H-\frac{1}{2}}~\txtd u, \quad \text{for} \quad 0<s<t,$$
where
$$c_H = \left[\frac{2H\Gamma(\frac{3}{2}-H)}{\Gamma(H+\frac{1}{2})\Gamma(2-2H)} \right]^{\frac{1}{2}}.$$
We note that, since $K_H$ is not integrable for $H\in\left(0,\frac{1}{2}\right)$, we can't employ standard methods \cite{duncan2009semilinear} to simplify the covariance over $H\in(0,1)$.
Nonetheless, we can write fractional Brownian motion as an It\^o integral with respect to a Brownian motion  $(\mathbf{B}_t)_{t\geq 0}\subset \R^2$ given by 
\begin{align}\label{eq:fBM by BM}
    \mathbf{W}_t^H = \int_\R K_H(t,r)~\txtd \mathbf{B}_r.
\end{align}
Using \eqref{eq:fBM by BM} and \cite[Section 1.3]{kubilius2017parameter} we can represent the solution as
\begin{align*}
    \psi_t^{(1)} 
    &= e^{\alpha t}\psi_0^{(1)} + \left\langle \mathbf{W}_t^H, \bm{\sigma}_1\right\rangle + \alpha \int_0^t e^{\alpha(t-r)}\left\langle \mathbf{W}_r^H, \bm{\sigma}_1\right\rangle ~\txtd r \\
    &=e^{\alpha t}\psi_0^{(1)} + \int_\R\left(K_H(t,u) + \alpha \int_0^te^{\alpha (t-r)}K_H(r,u)~\txtd r \right) \left\langle\txtd \mathbf{B}_r, \bm{\sigma}_1\right\rangle.
\end{align*}
The expectation $\E\left[\psi_t^{(1)}\right]$ is equal to $e^{\alpha t}\psi_0^{(1)}$. Hence, the covariance at time $t$ of lag $\tau\ge0$ is given by $\text{Cov}\left(\psi_t^{(1)}, \psi_{t+\tau}^{(2)} \right)= \E\left[\psi_t^{(1)} \overline{\psi_{t+\tau}^{(2)}} \right] - e^{\alpha t + \overline{\beta} (t+\tau)}\psi_0^{(1)} \overline{\psi_0^{(2)}}$. Since the second term is zero in the time-asymptotic limit, we use It\^o's isometry \cite{da2014stochastic} within the stochastic integral and the identity $1=-\alpha\int_0^te^{\alpha(t-r)}~\txtd r + e^{\alpha t}$ to split $\E\left[\psi_t^{(1)} \overline{\psi_{t+\tau}^{(2)}} \right]$ into four terms.
\begin{align*}
    &\E\left[\psi_t^{(1)} \overline{\psi_{t+\tau}^{(2)}} \right] \\
    =&\left\langle \bm{\sigma}_2, \bm{\sigma}_1 \right\rangle \int_\R \left(K_H(t,u) + \alpha \int_0^te^{\alpha(t-r)}K_H(r,u)~\txtd r \right)\\
    &~~~~\times\left(K_H(t+\tau,u) + \overline{\beta}\int_0^{t+\tau}e^{\overline{\beta}(t+\tau-r)}K_H(r,u)~\txtd r \right)~\txtd u \\
    =&\left\langle \bm{\sigma}_2, \bm{\sigma}_1 \right\rangle \int_\R \left(e^{\alpha t}K_H(t,u) + \alpha \int_0^te^{\alpha (t-r)}(K_H(r,u)-K_H(t,u))~\txtd r \right)\\
    &~~~~\times\left(e^{\overline{\beta}(t+\tau)}K_H(t+\tau,u) + \overline{\beta}\int_0^{t+\tau}e^{\overline{\beta}(t+\tau-r)}(K_H(r,u)-K_H(t+\tau,u))~\txtd r \right)~\txtd u\\
    =& \left\langle \bm{\sigma}_2, \bm{\sigma}_1 \right\rangle \Bigg( \underbrace{ e^{\alpha t + \overline{\beta} (t + \tau)} \int_\R K_H(t,u) K_H(t+\tau,u) \txtd u}_{=\RomanI}\\
    &+ \underbrace{ e^{\alpha t} \overline{\beta} \int_\R K_H(t,u)  \int_0^{t+\tau} e^{\overline{\beta} (t+\tau-r)} \left( K_H(r,u) - K_H(t+\tau,u) \right) \txtd r \txtd u}_{=\RomanII}\\
    &+ \underbrace{ e^{\overline{\beta} (t+\tau)} \alpha \int_\R K_H(t+\tau,u) \int_0^t e^{\alpha (t-r)} \left( K_H(r,u) - K_H(t,u) \right) ~\txtd r\txtd u}_{=\RomanIII} \\
    & \begin{rcases} +\alpha \overline{\beta} \int_\R \left( \int_0^t e^{\alpha (t-r_1)} \left( K_H(r_1,u) - K_H(t,u) \right) \txtd r_1 \right) \\
    \times \left( \int_0^{t+\tau} e^{\overline{\beta} (t+\tau-r_2)} \left( K_H(r_2,u) - K_H(t+\tau,u) \right) \txtd r_2 \right)  \txtd u \Bigg).
    \end{rcases}=\RomanIV
\end{align*}
We consider each of the terms individually and use $$\int_\R K_H(t,u)K_H(r,u) \txtd u =\frac{1}{2}(|t|^{2H} + |r|^{2H}-|t-r|^{2H}).$$ We write $C(t)$ to indicate a parameter that behaves polynomially in $t$. We obtain then
\begin{align*}
    \RomanI = e^{\alpha t + \overline{\beta} (t + \tau)} \frac{1}{2}(t^{2H} + (t+\tau)^{2H}-\tau^{2H}) \le C(t) e^{\alpha t + \overline{\beta} (t + \tau)} \xrightarrow{t\to\infty}0,
\end{align*}
since $\text{Re}(\alpha),~\text{Re}(\beta)<0$. 
\begin{align*}
    \RomanII
    =& e^{\alpha t} \overline{\beta} \int_0^{t+\tau} e^{\overline{\beta} (t+\tau-r)} \int_\R K_H(t,u)  \left( K_H(r,u) - K_H(t+\tau,u) \right) \txtd u \txtd r \\
    =&\frac{1}{2} e^{\alpha t} \overline{\beta} \int_0^{t+\tau} e^{\overline{\beta} (t+\tau-r)} (r^{2H} - |t-r|^{2H} -(t+\tau)^{2H} + \tau^{2H} ) \txtd r \le C(t) e^{\alpha t}\xrightarrow{t\to\infty}0.
\end{align*}
With the same steps, we have
\begin{align*}
    \RomanIII
    =& e^{\overline{\beta} (t+\tau)} \alpha  \int_0^t e^{\alpha (t-r)} \int_\R K_H(t+\tau,u)\left( K_H(r,u) - K_H(t,u) \right) ~\txtd u\txtd r
    \le C(t) e^{\overline{\beta}t}\xrightarrow{t\to\infty}0.
\end{align*}
Lastly, we simplify $\RomanIV$ as
\begin{align*}
    \RomanIV 
    =& \frac{1}{2} \alpha \overline{\beta} \int_0^t \int_0^{t+\tau} e^{\alpha (t-r_1) +  \overline{\beta} (t+\tau-r_2)} \left( |t-r_2|^{2H} + |t+\tau-r_1|^{2H} - |r_1-r_2|^{2H} - \tau^{2H}\right) \txtd r_2 \txtd r_1 \\
    =& \frac{1}{2} \alpha \overline{\beta} \int_0^t \int_0^{t+\tau} e^{\alpha v_1 +  \overline{\beta} v_2} \left( |\tau-v_2|^{2H} + |\tau+v_1|^{2H} - |v_2-v_1-\tau|^{2H} - \tau^{2H}\right) \txtd v_2 \txtd v_1,
\end{align*}
where $v_1 = t-r_1,~v_2=t+\tau-r_2$. Combining the calculations above, we get that
\begin{align*}
    &\underset{t\to\infty}{\lim}\text{Cov}\left(\psi_t^{(1)}, \psi_{t+\tau}^{(2)} \right) \\
    =&\frac{\left\langle \bm{\sigma}_2, \bm{\sigma}_1 \right\rangle}{2} \alpha \overline{\beta} \int_0^\infty \int_0^\infty e^{\alpha v_1 + \overline{\beta} v_2} \left( |\tau-v_2|^{2H} + |\tau+v_1|^{2H} - \tau^{2H} - |v_2-v_1-\tau|^{2H} \right) \txtd v_1 \txtd v_2.
\end{align*}
Then, we split the integral into four terms and study each of them using a substitution. We use the substitution $w=A(\lambda)(v_2-\tau)$ and \eqref{eq:complex_property_2} to obtain that
\begin{align*}
    &\frac{1}{2} \alpha \overline{\beta} \int_0^\infty \int_0^\infty e^{\alpha v_1 + \overline{\beta} v_2}  |\tau-v_2|^{2H}  \txtd v_1 \txtd v_2 \\
    =& -\frac{1}{2} \overline{\beta}  \int_0^\infty e^{\overline{\beta} v_2}  |\tau-v_2|^{2H}   \txtd v_2 \\
    =& \frac{\overline{\beta}^{-2H}}{2}e^{\overline{\beta}\tau}\left((-1)^{-2H} \Gamma(2H+1) - \gamma(2H+1, \overline{\beta}\tau) \right) \\
    =& \frac{\overline{\beta}^{-2H}}{2}\left((-1)^{-2H}e^{\overline{\beta}\tau}2H\Gamma(2H) - e^{\overline{\beta}\tau}2H\gamma(2H, \overline{\beta}\tau) + (\overline{\beta}\tau)^{2H} \right),
\end{align*}
and that
\begin{align*}
    &\frac{1}{2} \alpha \overline{\beta} \int_0^\infty \int_0^\infty e^{\alpha v_1 + \overline{\beta} v_2}  |\tau+v_1|^{2H}  \txtd v_1 \txtd v_2 \\
    =& -\frac{1}{2} \alpha \int_{0}^\infty  e^{\alpha v_1}  (\tau+v_1)^{2H}  \txtd v_1 \\
    =& \frac{(-\alpha)^{-2H}}{2} e^{-\alpha\tau}  \Gamma(2H+1,-\alpha\tau) \\
    =& \frac{(-\alpha)^{-2H}}{2} \left(e^{-\alpha\tau}  2H\Gamma(2H,-\alpha\tau) + (-\alpha\tau)^{2H}\right).
\end{align*}
The third term is studied as
\begin{align*}
    &-\frac{1}{2} \alpha \overline{\beta} \int_0^\infty \int_0^\infty e^{\alpha v_1 + \overline{\beta} v_2}  \tau^{2H}  \txtd v_1 \txtd v_2 
    = -\frac{1}{2}\tau^{2H}.
\end{align*}
For the last integral, we use initially the substitutions $w_1 = v_1-v_2$ and $w_2=v_1+v_2$ in
\begin{align*}
    &-\frac{1}{2}\alpha \overline{\beta} \int_0^\I\int_0^\I e^{\alpha v_1 + \overline{\beta} v_2}|v_2-v_1-\tau|^{2H}~\txtd v_1\txtd v_2 \\
    =& -\frac{1}{4}\alpha \overline{\beta} \int_{-\I}^\I e^{\frac{\alpha - \overline{\beta}}{2}w_1} |w_1+\tau|^{2H} \int_{|w_1|}^\I e^{\frac{\alpha + \overline{\beta}}{2}w_2} ~\txtd w_2\txtd w_1 \\
    =& \frac{\alpha \overline{\beta}}{2(\alpha+ \overline{\beta})} \int_{-\I}^\I e^{\frac{\alpha - \overline{\beta}}{2}w_1} e^{\frac{\alpha + \overline{\beta}}{2}|w_1|} |w_1+\tau|^{2H} ~\txtd w_1\\
    =& \frac{\alpha \overline{\beta}}{2(\alpha+ \overline{\beta})} \Bigg( \int_0^\I e^{\alpha w_1} (w_1+\tau)^{2H} ~\txtd w_1 +
    \int_{-\tau}^0 e^{-\overline{\beta} w_1} (w_1+\tau)^{2H} ~\txtd w_1\\
    &~~~~+\int_{-\I}^{-\tau} e^{- \overline{\beta} w_1} (-w_1-\tau)^{2H} ~\txtd w_1\Bigg)
\end{align*}
and then we employ the substitutions $u_1=-\alpha(w_1+\tau)$, $u_2=\overline{\beta}(w_1+\tau)$ and $u_3=\overline{\beta}(w_1+\tau)$, respectively in each integral above, to get
\begin{align*}
    & -\frac{1}{2}\alpha \overline{\beta} \int_0^\I\int_0^\I e^{\alpha v_1 + \overline{\beta} v_2}|v_2-v_1-\tau|^{2H}~\txtd v_1\txtd v_2 \\
    =& \frac{\alpha \overline{\beta}}{2(\alpha+ \overline{\beta})} \Bigg( e^{-\alpha \tau} (-\alpha)^{-2H-1} \int_{-\alpha\tau}^{e^{i \arg(-\alpha)}\infty} e^{-u_1} u_1^{2H} ~\txtd u_1 +
    e^{\overline{\beta} \tau} \overline{\beta}^{-2H-1} \int_0^{\overline{\beta}\tau} e^{-u_2} u_2^{2H} ~\txtd u_2\\
    &~~~~+e^{\overline{\beta} \tau} (-\overline{\beta})^{-2H-1} \int_0^{e^{i \arg(-\overline{\beta})}\infty} e^{-u_3} u_3^{2H} ~\txtd u_3 \Bigg)\\
    =& \frac{\alpha \overline{\beta}}{2(\alpha+ \overline{\beta})} \Bigg( e^{-\alpha \tau} (-\alpha)^{-2H-1} \Gamma(2H+1,-\alpha\tau) +
    e^{\overline{\beta} \tau} \overline{\beta}^{-2H-1} \gamma(2H+1,\overline{\beta}\tau)\\
    &~~~~+e^{\overline{\beta} \tau} (-\overline{\beta})^{-2H-1} \Gamma(2H+1) \Bigg) \\
    =& \frac{\alpha \overline{\beta}}{2(\alpha+ \overline{\beta})} \Bigg(2H e^{-\alpha \tau} (-\alpha)^{-2H-1} \Gamma(2H,-\alpha\tau)
    -\alpha^{-1} \tau^{2H}\\
    &~~~~+2H e^{\overline{\beta} \tau} \overline{\beta}^{-2H-1} \gamma(2H,\overline{\beta}\tau)
    - \overline{\beta}^{-1} \tau^{2H} 
    +e^{\overline{\beta} \tau} (-\overline{\beta})^{-2H-1} 2H\Gamma(2H) \Bigg).
\end{align*}
In the last equality we use properties of the gamma functions and integration by parts. Finally, we assemble the computations to obtain that
\begin{align*}
    &\frac{\left\langle \bm{\sigma}_2, \bm{\sigma}_1 \right\rangle}{2} \alpha \overline{\beta} \int_0^\infty \int_0^\infty e^{\alpha v_1 +  \overline{\beta} v_2} \left( |\tau-v_2|^{2H} + |\tau+v_1|^{2H} - |v_2-v_1-\tau|^{2H} - \tau^{2H}\right) \txtd v_2 \txtd v_1\\
    =& -\left\langle \bm{\sigma}_2, \bm{\sigma}_1 \right\rangle \frac{H}{\alpha+\overline{\beta}} \Bigg( e^{\overline{\beta} \tau} (-\overline{\beta})^{-2H+1} \Gamma(2H)  \\
    &+ e^{-\alpha \tau} (-\alpha)^{-2H+1} \Gamma(2H,-\alpha\tau)
    + e^{\overline{\beta} \tau} \overline{\beta}^{-2H+1} \gamma(2H,\overline{\beta}\tau)\Bigg)
\end{align*}
from which the conclusion follows.
\qed
\end{refproof}
In the following, we apply a Taylor expansion approach to the result of Theorem \ref{thm:autocov fast system} in order to prove Corollary \ref{cor:autocor fast system}.
\begin{refproof}[Proof of Corollary \ref{cor:autocor fast system}]\label{proof:autocor fast system}
    Through the Taylor expansion of the exponential function, we rewrite the integral in the statement of Theorem \ref{thm:autocov fast system} as
\begin{align*}
    \int_0^{|A(\lambda)| \tau} e^w w^{2H-1} \txtd w = \int_0^{|A(\lambda)| \tau} w^{2H-1} \sum_{n=0}^\infty \frac{w^n}{n!} \txtd w = \sum_{n=0}^\infty \frac{(|A(\lambda)| \tau)^{2H+n}}{n! (2H+n)}
\end{align*}
and the upper incomplete gamma function in the form
\begin{align*}
    \Gamma(2H,|A(\lambda)| \tau)= \Gamma(2H) - \int_0^{|A(\lambda)| \tau} w^{2H-1} \sum_{n=0}^\infty \frac{(-w)^n}{n!} \txtd w = \Gamma(2H) - \sum_{n=0}^\infty \frac{(-1)^n (|A(\lambda)| \tau)^{2H+n}}{n!(2H+n)} .
\end{align*}
In the limit $\lambda\to\lambda^*$, we then obtain the following:
\begin{align*}
    \frac{V_\infty(\tau)}{V_\infty(0)}
    =&\frac{\frac{\sigma^2}{2} H |A(\lambda)|^{-2 H} \left( e^{A(\lambda) \tau} \Gamma(2H) + e^{-A(\lambda) \tau} \Gamma(2H,|A(\lambda)| \tau) - e^{A(\lambda) \tau} \int_0^{|A(\lambda)| \tau} e^w w^{2H-1} \txtd w \right)}{\frac{\sigma^2}{2} |A(\lambda)|^{-2 H} \Gamma(2H+1)} \\
    =&\frac{H \left( e^{A(\lambda) \tau} \Gamma(2H) + e^{-A(\lambda) \tau} \Gamma(2H,|A(\lambda)| \tau) - e^{A(\lambda) \tau} \int_0^{|A(\lambda)| \tau} e^w w^{2H-1} \txtd w \right)}{\Gamma(2H+1)}\\
    =&\frac{H \left(\underset{n=0}{\overset{\infty}{\sum}} \frac{(-1)^n (|A(\lambda)| \tau)^n}{n!} \right) \Gamma(2H)}{\Gamma(2H+1)} \\
    &+\frac{H \left( \underset{n=0}{\overset{\infty}{\sum}} \frac{(|A(\lambda)| \tau)^n}{n!} \right) \left( \Gamma(2H) - \sum_{n=0}^\infty \frac{(-1)^n (|A(\lambda)| \tau)^{2H+n}}{n!(2H+n)} \right)}{\Gamma(2H+1)} \\
    & - \frac{ H \left(\underset{n=0}{\overset{\infty}{\sum}} \frac{(-1)^n (|A(\lambda)| \tau)^n}{n!} \right) \left( \sum_{n=0}^\infty \frac{(|A(\lambda)| \tau)^{2H+n}}{n! (2H+n)} \right)}{\Gamma(2H+1)}.
\end{align*}
As in the studied regime $|A(\lambda)| $ is in the proximity of 0, we are mainly interested in the leading order terms. Therefore, we collect the higher order terms into $\cO(|A(\lambda)|^2)$. Consequently,
\begin{align*}
    \frac{V_\infty(\tau)}{V_\infty(0)}
    =& \frac{H \Gamma(2H) - H |A(\lambda)| \tau \Gamma(2H)}{\Gamma(2H+1)} \\
    &+ \frac{ H \Gamma(2H) + H |A(\lambda)| \tau \Gamma(2H) - \frac{(|A(\lambda)| \tau)^{2H}}{2} - \frac{(|A(\lambda)| \tau)^{2H+1}}{2} + H\frac{(|A(\lambda)| \tau)^{2H+1}}{2H+1}}{\Gamma(2H+1)} \\
    &+ \frac{- \frac{(|A(\lambda)| \tau)^{2H}}{2} + \frac{(|A(\lambda)| \tau)^{2H+1}}{2}  - H\frac{(|A(\lambda)| \tau)^{2H+1}}{2H+1}}{\Gamma(2H+1)} + \mathcal{O}\left( |A(\lambda)|^2 \right) \\
    =& \frac{2 H \Gamma(2H)}{\Gamma(2H+1)} - |A(\lambda)|^{2H} \frac{\tau^{2H}}{\Gamma(2H+1)} + \mathcal{O}\left( |A(\lambda)|^2 \right)\\
    =& 1 - |A(\lambda)|^{2H} \frac{\tau^{2H}}{\Gamma(2H+1)} + \mathcal{O}\left( |A(\lambda)|^2 \right) .
\end{align*}
    \qed
\end{refproof}

\section{Time-asymptotic autocovariance and autocorrelation in the two-dimensional model} \label{app:B}
This appendix collects the proofs of the results in Section \ref{sec:auto 2d}. We obtain the rate of the time-asymptotic autocovariance and autocorrelation in the limit to a Hopf bifurcation in \eqref{eq:fast system simplified} under \nameref{S2} conditions. In the proof below, we extend the results of Lemma \ref{lem:autocov fast system complex lemma} to the study of the time-asymptotic autocovariance of the solution of \eqref{eq:fast system simplified} in \nameref{S2}. As we consider a two-dimensional setting, the stochastic differential equations in \eqref{eq:fast system simplified hopf 2} refer to the trajectory in spectral modes. In the case $\tau\gg0$, the divergence along such modes could in principle diminish the scaling law along general modes. However, we find and distinguish the rare occasions in which this phenomenon occurs.

\begin{refproof}[Proof of Theorem \ref{thm:autocov fast system complex}]\label{proof:autocov fast system complex}
By construction, any $\mathbf{v}\in\R^2$ can be written as
\begin{align*}
    \mathbf{v}=\left\langle \mathbf{v}, \mathbf{e}_{1}(\lambda) \right\rangle \mathbf{e}_{1}^*(\lambda) + \left\langle \mathbf{v}, \mathbf{e}_{2}(\lambda) \right\rangle \mathbf{e}_{2}^*(\lambda) .
\end{align*}
Then, we can use the fact that $\underset{t\to\infty}{\lim}\mathbb{E}[x_t]$ is the null vector to imply that
\begin{align} \label{eq:call_1}
    V_\infty(\tau)\left[ \mathbf{v}_1, \mathbf{v}_2 \right] =& \underset{t\to\infty}{\lim}\text{Cov}\left( \left\langle \mathbf{x}_t, \mathbf{v}_1 \right\rangle, \left\langle \mathbf{x}_{t+\tau}, \mathbf{v}_2 \right\rangle \right) \nonumber \\
    =& \underset{t\to\infty}{\lim}\mathbb{E}\left[ \left\langle \mathbf{x}_t, \mathbf{v}_1 \right\rangle \overline{ \left\langle \mathbf{x}_{t+\tau}, \mathbf{v}_2 \right\rangle} \right] \\
    =& \underset{t\to\infty}{\lim} \sum_{j_1,j_2\in\{1,2\}} \overline{\left\langle \mathbf{v}_1, \mathbf{e}_{j_1}(\lambda) \right\rangle} \left\langle \mathbf{v}_2, \mathbf{e}_{j_2}(\lambda) \right\rangle \mathbb{E}\left[ \left\langle \mathbf{x}_t, \mathbf{e}_{j_1}^*(\lambda) \right\rangle \overline{ \left\langle \mathbf{x}_{t+\tau}, \mathbf{e}_{j_2}^*(\lambda) \right\rangle} \right] \nonumber \\
    =& \sum_{j_1,j_2\in\{1,2\}}\overline{\left\langle \mathbf{v}_1, \mathbf{e}_{j_1}(\lambda) \right\rangle} \left\langle \mathbf{v}_2, \mathbf{e}_{j_2}(\lambda) \right\rangle V_\infty(\tau)\left[ \mathbf{e}_{j_1}^*(\lambda), \mathbf{e}_{j_2}^*(\lambda) \right] \nonumber,
\end{align}
for any $\mathbf{v}_1$ and $\mathbf{v}_2$ in $\R^2$.
We consider $\varphi_t^{(1)}:=\left\langle \mathbf{x}_t, \mathbf{e}_1^*(\lambda) \right\rangle$ and $\varphi_t^{(2)}:=\left\langle \mathbf{x}_t, \mathbf{e}_2^*(\lambda) \right\rangle$ and notice that they solve
\begin{align*}
    \begin{split}
        \txtd \varphi_t^{(1)} =& (A(\lambda) + i B(\lambda)) \varphi_t^{(1)}~\txtd t + \left\langle \txtd \mathbf{W}_t^H, \Sigma^\text{T} \mathbf{e}_1^*(\lambda) \right\rangle,\\
        \txtd \varphi_t^{(2)} =& (A(\lambda) - i B(\lambda)) \varphi_t^{(2)}~\txtd t + \left\langle \txtd \mathbf{W}_t^H, \Sigma^\text{T} \mathbf{e}_2^*(\lambda) \right\rangle,
    \end{split}
\end{align*}
for $t\geq0$ and with initial conditions $\varphi_0^{(1)}:=\left\langle \mathbf{x}_0, \mathbf{e}_1^*(\lambda) \right\rangle$ and $\varphi_0^{(2)}:=\left\langle \mathbf{x}_0, \mathbf{e}_2^*(\lambda) \right\rangle$.
Consequently, equations \ref{eq:autocov fast subsystem complex spectral} follow directly from \eqref{eq:autocov fast subsystem complex 2}. Moreover, we obtain the scaling laws of the time-asymptotic autocovariance associated with $\left\{\varphi_t^{(1)}\right\}_{t\geq0}$ and $\left\{\varphi_t^{(2)}\right\}_{t\geq0}$ in the limit $\lambda\to \lambda^*$, i.e. $A(\lambda)\to 0^-$. This is achieved upon considering that the assumption $B(\lambda^*)\neq 0$ implies that the memory term in \eqref{eq:autocov fast subsystem complex 2} satisfies
\begin{align} \label{eq: detail tau}
    |P(A(\lambda)\pm iB(\lambda),A(\lambda)\pm iB(\lambda),H,\tau)|\asymp1 \quad \text{for almost every} \quad \tau\geq 0,
\end{align}
with function $P$ introduced in \eqref{eq:function_P} that is analytic in $\tau$. Moreover, the masking term assumes the form $\left\langle  \mathbf{e}_{j_2}^*(\lambda), Q \mathbf{e}_{j_1}^*(\lambda)\right\rangle$ the form for corresponding $j_1$ and $j_2$. Without loss of generality, we can assume that $\left\langle \mathbf{e}_{1}^*(\lambda^*), Q \mathbf{e}_{1}^*(\lambda^*)\right\rangle>0$. Finally, the mitigation terms display divergence (hyperbolically in $|A(\lambda)|$) only for $j_1=j_2$. The rates of these terms and the modal variance are collected then in Table \ref{tab: tiny}.
\begin{center}
    \begin{tabular}{|c||c|c|c|}
        \hline
        \textnormal{Terms in} & \multicolumn{3}{c|}{$(j_1,j_2)$}  \\
        \cline{2-4}
        $V_\infty(\tau)\left[ \mathbf{e}_{j_1}^*(\lambda), \mathbf{e}_{j_2}^*(\lambda) \right]$ & $(1,1)$ & $(1,2)$ & $(2,2)$ \\
        \hline
        \textnormal{Masking} & $0$ & $0$ & $0$ \\
        \hline
        \textnormal{Memory} & $0$ & $0$ & $0$ \\
        \hline
        \textnormal{Mitigation} & $-1$ & $0$ & $-1$ \\
        \hline
        \textnormal{Modal} & $-1$ & $0$ & $-1$ \\
        \hline
    \end{tabular}
    \captionof{table}{We display the exponents $\nu$ in the scaling law $\mathcal{O}\left(|A(\lambda)|^\nu\right)$ associated to the corresponding term introduced in \eqref{eq:autocov fast subsystem complex 2}. An important case is the column corresponding to $j_1=j_2=1$, in which the rates are asymptotically comparable to $|A(\lambda)|^\nu$ for almost every $\tau\geq 0$.}
    \label{tab: tiny}
\end{center}

We note also that
\begin{align} \label{eq:call_4}
        &P(A(\lambda^*)+iB(\lambda^*), A(\lambda^*)+iB(\lambda^*),H, \tau) \\
        =& e^{-2i B(\lambda^*)\tau} P(A(\lambda^*)-iB(\lambda^*), A(\lambda^*)-iB(\lambda^*),H, \tau) \nonumber. \nonumber
\end{align}
In conclusion, we know from \eqref{eq:call_1} and \eqref{eq:call_4} that \eqref{eq:autocov fast subsystem complex 3} holds for $\tau\geq0$ that satisfies \eqref{eq: detail tau} and for $\mathbf{v}_1,\mathbf{v}_2\in\R^2$ that satisfy \eqref{eq: civil war}.
\qed
\end{refproof}

Below, we apply the statement of Lemma \ref{lem:autocov fast system complex lemma} to describe the limit and rate of convergence of the time-asymptotic autocorrelation of the solution to \eqref{eq:fast system simplified} in the \nameref{S2} setting. The distinction between diverging components in \eqref{eq:autocov fast subsystem complex 2} allows to pinpoint the main components capable to define the behavior of the observable in the limit $\lambda\to\lambda^*$.

\begin{refproof}[Proof of Lemma \ref{lem:autocor fast system complex}]\label{proof:autocor fast system complex}
We notice first that \eqref{eq:autocov fast subsystem complex 2} implies that
\begin{align*}
    \underset{t\to\infty}{\lim}\frac{\text{Cov}\left(\psi_t^{(1)}, \psi_{t+\tau}^{(1)} \right)}{\text{Cov}\left(\psi_t^{(1)}, \psi_{t}^{(1)} \right)}
    =
    \frac{\underset{t\to\infty}{\lim}\text{Cov}\left(\psi_t^{(1)}, \psi_{t+\tau}^{(1)} \right)}{\underset{t\to\infty}{\lim}\text{Cov}\left(\psi_t^{(1)}, \psi_{t}^{(1)} \right)}
    = \frac{P(\alpha(\lambda),\alpha(\lambda),H,\tau)}{P(\alpha(\lambda),\alpha(\lambda),H,0)}
\end{align*}
for any $\lambda\leq \lambda^*$. Moreover, the limit 
\begin{align*}
    &\underset{\lambda\to\lambda^*}{\lim} \underset{t\to\infty}{\lim}\frac{\text{Cov}\left(\psi_t^{(1)}, \psi_{t+\tau}^{(1)} \right)}{\text{Cov}\left(\psi_t^{(1)}, \psi_{t}^{(1)} \right)}
    =\underset{\lambda\to\lambda^*}{\lim} \underset{t\to\infty}{\lim} \frac{P(\alpha(\lambda),\alpha(\lambda),H,\tau)}{P(\alpha(\lambda),\alpha(\lambda),H,0)}
    = e^{-i \operatorname{Im}(\alpha(\lambda^*))\tau}
\end{align*}
is a direct consequence of the fact that $\alpha(\lambda^*)=-\overline{\alpha(\lambda^*)}$. As such, we study
\begin{align} \label{eq: autocov complex first step}
    \begin{split}
        &\frac{P(\alpha(\lambda),\alpha(\lambda),H,\tau)}{P(\alpha(\lambda),\alpha(\lambda),H,0)} - e^{-i \operatorname{Im}(\alpha(\lambda))\tau}\\
        =& \frac{P(\alpha(\lambda),\alpha(\lambda),H,\tau) - e^{-i \operatorname{Im}(\alpha(\lambda))\tau} P(\alpha(\lambda),\alpha(\lambda),H,0)}{P(\alpha(\lambda),\alpha(\lambda),H,0)}
    \end{split}
\end{align}
and the terms within. First, we consider that
\begin{align} \label{eq: autocov complex second step}
    \begin{split}
        P(\alpha(\lambda),\alpha(\lambda),H,0) =& 2 \operatorname{Re}\left( (-\alpha(\lambda))^{-2H+1} \right) \Gamma(2H) \\
        =& 2 (2H-1) \operatorname{Re}\left( (-\alpha(\lambda))^{-2H+1} \right) \Gamma(2H-1)
    \end{split}
\end{align}
as illustrated in \eqref{eq:function_P}. Since $\alpha(\lambda^*)\neq 0$, we know that $|P(\alpha(\lambda),\alpha(\lambda),H,0)|\asymp 1$ in the limit $\lambda\to\lambda^*$. The addends in the numerator are addressed as follows:
\begin{align}  \label{eq: autocov complex third step}
    P(\alpha(\lambda),\alpha(\lambda),H,\tau)=& e^{-i \operatorname{Im}(\alpha(\lambda))\tau} \Bigg(-\overline{\alpha(\lambda)} e^{\operatorname{Re}(\alpha(\lambda))\tau} \int_0^\infty e^{\overline{\alpha(\lambda)}w} w^{2H-1} \text{d}w \nonumber\\
    &-\alpha(\lambda) e^{- \operatorname{Re}(\alpha(\lambda))\tau} \int_\tau^\infty e^{\alpha(\lambda) w} w^{2H-1} \text{d}w\\
    &+ \overline{\alpha(\lambda)} e^{\operatorname{Re}(\alpha(\lambda))\tau} \int_0^\tau e^{-\overline{\alpha}(\lambda) w} w^{2H-1} \text{d}w\Bigg) \nonumber
\end{align}
and
\begin{align}  \label{eq: autocov complex fourth step}
    &e^{-i \operatorname{Im}(\alpha(\lambda))\tau} P(\alpha(\lambda),\alpha(\lambda),H,0) \\
    =& e^{-i \operatorname{Im}(\alpha(\lambda))\tau} \Bigg(-\overline{\alpha(\lambda)} \int_0^\infty e^{\overline{\alpha(\lambda)}w} w^{2H-1} \text{d}w
    -\alpha(\lambda) \int_0^\infty e^{\alpha(\lambda) w} w^{2H-1} \text{d}w \Bigg). \nonumber
\end{align}
We study the numerator in \eqref{eq: autocov complex first step} by combining \eqref{eq: autocov complex third step} and \eqref{eq: autocov complex fourth step} in
\begin{align} \label{eq: autocov complex fifth step}
    &P(\alpha(\lambda),\alpha(\lambda),H,\tau) - e^{-i \operatorname{Im}(\alpha(\lambda))\tau} P(\alpha(\lambda),\alpha(\lambda),H,0) \nonumber\\
    =& e^{-i \operatorname{Im}(\alpha(\lambda))\tau} \nonumber
    \Bigg( -\overline{\alpha(\lambda)} \left( e^{\operatorname{Re}(\alpha(\lambda))\tau} - 1 \right) \int_0^\infty e^{\overline{\alpha(\lambda)}w} w^{2H-1} \text{d}w \nonumber\\
    &~~~~~~~~~~~~~~ -\alpha(\lambda) \left( e^{-\operatorname{Re}(\alpha(\lambda))\tau} - 1 \right) \int_\tau^\infty e^{\alpha(\lambda) w} w^{2H-1} \text{d}w\\
    &~~~~~~~~~~~~~~ + \overline{\alpha(\lambda)} \left( e^{\operatorname{Re}(\alpha(\lambda))\tau} - 1 \right) \int_0^\tau e^{-\overline{\alpha(\lambda)} w} w^{2H-1} \text{d}w \nonumber\\
    &~~~~~~~~~~~~~~ + \int_0^\tau \left( \alpha(\lambda) e^{\alpha(\lambda) w} + \overline{\alpha(\lambda)} e^{-\overline{\alpha(\lambda)} w} \right) w^{2H-1} \text{d}w \Bigg). \nonumber
\end{align}
For $H\in\left( 0, \frac{1}{2}\right)$, through first-order Taylor expansion of $\operatorname{Re}(\alpha)$ on $0$, we obtain
\begin{align*}
    \bullet \quad \qquad& -\overline{\alpha(\lambda)} \left( e^{\operatorname{Re}(\alpha(\lambda))\tau} - 1 \right) \int_0^\infty e^{\overline{\alpha(\lambda)}w} w^{2H-1} \text{d}w \\
    =& i \operatorname{Im}(\alpha(\lambda^*)) \operatorname{Re}(\alpha(\lambda))\tau \int_0^\infty e^{-i \operatorname{Im}(\alpha(\lambda^*)) w} w^{2H-1} \text{d}w + \mathcal{O}\left( \operatorname{Re}(\alpha(\lambda))^2 \right),\\
    \bullet \quad \qquad& -\alpha(\lambda) \left( e^{-\operatorname{Re}(\alpha(\lambda))\tau} - 1 \right) \int_\tau^\infty e^{\alpha(\lambda) w} w^{2H-1} \text{d}w \\
    =& i \operatorname{Im}(\alpha(\lambda^*)) \operatorname{Re}(\alpha(\lambda))\tau \int_\tau^\infty e^{i \operatorname{Im}(\alpha(\lambda^*)) w} w^{2H-1} \text{d}w + \mathcal{O}\left( \operatorname{Re}(\alpha(\lambda))^2 \right),\\
    \bullet \quad \qquad& \overline{\alpha(\lambda)} \left( e^{\operatorname{Re}(\alpha(\lambda))\tau} - 1 \right) \int_0^\tau e^{-\overline{\alpha(\lambda)} w} w^{2H-1} \text{d}w \\
    =& - i \operatorname{Im}(\alpha(\lambda^*)) \operatorname{Re}(\alpha(\lambda))\tau \int_0^\tau e^{i \operatorname{Im}(\alpha(\lambda^*)) w} w^{2H-1} \text{d}w + \mathcal{O}\left( \operatorname{Re}(\alpha(\lambda))^2 \right) \quad \text{and}\\
    \bullet \quad \qquad& \int_0^\tau \left( \alpha(\lambda) e^{\alpha(\lambda) w} + \overline{\alpha(\lambda)} e^{-\overline{\alpha(\lambda)} w} \right) w^{2H-1} \text{d}w \\
    =& 2 (i \operatorname{Im}(\alpha(\lambda^*)) + 1) \operatorname{Re}(\alpha(\lambda))\tau \int_0^\tau e^{i \operatorname{Im}(\alpha(\lambda^*)) w} w^{2H} \text{d}w + \mathcal{O}\left( \operatorname{Re}(\alpha(\lambda))^2 \right).
\end{align*}
As such,
\begin{align*}
    &P(\alpha(\lambda),\alpha(\lambda),H,\tau) - e^{-i \operatorname{Im}(\alpha(\lambda))\tau} P(\alpha(\lambda),\alpha(\lambda),H,0) \nonumber\\
    =& 2 e^{-i \operatorname{Im}(\alpha(\lambda))\tau} \operatorname{Re}(\alpha(\lambda))\tau \nonumber
    \Bigg( i \operatorname{Im}(\alpha(\lambda^*)) \operatorname{Re} \left( \int_0^\infty e^{i \operatorname{Im}(\alpha(\lambda^*)) w} w^{2H-1} \text{d}w \right)\\
    &~~~~~~~~~~~~~~ - i \operatorname{Im}(\alpha(\lambda^*)) \int_0^\tau e^{i \operatorname{Im}(\alpha(\lambda^*)) w} w^{2H-1} \text{d}w \nonumber\\
    &~~~~~~~~~~~~~~ + (i \operatorname{Im}(\alpha(\lambda^*)) + 1) \int_0^\tau e^{i \operatorname{Im}(\alpha(\lambda^*)) w} w^{2H} \text{d}w \Bigg)
    + \mathcal{O}\left( \operatorname{Re}(\alpha(\lambda))^2 \right) \nonumber
\end{align*}
and
\begin{align*}
    \Delta(\tau):=&\left| \operatorname{Im}(\alpha(\lambda^*)) \left( \operatorname{Re}\left((-\alpha(\lambda^*))^{-2H+1}\right) \Gamma(2H) \right)^{-1} \right|\\
    &\times \Bigg| \operatorname{Re} \left( \int_0^\infty e^{i \operatorname{Im}(\alpha(\lambda^*)) w} w^{2H-1} \text{d}w \right) \\
    &+ \int_0^\tau e^{i \operatorname{Im}(\alpha(\lambda^*)) w} w^{2H-1} \left(- 1 + w - i \operatorname{Im}(\alpha(\lambda^*))^{-1} w \right) \text{d}w\Bigg|.
\end{align*}
In the case $H=\frac{1}{2}$, it follows that
\begin{align*}
    \bullet \quad \qquad& -\overline{\alpha(\lambda)} \left( e^{\operatorname{Re}(\alpha(\lambda))\tau} - 1 \right) \int_0^\infty e^{\overline{\alpha(\lambda)}w} w^{2H-1} \text{d}w \\
    =& \operatorname{Re}(\alpha(\lambda))\tau + \mathcal{O}\left( \operatorname{Re}(\alpha(\lambda))^2 \right),\\
    \bullet \quad \qquad& -\alpha(\lambda) \left( e^{-\operatorname{Re}(\alpha(\lambda))\tau} - 1 \right) \int_\tau^\infty e^{\alpha(\lambda) w} w^{2H-1} \text{d}w \\
    =& - \operatorname{Re}(\alpha(\lambda))\tau e^{i \operatorname{Im}(\alpha(\lambda^*)) \tau} + \mathcal{O}\left( \operatorname{Re}(\alpha(\lambda))^2 \right),\\
    \bullet \quad \qquad& \overline{\alpha(\lambda)} \left( e^{\operatorname{Re}(\alpha(\lambda))\tau} - 1 \right) \int_0^\tau e^{-\overline{\alpha(\lambda)} w} w^{2H-1} \text{d}w \\
    =& \operatorname{Re}(\alpha(\lambda))\tau \left( 1 - e^{i \operatorname{Im}(\alpha(\lambda^*)) \tau} \right) + \mathcal{O}\left( \operatorname{Re}(\alpha(\lambda))^2 \right) \quad \text{and}\\
    \bullet \quad \qquad& \int_0^\tau \left( \alpha(\lambda) e^{\alpha(\lambda) w} + \overline{\alpha(\lambda)} e^{-\overline{\alpha(\lambda)} w} \right) w^{2H-1} \text{d}w \\
    =& 2 \operatorname{Re}(\alpha(\lambda))\tau e^{i \operatorname{Im}(\alpha(\lambda^*)) \tau} + \mathcal{O}\left( \operatorname{Re}(\alpha(\lambda))^2 \right).
\end{align*}
Consequently,
\begin{align*}
    &P(\alpha(\lambda),\alpha(\lambda),H,\tau) - e^{-i \operatorname{Im}(\alpha(\lambda))\tau} P(\alpha(\lambda),\alpha(\lambda),H,0) \nonumber\\
    =& 2 e^{-i \operatorname{Im}(\alpha(\lambda))\tau} \operatorname{Re}(\alpha(\lambda))\tau
    + \mathcal{O}\left( \operatorname{Re}(\alpha(\lambda))^2 \right) \nonumber
\end{align*}
and $\Delta(\tau):=1$. In contrast, for $H\in\left( \frac{1}{2}, 1\right)$, we get
\begin{align*}
    \bullet \quad \qquad& -\overline{\alpha(\lambda)} \left( e^{\operatorname{Re}(\alpha(\lambda))\tau} - 1 \right) \int_0^\infty e^{\overline{\alpha(\lambda)}w} w^{2H-1} \text{d}w \\
    =& \operatorname{Re}(\alpha(\lambda))\tau (2H-1) \int_0^\infty e^{-i \operatorname{Im}(\alpha(\lambda^*)) w} w^{2H-2} \text{d}w + \mathcal{O}\left( \operatorname{Re}(\alpha(\lambda))^2 \right),\\
    \bullet \quad \qquad& -\alpha(\lambda) \left( e^{-\operatorname{Re}(\alpha(\lambda))\tau} - 1 \right) \int_\tau^\infty e^{\alpha(\lambda) w} w^{2H-1} \text{d}w \\
    =& \operatorname{Re}(\alpha(\lambda)) \left(- e^{i \operatorname{Im}(\alpha(\lambda^*)) \tau} \tau^{2H} - \tau (2H-1) \int_\tau^\infty e^{i \operatorname{Im}(\alpha(\lambda^*)) w} w^{2H-2} \text{d}w\right) + \mathcal{O}\left( \operatorname{Re}(\alpha(\lambda))^2 \right),\\
    \bullet \quad \qquad& \overline{\alpha(\lambda)} \left( e^{\operatorname{Re}(\alpha(\lambda))\tau} - 1 \right) \int_0^\tau e^{-\overline{\alpha(\lambda)} w} w^{2H-1} \text{d}w \\
    =& \operatorname{Re}(\alpha(\lambda))\left( -e^{i \operatorname{Im}(\alpha(\lambda^*)) \tau} \tau^{2H} + \tau (2H-1) \int_0^\tau e^{i \operatorname{Im}(\alpha(\lambda^*)) w} w^{2H-2} \text{d}w \right) + \mathcal{O}\left( \operatorname{Re}(\alpha(\lambda))^2 \right),\\
    \bullet \quad \qquad& \int_0^\tau \left( \alpha(\lambda) e^{\alpha(\lambda) w} + \overline{\alpha(\lambda)} e^{-\overline{\alpha(\lambda)} w} \right) w^{2H-1} \text{d}w \\
    =& 2 \operatorname{Re}(\alpha(\lambda)) \left( e^{i \operatorname{Im}(\alpha(\lambda^*)) \tau} \tau^{2H} - \tau (2H-1) \int_0^\tau e^{i \operatorname{Im}(\alpha(\lambda^*)) w} w^{2H-1} \text{d}w \right) + \mathcal{O}\left( \operatorname{Re}(\alpha(\lambda))^2 \right).
\end{align*}
This implies that
\begin{align*}
    &P(\alpha(\lambda),\alpha(\lambda),H,\tau) - e^{-i \operatorname{Im}(\alpha(\lambda))\tau} P(\alpha(\lambda),\alpha(\lambda),H,0) \nonumber\\
    =& 2 e^{-i \operatorname{Im}(\alpha(\lambda))\tau} \operatorname{Re}(\alpha(\lambda))\tau (2H-1) \nonumber
    \Bigg( \operatorname{Re} \left( \int_0^\infty e^{i \operatorname{Im}(\alpha(\lambda^*)) w} w^{2H-2} \text{d}w \right) \nonumber\\
    &~~~~~~~~~~~~~~ - \int_\tau^\infty e^{i \operatorname{Im}(\alpha(\lambda^*)) w} w^{2H-2} \text{d}w \\
    &~~~~~~~~~~~~~~ - \int_0^\tau e^{i \operatorname{Im}(\alpha(\lambda^*)) w} w^{2H-1} \text{d}w \Bigg) + \mathcal{O}\left( \operatorname{Re}(\alpha(\lambda))^2 \right) \nonumber
\end{align*}
and that
\begin{align*}
    \Delta(\tau):=& \left| \operatorname{Re}\left((-\alpha(\lambda^*))^{-2H+1}\right) \Gamma(2H-1) \right|^{-1} \\
    &\times \Bigg| \int_\tau^\infty e^{i \operatorname{Im}(\alpha(\lambda^*)) w} w^{2H-2} \text{d}w + \int_0^\tau e^{i \operatorname{Im}(\alpha(\lambda^*)) w} w^{2H-1} \text{d}w \\
    &- \operatorname{Re}\left( \int_0^\infty e^{i \operatorname{Im}(\alpha(\lambda^*)) w} w^{2H-2} \text{d}w \right)\Bigg|.
\end{align*}
We note that \eqref{eq:complex_property_1} implies $\arg\left(-\alpha(\lambda^*))^{-2H+1}\right)\in\left(-\frac{\pi}{2},\frac{\pi}{2}\right)$ and the well-posedness of $\Delta(\tau)$. In conclusion, \eqref{eq: autocov complex first step}, \eqref{eq: autocov complex second step} and \eqref{eq: autocov complex fifth step} imply \eqref{eq:autocor fast subsystem complex}. In fact, we have that $\Delta(0)\neq0$ for any $H\in(0,1)$. As a result, due to the analiticity in $\tau$ of the integrands in its definition, it is not null for almost every $\tau>0$.
\qed
\end{refproof}

The proof below, associated to Corollary \ref{cor:autocor fast system complex}, employs the scaling law of the time-asymptotic modal autocovariance in \eqref{eq:autocov fast subsystem complex spectral} and the rate of convergence described in Lemma \ref{lem:autocor fast system complex}. The first enables to define a main component to the time-asymptotic autocorrelation by excluding the terms of the time-asymptotic autocovariance along general modes that do not affect its scaling law. The second allows us to describe the limit and rate of convergence in the corollary through the study of the remaining items.

\begin{refproof}[Proof of Corollary \ref{cor:autocor fast system complex}]\label{proof: autocor fast system complex cor}
    In the case an eigenfunction of $M(\lambda^*)^\text{T}$ is in $\text{Ker}\left(\Sigma^T\right)$, the proof follows directly from Lemma \ref{lem:autocor fast system complex}. Consequently, we omit this scenario in the analysis to follow. Under such assumptions, the observables $AC_\infty(\tau)\left[\mathbf{e}_1^*(\lambda)\right]$ and $AC_\infty(\tau)\left[\mathbf{e}_2^*(\lambda)\right]$ are well-posed for $\lambda< \lambda^*$ by Theorem \ref{thm:autocov fast system complex} and extended in $\lambda=\lambda^*$ by Lemma \ref{lem:autocor fast system complex}. \\
    Throughout the proof we indicate the main component of the time-asymptotic autocorrelation function with
    \begin{align*}
        MC(\tau)[\mathbf{v}]:=\frac{ \left|\left\langle \mathbf{v}, \mathbf{e}_1(\lambda) \right\rangle\right|^2 V_\infty(\tau)\left[\mathbf{e}_1^*(\lambda),\mathbf{e}_1^*(\lambda)\right]
        + \left|\left\langle \mathbf{v}, \mathbf{e}_2(\lambda) \right\rangle\right|^2 V_\infty(\tau)\left[\mathbf{e}_2^*(\lambda),\mathbf{e}_2^*(\lambda)\right]}{\left|\left\langle \mathbf{v}, \mathbf{e}_1(\lambda) \right\rangle\right|^2 V_\infty(0)\left[\mathbf{e}_1^*(\lambda),\mathbf{e}_1^*(\lambda)\right]
        + \left|\left\langle \mathbf{v}, \mathbf{e}_2(\lambda) \right\rangle\right|^2 V_\infty(0)\left[\mathbf{e}_2^*(\lambda),\mathbf{e}_2^*(\lambda)\right]}
    \end{align*}
    for any $\lambda\leq\lambda^*$. From Theorem \ref{thm:autocov fast system complex} and \eqref{eq:I_ran_out_of_names}, we obtain that
        \begin{align*}
        &\left|AC_\infty(\tau)\left[\mathbf{v}\right]
            - \left( c_1(\lambda^*) e^{-i B(\lambda)\tau}
            + c_2(\lambda^*) e^{i B(\lambda)\tau} \right)\right| \\
        \leq& \left| MC(\tau)[\mathbf{v}] -\left( c_1(\lambda^*) e^{-i B(\lambda)\tau}
            + c_2(\lambda^*) e^{i B(\lambda)\tau} \right) \right|\\
        &+ \cO\left(\underset{s\in\{0,\tau\}}{\max} \left\{ \left| \frac{ \overline{\left\langle \mathbf{v}, \mathbf{e}_1(\lambda) \right\rangle} \left\langle \mathbf{v}, \mathbf{e}_2(\lambda) \right\rangle V_\infty(s)\left[\mathbf{e}_1^*(\lambda),\mathbf{e}_2^*(\lambda)\right]}{\left|\left\langle \mathbf{v}, \mathbf{e}_1(\lambda) \right\rangle\right|^2 V_\infty(s)\left[\mathbf{e}_1^*(\lambda),\mathbf{e}_1^*(\lambda)\right]
        + \left|\left\langle \mathbf{v}, \mathbf{e}_2(\lambda) \right\rangle\right|^2 V_\infty(s)\left[\mathbf{e}_2^*(\lambda),\mathbf{e}_2^*(\lambda)\right]} \right| \right\} \right),
    \end{align*}
    in which we use the equality
    \begin{align*}
        V_\infty(s)\left[\mathbf{e}_1^*(\lambda),\mathbf{e}_2^*(\lambda)\right]
        = \overline{V_\infty(s)\left[\mathbf{e}_2^*(\lambda),\mathbf{e}_1^*(\lambda)\right]}
    \end{align*}
    for any $s\geq 0$ and $\lambda\leq\lambda^*$. Lemma \ref{lem:autocov fast system complex lemma} and the construction of $c_1$ and $c_2$ imply also that
    \begin{align*}
        MC(\tau)[\mathbf{v}]= c_1(\lambda) AC_\infty(\tau)\left[\mathbf{e}_1^*(\lambda)\right] + c_2(\lambda) AC_\infty(\tau)\left[\mathbf{e}_2^*(\lambda)\right].
    \end{align*}
    Then, it follows from Lemma \ref{lem:autocor fast system complex} that
    \begin{align*}
        &\left| MC(\tau)[\mathbf{v}] -\left( c_1(\lambda^*) e^{-i B(\lambda)\tau}
            + c_2(\lambda^*) e^{i B(\lambda)\tau} \right) \right|\\
        \leq& \left| c_1(\lambda) AC_\infty(\tau)\left[\mathbf{e}_1^*(\lambda)\right]
            - c_1(\lambda^*) e^{-i B(\lambda) \tau} \right| \\
        &+ \left| c_2(\lambda) AC_\infty(\tau)\left[\mathbf{e}_2^*(\lambda)\right]
            - c_2(\lambda^*) e^{i B(\lambda) \tau} \right|\\
        =& \mathcal{O}\left(|A(\lambda)|\right) + \mathcal{O}\left( A(\lambda)^2 \right) + \mathcal{O}\left( \left| c_1(\lambda) - c_1(\lambda^*) \right| \right).
    \end{align*}
    From the definition of $c_1(\lambda)$ for any $\lambda\leq\lambda^*$ and from the fact that 
    \begin{align*}
        \left| \left\langle \mathbf{e}_j^*(\lambda), Q \mathbf{e}_j^*(\lambda) \right\rangle - \left\langle \mathbf{e}_j^*(\lambda^*), Q \mathbf{e}_j^*(\lambda^*) \right\rangle \right| = \cO\left(\left|\left| \mathbf{e}_j^*(\lambda) - \mathbf{e}_j^*(\lambda^*) \right|\right|\right)
    \end{align*}
    for any $j\in\{1,2\}$, it follows that
    \begin{align*}
        \left| c_1(\lambda) - c_1(\lambda^*) \right| = \cO\left(\left|\left| \mathbf{e}_1^*(\lambda) - \mathbf{e}_1^*(\lambda^*) \right|\right|\right) + \cO\left(\left|\left| \mathbf{e}_2^*(\lambda) - \mathbf{e}_2^*(\lambda^*) \right|\right|\right).
    \end{align*}
    The proof is implied by Theorem \ref{thm:autocov fast system complex} and Table \ref{tab: tiny}, which state that
    \begin{align*}
        &\left| \frac{ \overline{\left\langle \mathbf{v}, \mathbf{e}_1(\lambda) \right\rangle} \left\langle \mathbf{v}, \mathbf{e}_2(\lambda) \right\rangle V_\infty(s)\left[\mathbf{e}_1^*(\lambda),\mathbf{e}_2^*(\lambda)\right]}{\left|\left\langle \mathbf{v}, \mathbf{e}_1(\lambda) \right\rangle\right|^2 V_\infty(s)\left[\mathbf{e}_1^*(\lambda),\mathbf{e}_1^*(\lambda)\right]
        + \left|\left\langle \mathbf{v}, \mathbf{e}_2(\lambda) \right\rangle\right|^2 V_\infty(s)\left[\mathbf{e}_2^*(\lambda),\mathbf{e}_2^*(\lambda)\right]} \right| = \mathcal{O}\left(|A(\lambda)|\right),
    \end{align*}
    for any fixed $s\geq 0$ and by the fact that $|MC(\tau)[\mathbf{v}]|=\mathcal{O}(1)$.
    \qed
\end{refproof}

\end{document}